\documentclass[amsmath, amssymb, reqno]{amsart}

\usepackage{amsfonts}
\usepackage{amsmath}
\usepackage{amsthm}
\usepackage{amssymb}
\usepackage{graphics}
\usepackage{bm}
\usepackage{latexsym}
\usepackage{setspace}
\usepackage{mathdots}
\usepackage{mathrsfs}

\usepackage{bbold}
\usepackage{tkz-linknodes}

\newcommand{\be}{\begin{equation}}
\newcommand{\ee}{\end{equation}}
\newcommand{\ba}{\begin{eqnarray}}
\newcommand{\ea}{\end{eqnarray}}
\newcommand{\bi}{\begin{itemize}}
\newcommand{\ei}{\end{itemize}}
\newcommand{\bn}{\begin{enumerate}}
\newcommand{\en}{\end{enumerate}}
\newcommand{\bp}{\begin{proof}}
\newcommand{\ep}{\end{proof}}

\newcommand{\wt}{\ensuremath{\widetilde}}

\newcommand{\mr}{\ensuremath{\mathrm}}

\newcommand{\mc}{\ensuremath{\mathcal}}

\renewcommand{\bm}{\ensuremath{\mathbb }}
\newcommand{\mf}{\ensuremath{\mathfrak}}
\newcommand{\ms}{\ensuremath{\mathscr}}
\newcommand{\ov}{\ensuremath{\overline}}
\newcommand{\sm}{\ensuremath{\setminus}}

\newcommand{\intfty}{\ensuremath{\int _{-\infty} ^{\infty}} }

\newcommand{\Om}{\ensuremath{\Omega}}
\newcommand{\Ga}{\ensuremath{\Gamma}}

\newcommand{\La}{\ensuremath{\Lambda }}
\newcommand{\la}{\ensuremath{\lambda }}

\renewcommand{\H}{\ensuremath{\mathcal{H} }}
\newcommand{\J}{\ensuremath{\mathcal{J} }}
\newcommand{\K}{\ensuremath{\mathcal{K} }}

\newcommand{\R}{\ensuremath{\mathbb{R} }}

\newcommand{\Z}{\ensuremath{\mathfrak{Z} }}
\newcommand{\C}{\ensuremath{\mathbb{C} }}
\renewcommand{\L}{\ensuremath{\mathcal{L} }}
\newcommand{\N}{\ensuremath{\mathbb{N} }}
\newcommand{\T}{\ensuremath{\mathbb{T} }}
\newcommand{\D}{\ensuremath{\mathbb{D} }}

\newcommand{\symm}{\ensuremath{\mathscr{S}}}
\newcommand{\isom}{\ensuremath{\mathscr{V}}}

\newcommand{\lessim}{\lesssim}

\newcommand{\ip}[2]{\ensuremath{\left\langle {#1} , {#2} \right\rangle}}

\newcommand{\dom}[1]{\ensuremath{\mathrm{Dom} ({#1}) }}

\newcommand{\ext}[1]{\ensuremath{\mathrm{ext}  ({#1}) }}

\newcommand{\Ext}[1]{\ensuremath{\mathrm{Ext} ({#1}) }}
\newcommand{\Extu}[1]{\ensuremath{\mathrm{Ext} _U ({#1}) }}

\renewcommand{\dim}[1]{\ensuremath{\mathrm{dim} \left( {#1} \right) }}

\newcommand{\ran}[1]{\ensuremath{\mathrm{Ran} \left( {#1} \right) }}
\renewcommand{\ker}[1]{\ensuremath{\mathrm{Ker} ({#1}) }}

\newcommand{\im}[1]{\ensuremath{\mathrm{Im} \left( {#1} \right) }}
\newcommand{\re}[1]{\ensuremath{\mathrm{Re} \left( {#1} \right) }}

\newcommand{\sym}[2]{\ensuremath{\mathscr{S} _{#1} ({#2}) }}
\newcommand{\symn}[1]{\ensuremath{\mathscr{S} _{#1}}}

\numberwithin{equation}{section}

\newtheorem{thm}[subsection]{Theorem}
\newtheorem{claim}[subsection]{Claim}

\newtheorem{lemma}[subsection]{Lemma}
\newtheorem{prop}[subsection]{Proposition}
\newtheorem{cor}[subsection]{Corollary}

\theoremstyle{definition}
\newtheorem{defn}[subsection]{Definition}
\newtheorem{remark}[subsection]{Remark}
\newtheorem{eg}[subsection]{Example}

-.9truein

\advance\voffset by -.1truein 

\begin{document}

\title{Extensions of symmetric operators I: The inner characteristic function case.}

\author{R. T. W. Martin}
\address{Department of Mathematics and Applied Mathematics, University of Cape Town, Cape Town, South Africa}
\email{rtwmartin@gmail.com}

\begin{abstract}

Given a symmetric linear transformation on a Hilbert space, a natural problem to consider is the characterization of its set of symmetric extensions. This problem
is equivalent to the study of the partial isometric extensions of a fixed partial isometry.  We provide a new function theoretic characterization of the set of all
self-adjoint extensions of any symmetric linear transformation $B$ with equal indices and inner Livsic characteristic function $\Theta _B$  by constructing a natural bijection between the set of self-adjoint extensions
and the set of all contractive analytic functions $\Phi$ which are greater or equal to $\Theta _B$. In addition we characterize the set of all
symmetric extensions $B'$ of $B$ which have equal indices in the case where $\Theta _B$ is inner.

\end{abstract}

\maketitle

\section{Introduction}

\label{intro}

The purpose of this paper is to study of the family of all closed symmetric extensions of a given closed simple symmetric linear transformation $B$ with equal deficiency indices $(n,n)$, $1 \leq n < \infty$ defined on a domain in a separable Hilbert space in the case where the Livsic characteristic function of $B$ is an inner function. For $ n \in \N \cup \{ \infty \} $, $\sym{n}{\H}$ will denote the set of all closed simple symmetric linear transformations with indices $(n,n)$ defined in a separable Hilbert space $\H$. More generally $\symn{n}$ will denote the family of all closed simple symmetric linear transformations with indices $(n,n)$ defined in some
separable Hilbert space, and $\ms{S}$ the set of all closed simple symmetric linear transformations with equal indices defined in some separable Hilbert space.

If $A$ is a symmetric linear transformation which extends $B \in \sym{n}{\H}$ and $\dom{A}$ is also contained in $\H$ then we call $A$ a canonical extension of $B$. If, on the other hand $A$ is symmetric in $\K$ where
$\K \supsetneq \H$, then we call $A$ a non-canonical extension of $B$. The set of all canonical extensions of $B$ can be completely characterized by the set of all partial isometries between the deficiency subspaces $\ker{B^* - i}$ and $\ker{B^* +i}$, see for example \cite[Chapter VII]{Glazman}. Our goal is to provide a new characterization the set of extensions, canonical and non-canonical in the special case where the characteristic function $\Theta _B $ is inner. (Recall that in this case $B$ is unitarily equivalent to multiplication by $z$ in a model subspace $K^2 _{\Theta _B} = H^2 (\C _+ ) \ominus \Theta _B H^2 (\C _+) $ of Hardy space \cite{Martin-uni, Martin-dB, AMR}.)

We will begin with the study of the self-adjoint extensions of $B$, denoted $\Ext{B}$, and show that there is a bijective correspondence between $A \in \Ext{B}$ and the set of all contractive analytic (matrix) functions
$\Phi _A$ which obey: $$ \Phi _A \geq \Theta _B, $$ see Theorem \ref{bijection}.  Here, given contractive analytic matrix functions $\Phi , \Theta$ on $\C _+$,
we say that $\Theta \leq \Phi$ provided that $\Theta ^{-1} \Phi$ is contractive and analytic on $\C _+$. This provides an alternative to the classical results of M.G. Krein (see \emph{e.g} \cite[Theorem 6.5]{Krein} for the $(1,1)$ case) which are formulated in terms of generalized resolvents and $R$-functions. Our characterization has the advantage of providing a natural function-theoretic connection between the Livsic characteristic function of $B \in \symm$ and the set of its self-adjoint extensions.

We will also study a natural partial order on $\symm$: we say that $B_1 \lessim B_2$ for $B_1, B_2 \in \symm$ if $B_1 \simeq B_1 ' \subset B_2$, where $\simeq$ denotes unitary equivalence and we use the $\subset$ notation to denote when one linear transformation is an extension of another. In words, $B_1$ is less than or equal to $B_2$ if $B_2$ is an extension of $B_1 '$ where $B_1 '$ is unitarily equivalent to $B_1$.  Application of the Cayley transform, which is a bijection from $\symm$ onto $\isom$, the set of all partial isometries with equal indices, converts this into a partial order on $\isom$. Modulo unitary equivalence, this is the same as the partial order previously defined by Halmos and McLaughlin on partial isometries in \cite{Halmos}. In the case where $\Theta _{B _1}$ is an inner function, we provide necessary and sufficient conditions on $\Theta _{B_2}$ so that $B_1 \lessim B_2$ in Theorem \ref{pochar}.

Many of these results will be achieved using the concept of a generalized model. This is a reproducing kernel Hilbert space theory approach which generalizes the concept of a model for a symmetric operator as defined in \cite{AMR}.

\section{Preliminaries}

\label{prelim}

Recall that a linear transformation $B$ is simple, symmetric and closed with deficiency indices $(n,n)$ if it is defined on a domain $\dom{B}$ contained in a separable Hilbert space $\H$ and has the following properties:
\be \ip{Bx}{y}= \ip{x}{By}, \quad \quad \forall x,y \in \dom{B}, \quad \quad B \ \mbox{is \emph{symmetric}}; \ee
\be \bigcap _{z \in \C \sm \R} \ran{B-z}  = \{ 0 \}, \quad \quad B \  \mbox{is \emph{simple}}; \label{simple} \ee
\be  \{ (x , Bx) | \  x \in \dom{B} \} \ \mbox{is a closed subset of} \ \H \oplus \H, \quad B \ \mbox{is \emph{closed}}; \ee
$$ n_- := \dim{\ran{B-i} ^\perp} = n = \dim{\ran{B+i} ^\perp} =: n_+,  $$
\be \quad \quad \quad B \  \mbox{\emph{has equal deficiency indices}} \ (n_+, n_-). \ee

Condition (\ref{simple}) can be restated equivalently as: $B$ is simple if and only if there is no non-trivial subspace reducing for $B$ such that the restriction of $B$ to the intersection of its domain with this subspace is
self-adjoint.  For many of our results we will need to assume that $n < \infty$ is finite.

A partial isometry $V$ is called simple, or c.n.u. (completely non-unitary) if it has no unitary restriction to a proper (and non-trivial) reducing subspace.  The deficiency indices for $V$ are the pair of non-negative integers $(n_+ , n_-)$ defined by
$$ n_+ := \mr{dim} \left( \ker{V} \right) \quad \mbox{and} \quad n_- := \mr{dim} \left( \ran{V} ^\perp \right),$$
and it is not difficult to see that these are the same as the defect indices of $V$ as defined in \cite{NF}.

There is a bijective correspondence between $\ms{S} _n (\H)$ and $\ms{V} _n (\H)$ which we now describe:
Given a simple symmetric linear transformation $B \in \ms{S} _n (\H)$ and $z \in \C \setminus \R$, let $Q_z$ denote the
projection onto $\ran{B - \ov{z} }$. The Cayley transform $V _B$ of $B$ is the partial isometry
\be V_B := b(B)Q_{i} = (B-i )(B+i ) ^{-1} Q_i  ;  \quad \quad b(z) := \frac{z-i}{z+i},  \ee
where $b(B) = (B-i )(B +i  )^{-1}$ is a well-defined isometry from $Q_i \H = \ran{B+i I}$ onto $Q_{-i} \H = \ran{B-i }$.
Note that $\ker{V} = \ran{B + i }^{\perp}$ and $\ran{V} ^\perp = \ran{B - i }^{\perp}$, and so it follows that the deficiency indices of $V_B$
are the same as those of $B$.

Conversely suppose that $V$ is a simple partial isometry on $\H$ with defect indices $(n_+ , n_- )$. One can construct a
symmetric linear transformation $B_V$ by defining
$$ \dom{B_V} := (1-V) \ker{V} ^\perp, $$ and
$$ B_V f = b^{-1} (V) f =  i (1 + V) (1 - V)^{-1}f,  \quad f \in \dom{B_V}; \quad b^{-1} (z) := i \frac{1+z}{1-z}.$$ Again it is easy to check that $B_V$ and $V$ have the
same deficiency indices. One can further verify that $B_{V_B} = B$ and $V_{B_V} = V$ for any symmetric linear transformation $B$
and partial isometry $V$, respectively.  This shows that the maps $B \mapsto V_B$ and $V \mapsto B_V$ are inverses of each other so that
these maps are bijections between $\ms{S}$ and $\ms{V}$. We will use this bijection between $\isom$ and $\symm$ to formulate problems
in whichever setting is most convenient, and to obtain equivalent results for both classes of linear transformations.

Given $B \in \sym{n}{\H}$, $n < \infty$, one can construct a complete unitary invariant $\Theta _B$, called the Livsic characteristic function as follows: Pick orthonormal bases $\{ u _j \} _{j=1} ^n$ and $\{v _j\} _{j=1} ^n$ for $\ran{B+i} ^\perp$ and $\ran{B-i} ^\perp$ respectively and choose arbitrary (not necessarily orthonormal) bases $\{ w _j (z) \} _{j=1} ^n$ for $\ran{B-\ov{z} }^\perp$. Let
\be B(z) := \left[ \ip{ w_j (z)}{u_k} \right] _{1 \leq j,k \leq n} , \label{B}\ee  and
\be  A(z) := \left[ \ip{w_j (z)}{v_k} \right] _{1 \leq j,k \leq n}. \label{A}\ee  The Livsic characteristic function is then \cite{Livsic2}
\be \Theta _B (z) := b(z) B^{-1} (z) A(z), \label{Livcharfun} \ee  and this can be shown to be a contractive $n\times n$ matrix-valued analytic function on $\C _+$, the upper half-plane. Note that the characteristic function $\Theta _B$ always
vanishes at $z=i$. Different choices of bases in
the definition yield a new characteristic function $\wt{\Theta} _B$ which is related to the first by
$$ \wt{\Theta} _B (z) = R \Theta _B (z) Q, $$ where $R,Q$ are fixed unitary matrices.  Two Livsic characteristic functions $\Theta _1 , \Theta _2$ are said to coincide or to be equivalent if they are related in this way.

For most of this paper we will assume that $\Theta _B$ is an inner function, \emph{i.e.} $\Theta _B$ has non-tangential boundary values on $\R$ almost everywhere with respect to Lebesgue measure, and these non-tangential boundary values are unitary matrix-valued. In this case $B \simeq Z_{\Theta _B}$, where $Z_{\Theta _B} \in \ms{S} \left( K^2 _{\Theta_B} \right)$ is the symmetric operator of multiplication by $z$ on the domain
$$ \dom{Z_{\Theta _B} } = \{ f \in K^2 _{\Theta _B} | \ zf \in K^2 _{\Theta _B} \}, $$ in the model space $K^2 _{\Theta _B} = H^2 \ominus \Theta _B H^2$, and here $H^2 = H^2 (\C _+)$ is the Hardy space of the upper half-plane.

As shown in \cite{AMR}, one can also define the Livsic characteristic function for the case where $n = \infty$, and this new definition coincides with the old one for $n < \infty$.  As first shown by M.S. Livsic, the Livsic
characteristic function is a complete unitary invariant for $\sym{n}{\H}$:

\begin{thm}
Linear transformations $B_1 \in \sym{n}{\H _1}$ and $B_2 \in \sym{n}{\H _2}$ are unitarily equivalent if and only if their characteristic functions $\Theta _1 , \Theta _2 $ are equivalent. \label{Liv}
\end{thm}

One of the results of this paper, Theorem \ref{equivalence} will provide a similar result for all $A \in \Ext{B}$. Given any $A \in \Ext{B}$, we will define a characteristic function $\Phi _A = \Phi [A ; B]$ which has
the property that $\La _A \geq \Theta _B$ where $\La _A$ is a Frostman shift of $\Phi _A$ vanishing at $i$.  Theorem \ref{equivalence} will show that $\Phi _{A_1} = \Phi _{A_2}$ if and only if $A_1 \simeq A_2$ via a unitary which is the identity when restricted to $\H$.

Here is our formal definition of $\Ext{B}$:

\begin{defn}
Given $V \in \isom _n (\H )$, let $\Ext{V}$ denote the set of all unitary operators $U$ such that
\bn
    \item $U$ is an extension of $V = b(B)$, \emph{i.e.} $V \subseteq U$ ($U | _{\ker{V} ^\perp} = V| _{\ker{V} ^\perp}$) and $U$ is unitary in some Hilbert space $\mc{K} \supset \mc{H}$.
    \item $\mc{K}$ is the smallest reducing subspace for $\mr{vN} (U)$, the von Neumann algebra generated by $U$.
\en

Given $B \in \sym{n}{\mc{H}}$, we will define $\Ext{B}$ to be a relabeling of the set $\Ext{b(B)}$. Namely if $U \in \Ext{b(B)}$, and $1 \notin \sigma _p (U)$, the set of eigenvalues of $U$,
then we define $A $ to be the self-adjoint operator $b^{-1} (U)$. If however $U \in \Ext{b(B)}$ and $1 \in \sigma _p (U)$, then we formally define $A$ by $b^{-1} (U)$. In this
case $A$ is not a well defined linear transformation, it is just a renaming of $U \in \Ext{b(B)}$ with the understanding that $A_1 = A_2$ for $A_1 = b^{-1} (U_1 ), A_2 = b^{-1} (U_2)$
and $U_1 , U_2 \in \Ext{b(B)}$ if and only if $U_1 = U_2$. $\Ext{B}$ is then defined to be the set of all such $A$. In this way there is a bijection between $\Ext{b(B)}$ and $\Ext{B}$.
\label{extend}
\end{defn}

Recall that the subset notation $B \subset A$ means that $A$ is an extension of $B$, \emph{i.e.} $\dom{B} \subset \dom{A}$ and $A|_{\dom{B} } =B$. The subset notation
$V \subseteq U$ for partial isometries $V, U$ means that $ U | _{\ker{V} ^\perp} = V | _{\ker{V} ^\perp}$. For simple symmetric linear transformations $B_1, B_2$ we have that
$B_1 \subset B_2$ if and only if $b(B_1 ) \subseteq b(B _2 )$.

\begin{remark}
    If $B$ is densely defined then every unitary extension $U$ of $b(B)$ does not have $1$ as an eigenvalue \cite[Lemma 6.1.3]{Martin-uni},\cite{Livsic}, so that every element of
$\Ext{B}$ is a densely defined self-adjoint operator. Note that if $A \in \Ext{B}$ and $A = b^{-1} (U)$ for some $U \in \Ext{b(B)}$ such that $1 \notin \sigma _p (U)$, then the two conditions of the above definition are equivalent to
\bn
    \item $A$ is an extension of $B$, \emph{i.e.} $B \subset A$ and $A$ is self-adjoint in some Hilbert space $\K \supset \H$.
    \item $\mc{K}$ is the smallet reducing subspace for $\mr{vN} (A)$, the von Neumann algebra generated by $b(A)$.
\en
However if $B$ is not densely defined, then one
can find canonical unitary extensions $U$ of $V = b(B)$ which have $1$ as an eigenvalue \cite[Lemma 6.1.3]{Martin-uni},\cite{Livsic2}. In this exceptional case where $U$ is a unitary
extension of $b(B)$ and $1 \in \sigma _p (U)$, then we will always work with the unitary extension $U$ associated with $A = b^{-1} (U)$.
If $1 \in \sigma _p (U)$, one could define $A := b^{-1} (U) P_U (\T \sm \{ 1 \} ) = b^{-1} (U) \chi _{\T \sm \{1 \} } (U)$, where $\chi _\Om$ is the characteristic function of $\Om$, $\T$ is the unit circle and $P_U (\T \sm \{ 1 \} ) =\chi _{\T \sm \{1 \} } (U)$ projects onto the orthogonal complement of the eigenspace to eigenvalue $1$ of $U$. However we will have no need for this construction, and in this exceptional case where
$1 \in \sigma _p (U)$ for $U \in \Ext{b(B)}$ we will simply work with $U \in \Ext{b(B)}$ instead of its inverse Cayley transform $A = b^{-1} (U)$ in $\Ext{B}$. In this paper we are really
studying $\Ext{b(B)}$, but given $U \in \Ext{b(B)}$ we prefer to work with $A = b^{-1} (U) \in \Ext{B}$ whenever this is well-defined.
\end{remark}

It will also be convenient to define $\Extu{B}$ to be the set of all self-adjoint linear transformations $A$ on $\K$ for which $A \in \Ext{UBU^*}$ for some isometry $U :\H \rightarrow \K$.

The set $\Ext{B}$ is called the set of extensions of $B$. In the case where $\mc{K} = \mc{H}$, we say that $A$ is a canonical self-adjoint extension of $B$. Recall that the canonical self-adjoint extensions $A$ of $B$ can
all be obtained by first computing the Cayley transform $V := b(B)$, extending this by a $rank-n$ isometry $U: \dom{V} ^\perp \rightarrow \ran{V} ^\perp$ to obtain a unitary extension $V_U$ of $V$, and then taking the
inverse Cayley transform to obtain a self-adjoint linear transformation $A := b^{-1} (V_U)$.

\section{Linear relations}

\label{symrel}

In the case where $B$ is not densely defined, its adjoint $B^*$ is not a linear operator. Instead $B^*$ can be realized as a linear relation, and we will dicuss the basic facts about
linear relations that will be needed in this section. The material from this section is taken primarily from \cite{Silva-entire} and \cite[Section 1.1]{Habock}. A \emph{linear relation} $L$ is defined
to be a subspace of $\H \oplus \H$. Note that $L = \mf{G} (T)$ is the graph of some closed linear operator $T$ provided that $L$ is closed and $(0, f) \in L$ implies that $f=0$.

Given a linear relation $L$, one defines the adjoint linear relation $L^*$ by
\be L^* := \left\{ (g_1 ,g_2 ) | \  \ip{f_1 }{g_2} = \ip{f_2 }{g_1} \quad  \forall (f_1, f_2) \in L  \right\}. \label{adj} \ee
$L$ is called symmetric if $L \subset L^*$ and $L$ is self-adjoint if $L = L^*$. Clearly if $B$ is a closed symmetric linear operator with adjoint $B^*$ then the graph, $\mf{G} (B)$ of $B$ is a closed symmetric
linear relation, and the graph, $\mf{G} (B^* )$ of $B^*$ is the adjoint relation to $\mf{G} (B)$.

In this paper we will be considering closed symmetric linear transformations $B$ with deficiency indices $(n,n)$, which are not necessarily densely defined.  If this is the case then this means that $B$ does not have a uniquely defined adjoint operator, and it will be convenient to identify $B$ with its graph $\mf{G} (B)$:
$$ \mf{G} (B) := \{ (f, Bf) | \ \ f \in \dom{B} \},$$ in which case $$ \mf{G} (B ) ^* = \{ (g_1 ,g_2 ) | \ \ \ip{f}{g_2} = \ip{Bf}{g_1} \quad \forall f \in \dom{B} \} $$ is a closed linear relation but not the graph of a linear operator. Indeed, observe that if $g \perp \dom{B}$ then by equation (\ref{adj}),
$(0,g) \in \mf{G} (B) ^*$ since  $$ \ip{f}{g} = \ip{Bf}{0},$$ for every $(f, Bf) \in \mf{G} (B)$. For convenience we will simply write $B^* $ for $\mf{G} (B) ^*$ in the case where $B$ is not densely defined. Note that
$$ B^* (0) := \{ f \in \H | \ (0, f) \in \mf{G} (B) ^* = B^* \}  = \ov{\dom{B} } ^\perp .  $$ One can show that if $B$ has deficiency indices $(n,n)$ that the co-dimension of $\dom{B}$ is at most $n$:

\begin{lemma}
If $B \in \sym{n}{\H}$, the orthogonal complement of $\dom{B}$ is at most $n-$dimensional.
\end{lemma}

\begin{proof}
    If $V = b(B)$ then $\ker{V}$ is $n-$dimensional, and $\dom{B} = (1-V) \ker{V} ^\perp$. If $f \perp \dom{B}$, then
$$ 0 = \ip{f}{(1-V) \ker{V} ^\perp} = \ip{(1-V^*) f}{\ker{V} ^\perp}, $$ and so $(1-V^*) f \in \ker{V}$ which is $n-$dimensional.
Now $(1 -V^* )f \neq 0$ as then $f$ would be an eigenfunction to eigenvalue $1$ and $V$ would not be simple. It follows that the dimension
of $\dom{B} ^\perp$ is at most $n$ as otherwise we could find a $g \in \dom{B} ^\perp$ such that $ (1-V^* )g =0$.
\end{proof}

For $z \in \C$ define $$ (B^* -z) := \{ (f , g -zf) | \ \ (f,g) \in B^* \},$$ and $$ \ker{B^* -z} := \{ f \in \H | \ (f , 0) \in (B^* -z) \}.$$ Then, as in the case of densely defined $B$, it follows that
$$ \ker{B^* -z} = \ran{B-\ov{z} } ^\perp, $$ so that $$ \H = \ran{B-\ov{z}} \oplus \ker{B^* -z},$$ for any $z \in \C \sm \R$.

If $B$ is a symmetric linear transformation then one can show, whether or not $B$ is densely defined, that
$$ \dim{ \ker{B^* -z} } \quad \mbox{is constant for} \ z \in \C _\pm,$$ so that one can define $n_\pm =  \dim{ \ker{B^* -z} }$ for $z \in \C _\pm$.   For lack of a reference, here is an elementary proof of this fact.

\begin{prop}
    Let $B$ be a symmetric linear transformation in a separable Hilbert space $\H$.  Then $\dim{ \ker{B^* -z}}$ is constant in $\C _+$ and in $\C _-$.
\end{prop}

\begin{proof}
    Given $w \in \C \sm \R$ let $P_w := $ projection onto $\ker{B^* -w } = \ran{B -\ov{w} } ^\perp$, and let $Q_w :=$ projection onto $\ran{B-w}$ so that
$Q_w = 1 - P _{\ov{w}}$.

Now fix $w \in \C \sm \R$.  Choose any $f \in Q_w \H$ of unit norm, $\| f \| =1$. Since $f \in \ran{B-w}$, we have that $f = (B-w) g$ for some $g \in \dom{B}$.
Now $B-w$ is bounded below, an easy calculation shows that for any $g \in \dom{B}$:
$$ \| (B-w) g \| ^2 = \| (B-\re{w} ) g \| ^2 + | \im{w} | ^2 \|g \| ^2 \geq |\im{w} | ^2 \| g \| ^2.$$
Hence $$ \| g \| \leq \frac{ \| (B-w) g \| }{|\im{w}|} =  \frac{ \| f \| }{ | \im{w} | } = \frac{1}{| \im{w} | }.$$

Now choose $z$ in the same half-plane as $w$ and consider:
\ba  Q_z Q_w f & = &  Q_z f = Q_z (B-w)g = Q_z \left( (B-z) g + (z-w) g \right) \nonumber \\
& = & (B-z) g + (z-w) Q_z g. \nonumber \ea
It follows that
\ba  (Q_w - Q_z Q_w ) f & = & (B-w) g - (B-z)g  - (z-w) Q_z g = (z-w) (1 - Q_z) g \nonumber \\
&  = & (z-w) P_{\ov{z} } g. \nonumber \ea
This implies that
$$\| (Q_w -Q_z Q_w) f \| \leq |z-w | \| g \| \leq \frac{|z-w| }{| \im{w} | }.$$
Since $f$ was an arbitrary norm one vector in $Q_w \H$ we conclude that
$$ \| Q_w - Q_z Q_w \| \leq  \frac{|z-w| }{| \im{w} | }.$$
Taking adjoints it follows that we also have
$$ \| Q_w - Q_w Q_z \| \leq  \frac{|z-w| }{| \im{w} | }.$$

Now
\ba  \| Q_w - Q_z \| & = & \| Q_w - Q_w Q_z + Q_w Q_z - Q_z \| \nonumber \\
& \leq & \| Q_w - Q_w Q_z \| + \| Q_z - Q_w Q_z \| \nonumber \\
& \leq & \frac{|z-w| }{| \im{w} |} + \frac{|z-w| }{| \im{z} |}. \nonumber \ea
For fixed $w \in \C _+$ or $\C _-$, this is less than one for all $z$ in a small enough neighbourhood of $w$.

It follows that for $z$ close enough to $w$ we have
$$ \| P_{\ov{w}} - P _{\ov{z}} \| = \| (1-Q_w) - (1- Q_z ) \| = \| Q_w - Q_z \| < 1, $$
so that by \cite[Section 34]{Glazman} $P_{\ov{z} } \H $ and $P_{\ov{w}} \H$ have the same dimension.
It follows that the dimension of $P_z \H = \ker{B^* -z} = \ran{B-\ov{z} } ^\perp$ is constant for $z \in \C _+$, and for $z \in \C _-$.
\end{proof}

\section{Herglotz Spaces}

\label{Herglotz}

In this section we will show that any $B \in \ms{S} _n$ is unitarily equivalent to the operator of multiplication by $z$ in a certain space of analytic functions called a Herglotz space. Assume that $n< \infty$.

\subsection{Herglotz Functions}

It will be convenient to begin with a brief review of the Nevanlinna-Herglotz representation theory of Herglotz functions on both the unit disk $\D$ and the upper half-plane $\C _+$. Let $g$ be a $\C ^{n\times n}$-valued
Herlglotz function on $\D$, \emph{i.e.} an analytic function with non-negative real part. Here $\C ^{n\times n}$ is our notation for the $n\times n$ matrices over $\C$. Then by the Herglotz representation theorem there is a unique positive Borel $\C ^{n\times n}$-valued measure on the unit circle $\T$
such that
$$ \re{g(z)} = \int _\T \re{\frac{\alpha +z}{\alpha-z}} \sigma (d\alpha).$$ The measure $\sigma$ determines the Herglotz function $g$ up to an imaginary constant so that
$$ g(z) = i b + \int _\T \frac{\alpha +z}{\alpha-z} \sigma (d\alpha).$$ We will always impose the normalization condition that $b=0$ in this paper. Observe that this means that $\sigma$ is a probability
measure, \emph{i.e.} $\sigma $ is unital, $\sigma (\T ) = \bm{1}$, if and only if $g(0) =0$. We will also extend $g$ to a function
on $\C \sm \T$ using the convention that
$$ g(1/ \ov{z} ) ^* = - g (z). $$

Now let $G := g \circ b$ be the corresponding matrix-valued Herglotz function on $\C _+$ ($G$ has non-negative real part in $\C _+$). Setting $w := b^{-1} (z)$ and $t = b^{-1} (\alpha )$, we obtain that
$$ G (w) = -i\sigma(\{ 1 \} ) w + \intfty \frac{wt+1}{i(t-w)}(\sigma \circ b ) (dt).$$ The convention that $g (1 / \ov{z} ) ^* = - g(z)$ implies that $G (\ov{w} )^* = - G (w)$ and this extends $G$ to a function
on $\C \sm \R$. Again, we have that $\sigma $ is unital if and only if $g(0 ) = \bm{1}$ which happens if and only if $G (i ) = \bm{1}$. 

Now the Herglotz theorem on the upper half-plane states that
$$ \re{G (w)} = cy + \intfty P_w (t) \Sigma (dt) ,$$ for unique Borel measure $\Sigma$ obeying $$ \intfty \frac{\Sigma (dt)}{1+t^2} < \infty,$$ and positive constant matrix $c\geq 0$  where $y = \im{w} $ and
$$ P_w (t) = \re{ \frac{1}{i\pi} \frac{1}{t-w} }.$$ It will be convenient to determine the relationship between the Herglotz measure $\sigma$ of $g$ and $\Sigma$ of $G := g \circ b$. As above we let
$$ z(w) = \frac{w-i}{w+i} = b(w) \quad \mbox{and} \quad w(z) = i \frac{1+z}{1-z} = b^{-1} (z).$$
The function $g := G \circ b^{-1}$ obeys
$$ \re{g (z)} = \int _\T p _z (\alpha ) \sigma (d\alpha ) ,$$ where
$$ p _z (\alpha ) = \re{ \frac{\alpha +z }{\alpha -z}},$$ is the Poisson kernel on the disk.
We can write
$$ \re{g (z)} = p _z (1) \sigma ( \{ 1 \} ) + \int _{\T \sm \{ 1 \} } p_z (\alpha ) \sigma (d\alpha).$$

Now for $\alpha \in \T \sm \{ 1 \}$ we can let $\alpha = z (t) $ for $t \in \R$ to write
$$ \re{G (w)} = \re{g (z (w)) } = p_{z(w)} (1) \sigma ( \{ 1 \} ) + \intfty p_{z (w)} (z (t) ) \wt{\sigma} (dt), $$
where $\wt{\sigma}$ is the measure on $\R$ defined by $\wt{\sigma} (\Om) := \sigma (z (\Om ) ) = (\sigma \circ b) (\Om ) $, so that $z (\Om ) = b(\Om ) \in \T \sm \{ 1 \}$.

A bit of algebra shows that
$$ p _z (1 ) = \frac{1-|z| ^2}{|1-z | ^2 }, $$ and that if $w = x +iy \in \C _+$, then
$$ p _{z(w)} (1) = y.$$

Some more algebra shows that
$$ P_w (t) = \frac{1}{2\pi i} \frac{w - \ov{w}}{|t-w | ^2 }, $$ while
$$ p_{z(w)} (z(t)) = \pi (1+t^2 ) P_w (t).$$

We conclude that
$$ \re{ G (w) } = \re{g (z(w))} = y \sigma ( \{ 1 \} ) + \intfty P_w (t) \pi (1+t^2) \wt{\sigma} (dt). $$
Finally this shows how the measures $\wt{\sigma}$ and $\Sigma$ are related:
\be \Sigma (\Om ) = \int _\Om \pi (1+t^2) (\sigma \circ b) (dt). \label{measrel} \ee

Now let $\Theta$ be an arbitrary contractive $n\times n$ matrix-valued  analytic function on $\C _+$. Then  $$ G_\Theta := \frac{1+\Theta}{1-\Theta}, $$
is a Herglotz function on $\C _+$.

There is a bijective correspondence between $\C ^{n\times n}$-valued Herglotz functions $G$ on $\C \sm \R$ and $\C ^{n\times n}$-valued contractive analytic functions $\Theta $ on $\C _+$
defined by
$$ \Theta \mapsto G _\Theta := \frac{1 + \Theta}{1 - \Theta}  \quad \mbox{and} \quad G \mapsto \Theta _G := \frac{G-1}{G+1}. $$ 

The Nevanlinna-Herglotz representation theory can also be used to define a bijective correspondence between $\C^{n\times n}$-valued Herglotz functions on $\C ^+$ and a large class of $\C ^{n\times n}$-positive matrix-valued measures on $\R$. Namely if $g$ is a Herglotz function on the unit disk which obeys the normalization condition of the previous section (no non-zero constant imaginary part), then as discussed above it uniquely determined
by a regular, positive $\C ^{n\times n}$-valued Borel measure on the unit circle $\T$ by the formula:
\be g(z) = \int _\T \frac{\alpha +z}{\alpha-z} \sigma (d\alpha). \ee It follows that the Herglotz function $G := g \circ b$ on $\C _+$ is uniquely determined by the Herglotz measure $\Sigma$ and the value of $\sigma( \{ 1 \} )$
by the formula
\ba G (z) & = & -i\sigma(\{ 1 \} ) z + \intfty \frac{zt+1}{i(t-z)}(\sigma \circ b ) (dt)  \nonumber \\
& = & -i\sigma(\{ 1 \} ) z + \frac{1}{i\pi} \intfty \frac{zt+1}{(t-z)} \frac{1}{1+t^2} \Sigma (dt). \label{upglotz} \ea
Conversely given any non-negative matrix $P \in \C ^{n\times n}$ and positive $\C ^{n\times n}$ matrix-valued Borel measure on $\R$ that obeys the condition:
\be \left( \intfty \frac{1}{1+t ^2 } \Sigma (dt)  \vec{v} , \vec{w} \right) _{\C ^n} < \infty, \label{meascon} \ee for any $\vec{v}, \vec{w} \in \C ^n$, there is a unique Herglotz function $G$ on $\C _+$ that obeys equation
(\ref{upglotz}), or equivalently obeys:
$$ \re{G (z)} = P \im{z}  + \intfty \re{\frac{1}{i\pi} \frac{1}{t-z}} \Sigma (dt ). $$ It follows that there is a bijective correspondence between Herglotz functions $G$ on $\C _+$ and such pairs $(P, \Sigma )$, where
$P \in \C ^{n\times n}$ is positive and $\Sigma$ is a positive $\C ^{n\times n}$ valued measure obeying the condition (\ref{meascon}). This in turn implies there is a bijective correspondence between contractive analytic functions $\Theta$ on $\C _+$ and such pairs $(P , \Sigma )$. Given $\Theta$ we will call the corresponding $\Sigma$ the Herglotz measure of $\Theta$ and we will usually denote this by $\Sigma _\Theta$. Similarly $\sigma _\theta$ will denote the Herglotz measure of $\theta := \Theta \circ b^{-1}$. Note that since we assume any Herglotz function $g_\theta$ obeys our normalization condition (no non-zero imaginary constant part), it follows that
$\sigma _\theta$ is unital if and only if $g_\theta (0 ) = \bm{1} = G_\Theta (\bm{1} )$ which happens if and only if $\theta (0 ) = 0 = \Theta (i)$. 

\subsection{Herglotz spaces}

Let $\Theta$ be a $\C ^{n\times n} -$valued contractive analytic function on $\C _+$. The Herglotz space, $\mc{L} (\Theta)$ is the abstract reproducing kernel space of analytic $\C ^n$-valued functions on $\C \setminus \R$ with reproducing kernel $$ K _w ^\Theta (z) :=  \frac{i}{\pi} \frac{G_\Theta (z) + G_\Theta (w) ^* }{ z-\ov{w}}.$$

Namely given any $\vec{v} \in \C ^n$ and $f \in \mc{L} (\Theta )$ and $w \in \C \setminus \R$, we have that $K_w \vec{v} \in \mc{L} (\Theta)$ where
$K_w \vec{v} (z) := K_w (z) \vec{v} $ and,
$$ \left( f(z) , \vec{v} \right) _{\C ^n} = \ip{f}{K_z ^\Theta \vec{v} } _\Theta. $$

As shown in \cite{AMR}, if $\Theta$ is a Livsic characteristic function so that $\Theta (i) =0$, and the symmetric linear transformation $B$ with characteristic function $\Theta$ is densely defined then one can define a closed simple symmetric linear operator $\mf{Z} _\Theta \in \ms{S} _n (\mc{L} (\Theta ) )$ with domain $$\dom{\mf{Z} _\Theta} = \{ f \in \mc{L} (\Theta ) | \ zf \in \mc{L} (\Theta ) \}, $$ by
$$ (\mf{Z} _\Theta f ) (z) := z f(z); \quad \quad f \in \dom{\mf{Z} _\Theta}, $$ see \cite[Theorem 6.3]{AMR}.  Since we do not assume that all of our symmetric linear transformations are densely defined, we will need to extend this slightly:

\begin{lemma}
    Let $\Theta$ be a contractive analytic $\C ^{n\times n}$-valued function on $\C _+$. The linear transformation $\mf{Z} _\Theta$ defined on the domain $$ \dom{\mf{Z} _\Theta} := \{ F \in \L (\Theta ) | \ z F(z) \in \L (\Theta ) \},$$ by $$ ( \mf{Z} _\Theta F ) (z) = z F(z), \quad F \in \dom{\mf{Z} _\Theta }$$  belongs to $\sym{n}{\L (\Theta )}$. \label{Hergz}
\end{lemma}

The proof of this lemma follows from the vector-valued version of \cite[Theorem 5]{dB}, see also \cite{dB-Herglotz}.  In particular we use the identity
$$ (\ov{w} -w ) \ip{\frac{F -F(w)}{z-w}}{\frac{G -G(w)}{z-w}}_\Theta = \ip{F}{\frac{G - G(w)}{z-w} } _\Theta - \ip{ \frac{F - F(w) }{z-w}}{G} _\Theta, $$ valid
for all $F, G \in \L (\Theta )$ proven in \cite[Theorem 5]{dB} for the case $n=1$, and easily verified to also hold for the vector-valued case.

\begin{proof}
    Let $S_{\pm i} := \{ F \in \L (\Theta ) | \ F(\pm i ) = 0 \}$. By de Branges' results on Herglotz spaces, if $F \in S _{-i}$ then
$$ (V F ) (z) := \frac{z-i}{z+i} F(z) = b(z) F(z) \in \L (\Theta ), $$  so that the linear transformation $V$ which acts as multiplication by $b(z)$ obeys $V : S_{-i} \rightarrow S_i$.
We can show that $V$ is in fact an isometry: if $F \in S _{-i}$ then
 \ba \ip{VF}{VF} _\Theta & =& \ip{ F  - \frac{2i}{z+i} F }{F - \frac{2i}{z+i} F } _\Theta  \nonumber \\
& =& \ip{F}{F} _\Theta  -2i \left(  \ip{ \frac{1}{z+i} F}{F} _\Theta  - \ip{F}{  \frac{1}{z+i} F} _\Theta \right) + \ip{ \frac{2i}{z+i} F}{ \frac{2i}{z+i} F} _\Theta  \nonumber \\
&= & \ip{F}{F} _\Theta, \nonumber \ea
using the identity stated before the proof.

It is not hard to verify that $V$ is closed, and so $\mf{Z} _\Theta := b^{-1} (V)$ is a well-defined closed symmetric linear transformation. The symmetric linear transformation
$\mf{Z} _\Theta $ has indices $(n,n)$ since
$$ \ker{\mf{Z} _\Theta ^* + i } = \ker{V} = \bigvee K _{-i} ^\Theta \C ^n, $$ and
$$ \ker{\mf{Z} _\Theta ^* -i} = \ran{V} ^\perp = \bigvee K _i ^\Theta \C ^n.$$ Similarly,
$$ \ker{\mf{Z} _\Theta ^* -z } = \bigvee K_{\ov{z} } ^\Theta \C ^n, $$ so that
$$ \L (\Theta ) = \bigvee _{z \in \C \sm \R } \ker{\mf{Z} _\Theta ^* -z }, $$ proving that $\mf{Z} _\Theta$ is simple.
It remains to check that the domain of $\mf{Z} _\Theta$ is equal to $$\mf{D} _\Theta := \{ F \in \L (\Theta ) | \ z F(z) \in \L (\Theta ) \}.$$ Clearly
$\dom{\mf{Z} _\Theta } \subset \mf{D} _\Theta $, and conversely if $F \in \mf{D} _\Theta$ then $G(z) = (z+i ) F(z) \in S _{-i} = \ker{V} ^\perp$, and so
by definition $(1-V) G \in \dom{ \mf{Z} _\Theta}$, and
$$ (1-V) G(z) = (z+i)F(z) - (z-i) F(z) = 2i F(z).$$ This proves that $F \in \dom{\mf{Z} _\Theta }$ so that $\mf{D} _\Theta = \dom{\mf{Z} _\Theta }$.

\end{proof}

\begin{lemma}
\label{FS}
Let $\Theta$ be a contractive analytic function as above. The Livsic characteristic function of $\mf{Z} _\Theta$ is a Frostman shift of $\Theta$:
$$ \Theta _{\mf{Z} _\Theta} = (1 -\Theta (i) ^*) (1 - \Theta \Theta (i) ^* ) ^{-1} (\Theta - \Theta (i) ) (1 - \Theta (i) ) ^{-1}. $$
\end{lemma}

\begin{proof}
   This is a straightforward calculation using the definition of the characteristic function (equations (\ref{B}), (\ref{A}) and (\ref{Livcharfun})) and the reproducing kernel $$ K_w (z) = \frac{i}{\pi} \frac{ G_\Theta (z) + G_\Theta (w) ^* }{z-\ov{w}},$$ for $\mc{L} (\Theta )$.
Let $\{ e_j \}$ be the standard orthonormal basis of $\C ^n$. We can choose
$$u_j = K_{-i} K_{-i} (-i ) ^{-1/2} e_j, \quad v_j = K_i K_i (i) ^{-1/2} e_j \quad \mbox{and} \quad w_j (z) := K_{\ov{z}} e_j.$$
With this choice of bases, one obtains
$$ A(z) = [ \ip{ K_{\ov{z}} e_j}{ K_i K_i (i) ^{-1/2} e_k } ] = K_i (i) ^{-1/2} K_{\ov{z}} (i) \quad \mbox{and} \quad B(z) = K_{-i} (-i) ^{-1/2} K_{\ov{z} } (-i ).$$  Recall here that $$\Theta _{\mf{Z} _\Theta } (z) = b(z) B(z) ^{-1} A(z).$$
Now observe that $$ K_i (i) = \frac{i}{\pi} \frac{ G _\Theta (i) + G_\Theta (i) ^* }{2i}. $$ Using that $G_\Theta (\ov{z} ) ^* = -G_\Theta (z)$ for the Herglotz function $G_\Theta$,
we also obtain that $$ K_{-i} (-i) = \frac{i}{\pi} \frac{ G_\Theta (-i) + G_\Theta (-i) ^* }{-2i} = \frac{i}{\pi} \frac{ -G _\Theta (i) ^* - G_\Theta (i)  }{-2i} = K_i (i).$$
It follows that
$$ \Theta (z) := \Theta _{\mf{Z} _\Theta } (z) = b(z) B(z) ^{-1} A(z) = b(z) K_{\ov{z}} (-i) ^{-1} K_{\ov{z}} (i).$$

Substituting in our expression for the reproducing kernel $K_w (z)$ yields
\ba \Theta (z) & = & b(z) \left( \frac{i}{\pi} \frac{G (-i) + G (\ov{z} ) ^*}{ -i -z} \right) ^{-1} \left( \frac{i}{\pi} \frac{G (i) + G (\ov{z} ) ^* }{i-z} \right) \nonumber \\
& =& \left( G (-i) + G (\ov{z} ) ^* \right) ^{-1} \left( G (i) + G (\ov{z} ) ^* \right) \nonumber \\
& =& \left( - G (i)^* - G (z)  \right) ^{-1} \left( G (i) - G(z) \right) \nonumber \\
& =& \left( G (i)^* + G (z)  \right) ^{-1} \left( G (i) - G(z) \right). \quad \quad \mbox{(ignore the factor of $-1$)} \nonumber \ea
We can ignore the factor of $-1$ since $\Theta (z)$ is defined only up to conjugation by fixed unitaries.

Now straightforward algebra shows that
$$ G(z) - G(i) = 2 (1 -\Theta (z) ) ^{-1} (\Theta (z) - \Theta (i) (1 -\Theta (i) ) ^{-1}, $$ while
$$ G(i) ^* + G(z) = 2 (1 - \Theta (z) ) ^{-1} (1 -\Theta (z) \Theta (i ) ^* ) (1 - \Theta (i) ^* ) ^{-1}.$$ Putting these two formulas together yields the Frostman shift formula.
\end{proof}

In particular if $\Theta (i) =0$ then $\Theta $ is equal to the Livsic characteristic function of $\mf{Z} _\Theta$, and Theorem \ref{Liv} allows us to conclude:

\begin{cor}
    If $B \in \ms{S}$ has characteristic function $\Theta$ then $B \cong \mf{Z} _\Theta$. \label{lessthan}
\end{cor}

The following example of symmetric extensions of a symmetric operator $B$ with $\Theta _B$ inner will be important:

\begin{eg}
\label{motive}

Let $\Theta , \Phi $ be $\C ^{n\times n}$-valued  inner functions on $\C _+$ such that $\Theta \leq \Phi$. In this case  $\Theta ^{-1} \Phi$ is also an inner function.

Given any inner function $\Theta$ one can define a symmetric linear transformation $Z_\Theta$ acting in $K^2 _\Theta$ by:
$$ \dom{Z_\Theta } := \{ f \in K^2 _\Theta | \ z f(z) \in K^2 _\Theta \}, $$ and
$$ Z_\Theta f (z) := z f(z), \quad f \in \dom{Z _\Theta }, $$
see for example \cite{Martin-dB,AMR}.  It is straightforward to show that the characteristic function of $Z_\Theta$ is the Frostman shift of
$\Theta$ as above so that by Livsic's theorem $Z_\Theta \simeq \mf{Z} _\Theta$.

It follows that since $K^2 _\Theta \subset K^2 _\Phi$ that $\dom{Z _\Theta } \subset \dom{Z _\Phi }$ and that $Z_\Theta \subset Z _\Phi$ so that $Z_\Theta \lessim Z_\Phi$.
Moreover given any $A \in \Ext{Z _\Phi }$, then the restriction $A'$ of $A$ to its smallest invariant subspace containing $K^2 _\Theta$ belongs to $\Ext{Z_\Theta}$.

This can be generalized further: Suppose that $\Phi$ is an arbitrary contractive analytic function such that $\Phi \geq \Theta$ where $\Theta$ is inner. Then by \cite[II-6]{Sarason-dB}, $K^2 _\Theta$ is contained isometrically
in the deBranges-Rovnyak space $K^2 _\Phi$, $K^2 _\Theta \subset K^2 _\Phi$. Moreover \cite[Theorem 7.1]{AMR} shows that multiplication by $V (z) := \frac{2}{1-\Phi (z)}$ is an isometry from $K^2 _\Phi$ into $\mc{L} (\Phi )$. Hence $V : K^2 _\Theta \rightarrow \mc{L} (\Phi )$, the operator of multiplication by $V(z)$ is an isometry of $K^2 _\Theta $ into $\mc{L} (\Phi )$, and by the definition of $\dom{Z_\Theta}$, and the definition of $\dom{\mf{Z} _\Phi } $ in Lemma \ref{Hergz}, it follows that $V \dom{Z _\Theta } \subset \dom{\mf{Z} _\Phi } $ and that $V Z_\Theta V ^* \subset \mf{Z} _\Phi $ so that
$Z_\Theta \lessim \mf{Z} _\Phi$. Since $\mf{Z} _\Theta \cong Z_\Theta$, this also shows that $\mf{Z} _\Theta \lessim \mf{Z} _\Phi$ whenever $\Theta$ is inner, $\Phi$ is contractive and $\Theta \leq \Phi$.
Again the restriction of any $A \in \Ext{\mf{Z} _\Phi}$ to its smallest invariant subspace containing $V K^2 _\Theta$ belongs to $\Extu{Z_\Theta}$. Here recall that given $B \in \symm$, $\Extu{B}$ is the set of all
self-adjoint linear transformations $A$ such that $A \in \Ext{UBU^*}$ for some isometry $U:\H \rightarrow \K$.

We can also construct examples of symmetric $B_1 \in \sym{n}{\H _1 }$ and $B_2 \in \sym{m}{\H _2}$ such that $B_1 \lessim B_2$ where $n \neq m$:
Suppose that $\Phi := \Theta \Gamma$ where $\Phi, \Theta, \Gamma$ are all scalar-valued inner functions on $\C _+$.  Let
$$ \Lambda := \left( \begin{array}{cc} \Theta & 0 \\ 0 & \Gamma \end{array} \right). $$ Then $\Lambda$ is a $2\times 2$ matrix-valued inner function,
and note that $Z _\Lambda$ has indices $(2,2)$, and that there is a natural unitary map $W$ from $K^2 _\Lambda = K^2 _\Theta \oplus K^2 _\Gamma$ onto
$K^2 _\Phi = K^2 _\Theta \oplus \Theta K^2 _\Gamma$. Namely $$ W (f \oplus g ) := f + \Theta g, $$ so that if we view elements of $K^2 _\Lambda $ as column vectors then $W$ acts as multiplication by the $1 \times 2$
matrix function $$ W (z) = (1 , \Theta (z) ).$$ It follows that $Z_\Lambda \lessim Z_\Phi$, where $Z_\Lambda$ has indices $(2,2)$ and
$Z _\Phi$ has indices $(1,1)$.
\end{eg}

\begin{thm}
If $B_1, B_2 \in \ms{S}$ with characteristic functions $\Theta _1, \Theta _2$, the characteristic function $\Theta _1$ is inner and $\Theta _1 \leq \Theta _2$ then $B_1 \lessim B_2$. \label{conless}
\end{thm}

\begin{proof}
By Corollary \ref{lessthan}, $B_j \simeq \mf{Z} _{\Theta _j}$. As discussed in the above example if $\Theta _1$ is inner and $\Theta _1 \leq \Theta _2$ then $\mf{Z} _{\Theta _1} \lessim \mf{Z} _{\Theta _2}$
so that $B_1 \lessim B_2$.
\end{proof}

Given any $B \in \sym{1}{\H}$, it is well known that there is a conjugation $C_B$ which commutes with $B$, \emph{i.e.} $C _B : \dom{B} \rightarrow \dom{B}$ and $C_B B = B C_B$. Recall here that a conjugation is an anti-linear, idempotent onto isometry \cite[Theorem 7.1]{Krein}. It will be useful for us to extend this construction to the case of arbitrary $B \in \symm$. We say that $C$ is a conjugation intertwining $B_1 \in \sym{n}{\H _1}$ and
$B_2 \in \sym{n}{\H _2 }$ provided that $C B_1 = B_2 C $, and $C$ is an anti-linear and onto isometry.

\begin{prop}
    Let $\Theta$ be a contractive $\C ^{n \times n}$-valued analytic function in $\C _+$, $n \in \N $. The map $C _\Theta : \L (\Theta ) \rightarrow \L (\Theta ^T )$, defined
by $C _\Theta  F (z) = F ^\dag (z) :=  \ov{F(\ov{z} )}$ is a conjugation intertwining $\mf{Z} _\Theta$ and $\mf{Z} _{\Theta ^T}$, and $C _\Theta ^* = C _{\Theta ^T}$.
\end{prop}

In the above $^T$ denotes matrix transpose and for a vector $F (z)$, $\ov{F(z)}$ denotes the vector obtained by taking the complex conjugate of each component in the fixed canonical basis of $\C ^n$.

\begin{proof}

Let $\{ e_k \}$ denote the canonical orthonormal basis of $\C ^n$. Let $C : \C ^n \rightarrow \C ^n$ denote the conjugation defined by entrywise complex conjugation: if $\vec{v} = \sum c_i e_i $ for $c _i \in \C$,
then $C \vec{v} := \sum \ov{c_i} e_i$. Given any matrix $A \in \C ^{n\times n}$, with entries $A = \left[ a_{ij} \right]$, it is easy to check that $C A C = \left[ \ov{a_{ij} } \right] = (A ^*   ) ^T = (A^T ) ^* $.
By definition, given $F \in \L (\Theta )$, we have that $$(C _\Theta F ) (z) = C (F (\ov{z} )).$$

The closed linear span of the evaluation vectors $K_w ^\Theta \vec{v}$ for $w \in \C \sm \R$, $\vec{v} \in \C ^n$ is dense in $\L (\Theta )$. The action of $C _\Theta$ on such functions is
 \ba (C _\Theta K_w ^\Theta )(z) \vec{v} & = & C K_w ^\Theta (\ov{z} ) \vec{v} = C K_w ^\Theta (\ov{z} ) C  C \vec{v} \nonumber \\
& =& (K_w ^\Theta (\ov{z} ) ^T) ^* C \vec{v}. \nonumber \ea

Now \ba K_w ^\Theta (\ov{z} ) ^T & = & \left( \frac{i}{\pi} \frac{G _\Theta (\ov{z} ) + G_\Theta (w) ^* }{\ov{z} -\ov{w}} \right) ^T \nonumber \\
& = & \frac{i}{\pi} \frac{G _{\Theta ^T} (\ov{z} ) + G_{\Theta ^T} (w) ^* }{\ov{z} -\ov{w}}, \nonumber \ea since $G_\Theta = \frac{1+\Theta}{1-\Theta}$ so that
$G_\Theta ^T = G _{\Theta ^T}$. It follows that
\ba C K_w ^\Theta (\ov{z}) C & =& (K_w ^\Theta (\ov{z} ) ^T) ^* \nonumber \\
 & = & \frac{-i}{\pi} \frac{G _{\Theta ^T} (\ov{z} ) ^*  + G_{\Theta ^T} (w) }{z -w} \nonumber \\
& = & \frac{i}{\pi} \frac{G _{\Theta ^T} (z )   + G_{\Theta ^T} (\ov{w}) ^*  }{z -w} = K_{\ov{w}} ^{\Theta ^T } (z). \nonumber \ea
This proves that
$$ C _\Theta K_w ^\Theta  \vec{v} = K_{\ov{w}} ^{\Theta ^T} C \vec{v} \in \L (\Theta ^T ), $$ and it follows from the density of the point evaluation vectors that
$C _\Theta : \L (\Theta ) \rightarrow \L (\Theta ^T )$, and that it has dense range. It is clear by definition that $C _\Theta$ is anti-linear. To see that it is an (anti-linear) isometry
note that
\ba  \ip{C _\Theta K_w ^\Theta \vec{v}}{C _\Theta K_z ^\Theta \vec{w} } _\Theta  & = &  \ip{ K_{\ov{w}} ^{\Theta ^T} C \vec{v} }{ K _{\ov{z} } ^{\Theta ^T} C \vec{w} } _{\Theta ^T} \nonumber \\
& =& \left( K_{\ov{w} } ^{\Theta ^T} (\ov{z} ) C \vec{v} , C \vec{w} \right) _{\C ^n } \nonumber \\
& =& \left( C K_w ^\Theta (z) C C \vec{v} , C \vec{w} \right)  \nonumber \\
& = &  \left( \vec{w} , K_w ^\Theta (z) \vec{v} \right)   = \ip{ K_z ^\Theta \vec{w} }{K_w ^\Theta \vec{v} }_\Theta. \nonumber \ea
Using the fact that linear combinations of such functions are dense in $\L (\Theta )$ and $\L (\Theta ^T )$, we conclude that $C _\Theta$ is an isometry with dense range, and hence is onto. In other words,
$C _\Theta$ is anti-unitary, so that $C _\Theta ^* C_\Theta = \bm{1}$. As is easy to check:
$$ C _{\Theta ^T} C _{\Theta} K_w ^\Theta \vec{v} = C_{\Theta ^T} K_{\ov{w} } ^{\Theta ^T } C \vec{v} = K_w ^\Theta C ^2 \vec{v} = K_w ^\Theta \vec{v},$$ and it follows that $C _\Theta ^* = C _{\Theta ^T}$.

Finally, since $\dom{\Z _\Theta } := \{ F \in \L (\Theta ) | \ zF \in \L (\Theta ) \}, $ and similarly for $\Z _{\Theta ^T}$, $C _{\Theta } \dom{\Z _\Theta } = \dom{\Z _{\Theta ^T}}$.
Indeed, if $F \in \dom{\Z _\Theta}$, then
$$ C _\Theta z F(z) = C (\ov{z} F (\ov{z} )) =  z (C _\Theta F  ) (z), $$ so that $C_\Theta F \in \dom{\Z _{\Theta ^T}}$, and conversely given any $G \in \dom{Z _{\Theta ^T}}$, $ C _{\Theta ^T } G \in \dom{\Z _\Theta}$,
and $C _\Theta C _{\Theta ^T} G = G$, showing that $C _\Theta \dom{\Z _\Theta } = \dom{\Z _{\Theta ^T}}$. The above arguments also show that for any $F \in \dom{\Z _\Theta}$,
$$ C _\Theta \Z _\Theta F   = \Z _{\Theta ^T} C _\Theta F,$$  completing the proof.
\end{proof}

\begin{cor}
\label{conjugation}
Suppose that $B \in \sym{n}{\H }$ has characteristic function $\Theta _B$. Let $B_T \in \sym{n}{\H _T}$ have characteristic function $\Theta _B ^T$. Then there are conjugations
$C _B : \H \rightarrow \H _T$, and $C _{B^T} = C_B ^*$ such that $C _B B = B_T C_B$ and $C _{B^T} B_T = B C _{B^T}$.
\end{cor}

Note that any such conjugation $C _B$ obeys $C _B \ran{B-z} = \ran{B_T -\ov{z} }$, $C _B \ker{B^* -z} = \ker{B_T ^* -\ov{z} }$, and $C _B b (B) = b (B_T ) ^* C _B$.

\begin{proof}
    We have $B \simeq \Z _\Theta$ and $B_T \simeq \Z _{\Theta ^T}$. Composing the unitary operators effecting these equivalences with $C _\Theta$ yields $C _B$.
\end{proof}

\subsection{Measure spaces}

Let $\Sigma$ be any  $\C ^{n\times n}$ positive regular matrix-valued measure on $\R$ which obeys the Herglotz condition:
$$ \left( \intfty \frac{1}{1+t ^2 } \Sigma (dt)  \vec{v} , \vec{w} \right) _{\C ^n} < \infty, $$ for any $\vec{v}, \vec{w} \in \C ^n$.
We define the measure space $L^2 _\Sigma$ to be the space of all $\C ^n$-valued functions on $\R$ which are square-integrable with respect to $\Sigma$, \emph{i.e.} $f \in L^2 _\Sigma$ provided that $$ \intfty \left( \Sigma (dt) f(t) , f(t) \right) _{\C ^n } < \infty.$$ for any $z \in \C \sm \R$, define the $\C ^{n\times n}$ matrix function
$$ \delta _z (t)  := \frac{i}{\pi}\frac{1}{t-\ov{z}} \bm{1} _n. $$

Suppose that $\Theta$ is a contractive analytic function such that
$$ \re{G_\Theta (z)} = P \im{z}  + \intfty \re{\frac{i}{\pi} \frac{1}{t-z}} \Sigma (dt ). $$

The deBranges isometry $$ W _\Theta : L^2 _\Sigma \rightarrow \L (\Theta ), $$ defined by
$$ \left( (W _\Theta h ) (z) , \vec{v} \right) _{\C ^n} := \left( \frac{1}{i\pi} \intfty \ov{\frac{1}{\pi(t-\ov{z})}} \Sigma (dt) h(t) , \vec{v} \right) =  \ip{h}{\delta _z \vec{v}} _{\Sigma}, $$
where $\ip{\cdot}{\cdot} _\Sigma$ denotes the inner product in $L^2 _\Sigma$ is an isometry of $L^2 _\Theta := L^2 _\Sigma $ into $\L (\Theta)$. The range of $W_\Theta$ is $\L (\Psi ) \subset \L (\Theta )$ where $$ G _\Psi (z) = G_\Theta (z)  +i z P,$$  and the orthogonal complement of the range of $W_\Theta$
is the closed linear span of the constant functions $ \bigvee P \C ^n.$
 One can then check that the reproducing kernel for $\L (\Theta )$ is given by the formula
\be  \left( K _w ^\Theta (z) \vec{v} , \vec{w} \right) _{\C ^n} = \left( (\pi W \delta _w (z) + \frac{P}{\pi}) \vec{v} , \vec{w} \right) = \ip{\delta _w \vec{v} }{\delta _z \vec{w} } _\Sigma + \left( \frac{P}{\pi} \vec{v} ,
\vec{w} \right) _{\C ^n} \label{Hergkern} \ee

Also notice that if $P = 0$ and  $\Theta $ is a characteristic function so that $\Theta (i) =0$, that this implies that $G _\Theta (i) = \bm{1}$ so that
$$ \bm{1} =  \intfty \re{\frac{1}{i\pi} \frac{1}{t-i}} \Sigma (dt ) = \intfty \frac{1}{1+t^2} \Sigma (dt), $$ and this implies that the vectors
$\delta _i e_k$, $1 \leq k \leq n$ are an orthonormal set.

\section{Non-canonical representations of symmetric operators}
\label{represent}

We are now sufficiently prepared to begin pursuing the main theory and results of this paper. For any $A \in \Ext{B}$ we can construct a representation of $B$ as multiplication on a space of analytic functions on $\C \sm \R$ as follows:

Let $$\mc{K} _z := \mc{K} \ominus \ran{B  -\ov{z}} = ( \mc{K} \ominus \H  ) \oplus \ker{B ^* -z }. $$

For any $w, z \in \C \sm \R$, if $A$ is densely defined (so that $A = b^{-1} (U)$ and $U$ does not have $1$ as an eigenvalue) let \be  U_{w,z} := (A-w) (A-z) ^{-1}. \label{normal} \ee If
However $A = b^{-1} (U) P _{U} (\T \sm \{ 1 \} )$ and $U$ is a unitary extension of $V = b(B)$ which has $1$ as an eigenvalue let
\be U_{w,z} := \left( (i-w) + U (i + w ) \right) \left( (i-z) + U (i +z) \right) ^{-1}. \label{except} \ee  These two formulas coincide when $U$ does not have $1$ as an eigenvalue.

Then it is not difficult to verify as in \cite[Section 1.2]{Krein} that (regardless of whether $A$ is densely defined or not) for any $w, z \in \C \sm \R$, $U_{w,z}$ has the following properties:
\bn
    \item $U_{w,z}$ is invertible.
    \item $U _{w,z} : \mc{K} _w  \rightarrow \mc{K} _z$ is a bijection.
\en

Note that
$$ P_\H U_{w,z} \ker{B ^* -w} \subset P _\H \left( \ker{B ^* -z} \oplus (\K \ominus \H ) \right) \subset \ker{B^* -z}. $$

Given any fixed $w \in \C \sm \R$, let $J_w : \C ^n \rightarrow \ker{B^* -w }$ be a bounded isomorphism (a bounded linear map with bounded inverse).
We can then define the map
$$ \Gamma _A ^w: \C \sm \R \rightarrow \mc{B} (\C ^n , \H ) ,$$ by
\be \Gamma _A ^w (z) := P _\H U _{w, \ov{z}} P_w J_w = P_\H (A-w)(A-\ov{z}) ^{-1} J_w, \label{ncmodel} \ee (the last formula holds for the case where $A$ is densely defined) where $P_w$ projects onto $\ker{B^* -w}$ and it follows that if $A \in \Ext{B}$ is actually a canonical element of $\Ext{B}$ that
$\Gamma _A $ is a \emph{model} for $B$ as defined in \cite{AMR}. Namely, recall:

\begin{defn}
Given $B \in \sym{n}{\H}$, let $\J$ be a Hilbert space with $\dim{\J} =n$. A map  $\Gamma: \C \setminus \R \to \mathcal{B}(\mathcal{J}, \mathcal{H}),$ the space of bounded linear maps from $\J$ to $\H$, is a
  \emph{model} for $B$ if $\Gamma$ satisfies the following conditions:
\begin{equation} \label{I}
\Gamma: \C \setminus \R \to \mathcal{B}(\mathcal{J}, \mathcal{H}) \quad \mbox{is co-analytic};
\end{equation}
\begin{equation}
\Gamma(\lambda): \mathcal{J} \to \ran{B  - \lambda I} ^{\perp} \quad \mbox{is invertible for each $\lambda \in \C \setminus \R$};
\end{equation}
\begin{equation}
\Gamma(z)^{*} \Gamma(\lambda): \mathcal{J} \to \mathcal{J} \quad \mbox{is invertible when $\lambda, z \in \C_{+}$ and when $\lambda, z \in \C_{-}$}; \label{condition}
\end{equation}
\begin{equation}
\bigvee_{\Im \lambda \not = 0} \ran{ \Gamma(\lambda)} = \mathcal{H},
\end{equation}
where $\bigvee$ denotes the closed linear span. \label{model}
\end{defn}

Recall that as shown in \cite{AMR}, any model $\Gamma$ for $B \in \sym{n}{\H}$ can be used to construct a reproducing kernel Hilbert space of analytic functions $\H _\Gamma$ on $\C \sm \R$ and
a unitary $U_\Gamma : \H \rightarrow \H _\Gamma$ such that the image of $B$ under this unitary transformation acts as multiplication by $z$.

Now if $A \in \Ext{B}$ is non-canonical then $\Gamma _A ^w$ as defined in equation (\ref{ncmodel}) does not necessarily satisfy the conditions of a model as defined in Definition \ref{model}.
Despite this $\Gamma _A ^w$ has similar properties to a model and can still be used to construct a reproducing kernel Hilbert space of analytic functions $\H _A$ on $\C \sm \R$, and (at least in the case under consideration where $\Theta _B$ is inner) an isometry $U_A : \H \rightarrow \H _A$ such that $U_A B U_A ^*$ again acts as multiplication by $z$ in $\H _A$.

This motivates the definition of a non-canonical model which includes these generalized models $\Gamma _A$ arising from non-canonical $A \in \Ext{B}$:

\begin{defn}
    Let $\J $ be any $n-$dimensional Hilbert space and suppose that $B \in \sym{n}{\H}$. If $\mc{B}(\J , \H)$ is the space of bounded linear maps from $\J$ to $\H$, we say that $\Gamma : \C \sm \R \rightarrow \mc{B} (\J , \H)$ is a \emph{quasi-model} for $B \in \sym{n}{\H}$ if $\Gamma$ satisfies the following two conditions:
 \be \Gamma : \C \sm \R \rightarrow \mc{B} (\J , \H) \quad \mbox{is co-analytic}; \ee
 \be \Gamma (z) : \J \rightarrow \ker{B^* - \ov{z}}. \ee
\end{defn}

Given a quasi-model $\Ga$, we define
 \be m_{\pm} := \max _{z\in \C _\pm} \dim{\ker{\Gamma (z) } ^\perp} ,\ee $\Ga$ is then said to have \emph{rank} $(m_- , m_+)$, $0 \leq m_\pm \leq n$. The quasi-model $\Gamma$ is said to have full rank if $m_+ = n = m_-$.

\subsection{Basic properties of quasi-models}

\begin{defn}
    Let $\Pi _\Ga ^+$ be the set of all points in $\C _+$ for which $\dim{\ker{\Gamma (z) ^\perp}} = m _+$, and define $\Pi _\Ga ^-$ similarly.
Let $\Sigma _\Ga ^\pm := \C _\pm \sm \Pi _\Ga ^\pm$. We will also use the notation $\Pi _\Ga = \Pi _\Ga ^+ \cup \Pi _\Ga ^-$ and $\Sigma _\Ga = \Sigma _\Ga ^+ \cup \Sigma _\Ga ^-$.

\end{defn}

We will now show that any quasi-model $\Gamma$ of rank $(n, n)$ has a property similar to the property (\ref{condition}) for a model.

\begin{prop}
    If $B \in \sym{n}{\H}$ and $\Gamma$ is a quasi-model for $B$  then $\Gamma ( z ) ^* \Gamma (w) $ is a quasi-affinity on $\J$ whenever $m_+ =n$ and $z,w \in \Pi _\Gamma ^+$ or
whenever $m_- = n$ and $z,w \in \Pi _\Gamma ^-$.
  \label{qap}
\end{prop}

\begin{remark}
Note that in the case where $n<\infty$, which is the case we are primarily studying, when $\Gamma (z) ^* \Gamma (w)$ is a quasi-affinity,
it is acting between finite dimensional spaces and hence is in fact bounded and invertible.  Also the reason this proposition is important is that we will shortly construct a reproducing kernel Hilbert space $\H _\Ga$ whose reproducing kernel is $K_w (z) = \Gamma ^* (z) \Gamma (w)$, and it will be useful to know when this is invertible.
\end{remark}

This proposition will be the consequence of the following:

\begin{prop}
    For each $z \in \C \sm \R$, let $\{ \delta  _k(z) \} _{k=1} ^n$ be a basis for $\ker{B^* -z}$.  Then the linear operator $Y$ on $l^2 (\N )$ with entries
    $$ Y (w,z) := \left[ \ip{\delta _j (w) }{\delta  _k (z)} \right] _{1\leq j,k \leq n} ,$$ is a quasi-affinity for any $z,w \in \C _+$ or $z,w \in \C _-$, \emph{i.e.} it is injective and has dense range (and
    hence an inverse which is potentially unbounded). \label{invert}
\end{prop}

The proof of this proposition needs a little set up. Given a closed linear transformation $T$ with domain $\dom{T} \subset \H$, a point $z \in \C$ is called a \emph{regular} point of $T$ if
$T-z$ is bounded below on $\dom{T}$, \emph{i.e.}, $\| (T-z)f) \| \geq c_z \| f \| $ for all $f \in \dom{T}$. Let $\Om _T$ denote the set of regular points of $T$. If $B \in \sym{n}{\H}$, then since
$B$ is symmetric we have that $\C \sm \R \subset \Om _B \subset \C$.  The symmetric linear transformation $B$ is called regular if $\Om _B = \C$. For any $z \in \Om _B$, let
$\mf{G} _z $ be the closure of the linear relation:
$ \mf{G} (B) \dotplus \{ (h_z , z h_z) | \ h_z \in \ker{B^* -z} \} ,$ and $\dotplus$ denotes the non-orthogonal direct sum of linearly independent subspaces.

\begin{lemma}
    There is a closed linear operator $B_z$ extending $B$ such that $\mf{G} (B_z) = \mf{G} _z$.
\end{lemma}

\begin{proof}
    It suffices to prove that $\mf{G} _z$ is the graph of a densely defined closed linear operator.

Clearly $\mf{G} (B_z) \subset B^*$. To prove that $\mf{G} (B_z)$ is the graph of a linear transformation, we need to prove that the intersection of the multi-valued part of $B^*$ with $\mf{G} _z$ is the zero element:
$$ \{ (0 ,g ) | \ g \in B^* (0) \} \cap \mf{G} _z = \{ (0, 0 ) \} ,$$  where recall that $B^* (0) = \dom{B} ^\perp$.

Suppose not, then we can find a sequence $(f_n ) \subset \dom{B}$ and a sequence $h_n \in \ker{B^* -z}$ such that $(f_n +h_n, Bf_n + zh_n ) \rightarrow (0, g) $ where $g \perp \dom{B}$.
It follows that $$ (B - z) f_n  = Bf_n +z h_n -z(f_n +h_n) \rightarrow g - 0 = g.$$ Since $\ran{B-z}$ is closed it follows that there is an $f \in \dom{B}$ such that
$$ (B-z) f =g \perp \dom{B}.$$ However this would then imply that
$$ 0 = \ip{g}{f} = \ip{(B-z)f}{f} = \ip{Bf}{f} - z \ip{f}{f}, $$ which is impossible as $B$ is symmetric and $z \in \C \sm \R$. This proves
that $\mf{G} _z$ is the graph of a linear transformation $B_z$, it remains to prove that $B_z$ is densely defined.

To prove that $B_z$ is a linear operator, \emph{i.e.} densely defined, suppose that $\phi \in \H$ is orthogonal to $\dom{B_z}$. Then $\phi \perp \dom{B}$ and $\phi \perp \ker{B^* -z}$. Hence $\phi \in \ran{B-\ov{z}}$ and so $\phi = (B-\ov{z})f$ for
some $f \in \dom{B}$. But $\phi$ is orthogonal to $\dom{B}$ as well so that
$$  0 = \ip{f}{\phi} = \ip{f}{(B-\ov{z}) f},$$ showing that
$$ \ip{Bf}{f} = z \ip{f}{f}, $$ which as before is impossible as $B$ is symmetric.
\end{proof}

\begin{lemma}
    Suppose that $z \in \C \sm \R$. The spectrum of the operator $B _z$ is contained in $\ov{\C _+}$ or $\ov{\C _-}$ when $z \in \C_+$ or $\C_-$, respectively.
\end{lemma}

Since $B_z$ is a closed linear operator, the proof is identical to that of \cite[Lemma 2.6]{AMR}, and we omit it.

\begin{proof}{ (of Proposition \ref{invert})}
Given a unit vector $\vec{c} \in \C ^n $ (we take $\C ^\infty := \ell ^2 (\N)$), let
$$ \Delta _{\vec{c}} (z) := \sum _{k} \ov{c_k} \delta _k (z).$$

Now observe that
$$ Y (w,z) \vec{c} = \left( \ip{\delta _j (w) }{\Delta _{\vec{c}} (z)} \right) _{1 \leq j \leq n}.$$  Now if $Y (w,z)$ was not injective then there would be a $\vec{c} \in \C ^n$ for which
$Y \vec{c} = 0$, and hence $ 0 = \ip{\delta _j (w) }{\Delta _{\vec{c}} (z) } $ so that $\psi _z := \Delta _{\vec{c}} (z) \perp \ker{B^*  -w}$ and hence $\psi _z \in \ran{B-\ov{w}}$,
$\psi _z = (B- \ov{w}) f$ for some $f \in \dom{B}$. But then, since $\ov{w}$ does not belong to the spectrum of $B_z$,
$$(z - \ov{w} ) ^{-1} \psi _z = (B_z - \ov{w} ) ^{-1} \psi _z = f ,$$ which shows that $\psi _z \in \dom{B}$, contradicting the fact that $B$ is symmetric.

Hence $Y(w,z)$ is injective whenever $w,z \in \C _+$ or in $\C _-$. But then $Y^* (w,z) = Y (z,w )$ is also injective, proving that $Y(w,z)$ also always has dense range. This proves that $Y(z,w)$
is always a quasi-affinity of $\mc{B} (\ell ^2 (\N ))$ whenever $z,w$ are both in $\C _+$ or are both in $\C _-$.

\end{proof}

\begin{proof}{ (of Proposition \ref{qap})}

If $z, w \in \Pi _\Gamma ^+$ this follows from the observation that given any orthonormal basis $\{ j_k \}$ of $\J$, and $z \in \Pi _\Gamma ^+$, $\delta _k (\ov{z}) := \Gamma (z) j_i $ forms a basis for $\ker{B^* -\ov{z}}$, and that
$$ \Gamma ( z )  = \sum \ip{ \cdot}{j_i} \delta _i (\ov{z}), $$ so that
$$ \Gamma ^* (z) \Gamma (w) = \left[ \ip{ \delta _j (\ov{z}) }{\delta _k (\ov{w})} \right] _{1 \leq j,k \leq n }.$$

The proof of the other half of the proposition is analogous.

\end{proof}

For the remainder of this section we will assume that $n<\infty$, although
many of our arguments generalize to the case $n=\infty$ without too much difficulty.

\begin{lemma}
    The sets $\Sigma _\Gamma ^\pm = \C _\pm \sm \Pi _\Gamma ^\pm $ are contained in the zero-sets of  non-zero analytic functions in $\C _\pm$ (and hence are purely discrete with accumulation points lying only on $\R \cup \{ \infty \}$). \label{inv}
\end{lemma}

\begin{proof}
   Choose any $w \in \Pi _\Gamma ^+$. Let $\{ j_k \}$ be an orthonormal basis of $\J$ such that $\{ j_k \} _{k=1} ^{m_+}$ is an orthonormal basis of $\ker{\Gamma (w)} ^\perp$. Let
$\{ v_k \} _{k=1} ^{m_+}$ be the basis of $\ran{\Gamma (w)}$ defined by $v_k = \Gamma (w) j_k$ and set
$$ D_w (z) := \left[ \ip{\Gamma(w) j_k }{\Gamma (z) j_l } \right] _{1\leq k,l \leq m_+}, $$ and let $\delta _w (z) := \det{D_w (z)}$. Then $\delta _w $ is analytic (as a function of $z$) in $\C _+$ and
$\delta _w $ is not identically zero since $\delta _w (w) = \det{D_w (w) }$, and it is clear that by construction $D_w (w)$ is invertible. Now if $z \in \C _+$ is any point such that
$\delta _w (z) \neq 0$ then $D_w (z)$ is invertible and hence $\Gamma (z) | _{\ker{\Gamma (w)} ^\perp} $ is invertible as a map onto its range. Let $\wt{j} _k := P_z j_k$ where $P_z$ projects
onto $\ker{\Gamma (z) } ^\perp$. The $\wt{j} _k$ form a linearly independent set since otherwise the set of all
$$ \Gamma (z) j_k = \Gamma (z) \wt{j} _k,$$ would not be linearly independent, contradicting the fact that $\Gamma (z) | _{\ker{\Gamma (w)} ^\perp} $ is invertible. It follows that
$$ \dim{ \ker{\Gamma (z)} ^\perp } \geq m_+ = \max _{z \in \C _+} \dim{\ker{\Gamma (z) ^\perp}}, $$ for any $z \in \C _+$ such that $ \delta _w (z) \neq 0$, proving the claim.
\end{proof}

\begin{cor}
    Given any $w \in \Pi _\Gamma ^\pm $ we have that the set $$ \C _\pm \sm \{ z \in \C _\pm | \ \Gamma (z) ^* \Gamma (w) | _{\ker{\Gamma (w)  ^\perp} } \ \mbox{is invertible} \ \},$$ is contained
in the zero set of an analytic function which is not identically zero.
\end{cor}

\begin{lemma}
   Suppose that $n<\infty$. If $m_+ =n$ then $\bigvee _{z \in \C _+} \Gamma  (z ) \J = \bigvee _{z \in \C _+} \ker{B^* -\ov{z}}$. Similarly if $m_- = n$ then $\bigvee _{z \in \C _-} \Gamma  (z ) \J = \bigvee _{z \in \C _-} \ker{B^* -\ov{z}}$. Consequently if $m_+ = n = m_-$ then the simplicity of $B$ implies that $\bigvee _{z \in \C \sm \R} \Gamma (z) \J = \H$.  \label{lemma:lspan}
\end{lemma}

\begin{proof}
    This is intuitively clear. Since $B$ is simple, $\bigvee _{z \in \C \sm \R} \ker{B^* -z} $ is dense in $\mc{H}$.
By definition if $z \notin \Sigma _\Gamma ^+$ and $m_+ =n$ then $\ker{B^*-\ov{z}} =  \Gamma (z ) \J $. By Lemma \ref{inv} the set $\Pi _\Gamma ^+$ of all $z \in \C _+$ for which $\Gamma (z)$ is invertible is dense in $\C _+$.

If $f \in \H$ and $f \perp \bigvee _ {z \in \C _+}  \Gamma (z) \J $ then $f \perp \ker{B^* -\ov{z}}$ for all $z \in \Pi _\Gamma ^+$. Let $\wt{\Ga}$ be a canonical model for $B$, and let $\wt{f}  (z) :=\wt{\Ga} (z) ^* f$. Since $f \perp \ker{B^* -\ov{z}}$ for all $z \in \Pi _\Gamma ^+$, the $\J$-valued analytic function $\wt{f} (z)$ vanishes everywhere on $\Pi _\Gamma ^+$. Since this set is dense in $\C _+$, $\wt{f}  =0$ identically on $\C _+$. This shows that $f \perp \bigvee _ {z \in \C _+}  \ker{B^* -\ov{z} }$. The same argument in $\C _-$ completes the proof.
\end{proof}

\begin{defn}
    We say that a quasi-model $\Gamma$ is a \emph{generalized} or \emph{non-canonical model} for $B$ if $$ \bigvee _{z \in \C \sm \R} \ran{\Gamma (z)} = \H.$$

    \label{genmodel}
\end{defn}

By Lemma \ref{lemma:lspan}, any full rank quasi-model (a rank $(n,n)$ quasi-model) is a generalized model for $B$. The next proposition verifies that the linear maps $\Ga _A ^w$ defined for $A \in \Ext{B}$ and $w \in \C \sm \R$ in equation (\ref{ncmodel}) satisfy our definition of a quasi-model.

\begin{prop}
If $B \in \sym{n}{\H}$ with $\Theta _B$ inner and $A \in \Ext{B}$, then for any $w \in \C \sm \R$ one can construct a generalized model $\Gamma ^w _A$ for $B$ by defining $\J := \C ^n$, $J_w : \J \rightarrow \ker{B^* -w}$ a bounded isomorphism and letting
$$ \Gamma _A ^w (z) := P _\H U_{w, \ov{z}} J_w.$$ The quasi-model $\Gamma ^w _A$ has rank
$( n, m_+ )$ if $w \in \C _+$ and rank $(m_-, n)$ if $w \in \C _-$ where $0 \leq m_\pm \leq n$.
\end{prop}

We will usually assume that $J_w$ is chosen to be an isometry. Recall that if $A$ is such that $A = b^{-1} (U)$ and $1 \notin \sigma _p (U)$, then $U_{wz} = P _\H (A-w) (A-\ov{z} ) ^{-1}$, as in equation (\ref{normal}).
In the exceptional case where $1 \in \sigma _p (U)$, $U_{wz}$ is given by equation (\ref{except}).

\begin{proof}
    First, clearly $\Gamma _A ^w $ is anti-analytic on $\C \sm \R$. Also as discussed previously, $\Gamma _A ^w  (\ov{z}) \in \ker{B^* -z}$ since
$U_{w,z} $ maps $\ker{B ^* -w}$ into $(\K \ominus \H ) \oplus \ker{ B_A ^* -z}$ (as discussed at the beginning of this section).
\end{proof}

Note that by construction $\Gamma ^w _A (\ov{w}) = J_w$, which is invertible by assumption.

Given any $A \in \Ext{B}$, we are free to choose $w \in \C \sm \R $ in the construction of a quasi-model $\Gamma _A ^w $ associated with $A$. For the remainder of this paper we will choose $w =-i$ unless otherwise specified
and define
$$ \Gamma _A (z) := \Gamma _A ^{-i} (z), $$ which (excluding the exceptional case) is equal to
$$ P_\H (A+i) (A-\ov{z} ) ^{-1}  J_{-i}, $$ and $\Gamma _A (i) = J_{-i} $.
We will also simply write $J$ for $J_{-i}$ where $J = P_{-i} J : \C ^n \rightarrow \ker{B^* +i}$, and usually we assume $J$ is an isometry.

\begin{remark}
Suppose that the characteristic function $\Theta $ of $B$ is inner.  If this is the case then we will show that for any $A \in \Ext{B}$, that $\bigvee _{z \in \C _-} \ran{ \Gamma _A ^{w}  (\ov{z})}  = \H$ for any $w \in \C_+$ and $\bigvee _{z \in \C _+} \ran{ \Gamma _A ^{w}  (\ov{z})}  = \H$ whenever $w \in \C _-$.

To see this note that in this case that $B$ is unitarily equivalent to $Z_\Theta$, which acts as multiplication by $z$ in some model space $K ^2 _\Theta$. Suppose that $U : \H \rightarrow K^2 _\Theta$ is this unitary transformation such that $U ^* Z _\Theta U = B$. Let $C_\Theta = \dag \circ \Theta ^*$ be the canonical anti-linear isometry from $K^2 _\Theta$ onto $K^2 _{\Theta ^T}$, where $T$ denotes transpose, as defined in \cite[Claim 3]{Martin-semi}. The existence of $C_\Theta$ also follows from our Corollary \ref{conjugation}.

Suppose that $w \in \C _+$. Then by Lemma \ref{lemma:lspan}, since $m_- = n$ for any $\Gamma _A ^{w}$ (because $\Gamma _A ^w (\ov{w})$ is invertible),
\ba \bigvee _{z \in \C _-} \ran{\Gamma ^w _A (z)} & = & \bigvee _{z \in \C _+} \ker{B^* -z} \nonumber \\
& =&  U^* \bigvee _{z \in \C _+} \ker{Z_\Theta ^* -z } \nonumber  \\
& =& U^* \bigvee \{ C_\Theta k_z ^{\Theta ^T} \} \nonumber  \\
& =& U^* K^2 _\Theta = \H . \nonumber \ea
Similarly if $w \in \C _-$ then
\ba  \bigvee _{z \in \C _+ } \ran{\Gamma _A ^w (z)} & = & U^* \bigvee _{z \in \C _- } \ker{Z_\Theta ^* -z } \nonumber \\
& = & U^* \bigvee \{  k_z ^{\Theta } \} \nonumber \\
& =& U^* K^2 _\Theta = \H . \nonumber \ea

This proves that if $\Theta _B$ is inner, then every quasi-model $\Gamma _A ^w$ for
$A \in \Ext{B}$ and $w \in \C \sm \R$ is a generalized model.

\label{genex}
\end{remark}

\begin{eg}{ (An example of $Z_A$ with indices $(n, n)$ where $m_+ = m < n$.)}

Suppose that $B \in \sym{n}{\H}$ and $\Theta _B$ is inner so that for any $A \in \Ext{B}$ and $w \in \C \sm \R$, $\Gamma _A ^w$ is a generalized model for $B$ (see Remark \ref{genex} above).

Let $V:=$ the partial isometric extension of $b(B)$ to $\H$. Given any $C \in \ov{B_1 (\C ^{n\times n})}$ define
$$\hat{C} := \sum _{jk} C_{jk} \ip{\cdot}{u_j}{v_k}, $$ where $\{ u_j \}$ is an orthonormal basis of $\ker{V} = \ker{B^* -i}$ and
$\{ v_k \}$ an orthonormal basis of $\ran{V} ^\perp = \ker{B^* +i}$. Let $$ V (C) := V + \hat{C},$$ a contractive extension of $V$ and let $( U _C , \K )$ be the minimal unitary dilation of $V (C)$.  Choose $C = \bm{1} _m$ where $0 \leq m < n$ so that
$\hat{C} u_j = 0$ for any $n \geq j > m$. Let us assume that $V(C)$ does not have $1$ as an eigenvalue. This is the case, for example, if $B$ is densely defined (see \emph{e.g} \cite[Lemma 6.1.3]{Martin-uni}). Then it follows from \cite[Proposition 6.1, Chapter 2]{NF}, that $1$ is not an eigenvalue of $U$ so that $b^{-1} (U _C ) =: A _C \in \Ext{B}$ and $U _C = b(A_C)$. Define
$$ \Gamma _C (z) := \Gamma _{A _C}  ^i (z) = P _\H (A_C -i) (A_C -\ov{z}) ^{-1} J_i ,$$ where $J_i : \C ^n \rightarrow \ker{B^* -i}$ is chosen to be an isometry such that $J_i e_k = u_k$, where $\{e_k \}$ is the standard orthonormal basis of $\C ^n$. Now
$$ \Gamma _C   (i) = P_\H (A_C -i) (A_C +i) ^{-1} J_i = P_\H b(A_C ) P_\H J_i = V (C) J_i = \hat{C}, $$ since $b(A_C) =U_C$ is an extension of $V(C)$. It follows that $$\Gamma _C (i) e_j =0,$$ for any $n\geq j > m$. Now given any $z \in \C _+$,
\ba \Gamma _C (z) & = & P_\H (A_C -i) (A_C -\ov{z} )^{-1} J_i \nonumber \\
& =& P_\H (A_C +i ) (A_C -\ov{z} ) ^{-1} (A_C -i) (A_C -\ov{i} ) ^{-1} P_\H J_i . \ea Since both $z , i \in \C _+$, by dilation theory this is just equal to
$$ \Gamma _C (z) = P_\H (A_C +i)(A_C -\ov{z} ) ^{-1} P_\H b(A _C) P_\H J_i, $$ so that $\Gamma _C (z) e_j =0$ for any $n\geq j > m$ as before. To see this note that since $U_C = b(A_C)$ is a dilation of
$V(C)$, that for any $n \in \N \cup \{0 \}$, $$ P_\H b(A_C) ^n P _\H = (P_\H b(A_C) P _\H) ^n $$ It follows that $$ P _\H (A_C +i ) ^{-n} P_\H = ( P_\H (A_C +i ) ^{-1} P_\H ) ^n. $$ Given any $z \in \C _-$
that lies in the open ball of radius $1$ about $z = -i$ we have that $(A_C -z ) ^{-1}$ can be expressed as a power series in $(A_C +i ) ^{-1}$, and it follows from this that the resolvent formula:
\ba  (z-w) P_\H (A_C -z ) ^{-1} (A_C -w ) ^{-1} P _\H & = & P_\H (A_C -w) ^{-1} P _\H  - P _\H (A-z) ^{-1} P_\H \nonumber \\
& = & (z-w) P_\H (A -z ) ^{-1} P_\H (A-w ) ^{-1} P _\H, \nonumber  \ea holds for all $z, w \in \C _-$.

 Hence $\Gamma  (z) u_j$ is identically zero in $\C _+$ for any $n \geq j >n$  so that
$$ m _+ = \max _{z \in \C _+} \ker{\Gamma (z) } ^\perp = m.$$

Similarly using $\Gamma  = \Gamma _{A _C} ^{-i}$ instead, one can construct an example of $Z_A $ with indices $(n,n)$ where $n > m_-$.
\end{eg}

\section{Construction of the model reproducing kernel Hilbert space}

Given any quasi-model $\Gamma$ for $B \in \sym{n}{\H}$, we can construct a reproducing kernel Hilbert space $\H _\Gamma$ as follows:

\begin{defn}
For $f \in \H$ define
$$ \hat{f} (z) :=  \Gamma ^* (z) f, $$ an analytic function on $\C \sm \R$, and let $\H _\Gamma :=$ the vector space of
all the functions $\hat{f} $.

Let $\H _\Gamma ^{-1}  = \bigvee _{z\in \C \sm \R } \ran{\Gamma (z)}$, $\H _\Gamma ^{-1} \subset \H$.  Then clearly $\H _\Gamma$ is the set of all
functions $\hat{f} $ for $f \in \H _\Gamma ^{-1}$, and if $f \perp \H _\Gamma ^{-1}$ then $\hat{f} =0$. We define an inner product on $\H _\Gamma$ by
$$ \ip{\hat{f}}{\hat{g}} _{\Gamma }  := \ip{f}{g} ,$$ whenever $f,g \in \H _\Gamma ^{-1}$.

According to Definition \ref{genmodel}, we call $\Gamma$ a generalized model if $\H _{\Gamma} ^{-1} = \H$.
\end{defn}

We say that the reproducing kernel Hilbert space $\H _\Ga$ has the division property in $\C _\pm$ if whenever $\hat{f} \in \H _\Ga$ and
$\hat{f} (w) = 0$ for $w \in \Pi _\Ga ^\pm$ we have that $$\frac{\hat{f} (z )}{z-w} \in \H _\Ga.$$

\begin{prop}
    With the above inner product $\H _\Gamma$ is a reproducing kernel Hilbert space of analytic functions on $\C \sm \R$ with reproducing kernel
$$ k_w ^\Gamma (z) = \Gamma (z) ^* \Gamma (w), $$ and point evaluation vectors $$ k_w ^\Gamma j = U_\Gamma \Gamma (w) j ,$$ for $j \in \J$. If the rank of
$\Ga$ is $(m_+ , m_- )$ then $\H _\Ga$ has the division property in $\C _\pm$ whenever $m_\pm =n$.

The map $U _\Gamma : \H \rightarrow \H _\Gamma$ defined by $U_\Gamma f = \hat{f}$ is a co-isometry
with initial space $\H _\Gamma ^{-1}$, and is unitary if and only if $\Gamma$ is a generalized model for $B$. If $\Gamma$ is a generalized model then
$Z _\Gamma := U _\Gamma B U_\Gamma ^{-1}$ acts as multiplication by $z$ on the domain $U_\Ga \dom{B}$,
and if either $m_+$ or $m_-$ is equal to $n$ then
$$ \dom{Z_\Gamma } = \{ \hat{f} | \ z \hat{f} (z) \in \H _\Gamma \}. $$

\label{modelspace}
\end{prop}

Recall that if $\Theta _B$ is inner then given any $A \in \Ext{B}$, and $w\in \C \sm \R$, any
quasi-model $\Ga _A ^w$ is a generalized model with indices $(m_+ , n)$ or $(n , m_-)$.

\begin{proof}
This is all fairly straightforward to check. First of all one should verify that $\| \hat{f} \| _\Gamma = 0$ implies that $f \perp \H _\Gamma ^{-1}$,
\emph{i.e.} that $\hat{f} (z) = \Gamma (z) ^* f =0$ for all $z \in \C \sm \R $.
Indeed $\Gamma (z) ^* f  =0$ for $z \in \C \sm \R$ if and only if $\ip{f}{\Gamma (z) j } = 0 $ for all $z \in \C \sm \R$, and $j \in \J$ which happens if and only if
$f \perp \ran{\Gamma (z)}$ for all $z \in \C \sm \R$, in other words $f \perp \H _\Gamma ^{-1}$.

Now given any $j \in \J$ and $f \in \H _\Gamma ^{-1}$,
\be \ip{\hat{f} (w)}{j} _\J = \ip{f}{\Gamma (w) j} _\H = \ip{\hat{f}}{U_\Ga \Gamma  (w) j} _{\Gamma} ,\ee and it follows from this that for any $j \in J$,
$k_w j := U_\Ga \Gamma (w) j$ are reproducing kernel vectors in $\H _\Ga$ and the reproducing kernel is given by
\ba  \ip{k_w (z) j_1 }{j_2} _\J & = &\ip{k_w j_1}{k_z j_2} _\Ga = \ip{\Ga (w) j_1}{\Ga (z) j_2} _\H \nonumber \\
& = & \ip{\Ga (z) ^* \Ga (w) j_1}{j_2 } _\J. \nonumber \ea

Now suppose $U_\Gamma$ is unitary and define $Z_\Gamma := U_\Gamma B U_{\Gamma} ^{-1}$ on $U_\Gamma \dom{B}$. Let us first show that $Z_\Gamma$ acts as multiplication by $z$ on its domain.
If $f \in \dom{B}$ then $$ Z_\Gamma \hat{f} = U_\Ga Bf $$ so that for any $j \in \J$,
\ba \ip{ (U_\Ga Bf) (z) }{j} _\J & =& \ip{Bf}{\Gamma (z)j} _\H  = \ip{f}{B^* \Gamma (z) j} \\
& =& z \ip{f}{\Gamma (z) j}  = \ip{z \hat{f} (z)}{j} _\J ,\ea showing that $Z_\Gamma \hat{f} (z) = z \hat{f} (z)$.

Now suppose that $m_+ = n$, and let's prove that $\H _\Ga$ has the division property in $\C _+$. In this case if $\hat{f} \in \H _\Ga$ and
$\hat{f} (w) = 0$ then
$$ 0 = \ip{\hat{f}(w)}{ j } = \ip{f}{\Ga (w) j}, $$ for any $j \in J$. If $w \in \Pi _\Ga ^+$, then $\Ga (w) : \J \rightarrow \ker{B^* -\ov{w} }$ is onto
which implies that $f \in \ran{B-w}$. Then $f =(B-w) g$, and $\hat{f} = (Z_\Gamma -w ) \hat{g}$, or
$$ \hat{g} (z) = \frac{\hat{f} (z)}{z-w}. $$

It remains to prove that if (without loss of generality) $m_+ =n$ and $\hat{f} \in \H _\Gamma$ is such that $z \hat{f} (z) \in \H _\Gamma$ then $\hat{f} \in \dom{Z_\Gamma}$. If $\hat{f}$ and $z \hat{f} \in \H _\Gamma$
then so is $(z-w) \hat{f} =: \hat{g} $ for any fixed $w \in \Pi _\Gamma ^+$, and some $g \in \H$. Since $\hat{g} $ vanishes at $w$, it follows that $\Gamma ^* (w) g = 0$, so that
$g \perp \ran{\Gamma (w)} = \ker{B^* -\ov{w}}$ since $w \in \Pi _\Gamma$. It follows that $g = (B-w) h$ for some $h \in \H$ so that $\hat{g} (z) = (z-w) \hat{h} (z)$ for any
$z \in \C \sm \R$. But since $\hat{g} (z) = (z-w) \hat{f} (z)$ it follows that $\hat{f} = \hat{h}$ so that $f=h \in \dom{B}$.
\end{proof}

\subsection{Alternate formulas for the Livsic characteristic function}
\label{alternate}

In this subsection we pause to compute an alternate formula for the Livsic characteristic function. This will be useful, in particular, for computing formulas for the reproducing kernel of $\H _\Ga$ in the next
subsection.

Suppose that $B \in \sym{n}{\H}$ where $n< \infty$. As mentioned in the introduction the Livsic characteristic function of $B$ is usually defined using
$$ \{ u_k \} _{k=1} ^n  \quad  \mbox{orthonormal basis of} \  \ker{B^* -i} $$
$$ \{ v_k \} _{k=1} ^n  \quad  \mbox{orthonormal basis of} \  \ker{B^* +i} $$
$$ \{ w_k (z) \} _{k=1} ^n  \quad  \mbox{arbitrary basis of} \ \ker{B^* -z} $$
and
$$ A(z) := \left[ \ip{w_j (z)}{v_k} \right] \quad \quad \quad B(z) := \left[ \ip{w_j (z)}{u_k} \right] ,$$
by $$ \Theta _B (z) = b(z) B(z) ^{-1} A(z).$$

Here is an alternate formula that is sometimes useful. Let $A$ be a canonical self-adjoint extension of $B$ and let $$ w_j (z) := \Gamma _A ^i (\ov{z} ) e_j \in \ker{B^* -z}, $$ and
choose $v_j := \Gamma _A ^i (i) e_j = (A-i) (A+i) ^{-1} u_j$, where recall we choose $J_i : \C ^n \rightarrow \ker{B^* -i}$ so that $J_i e_j = u_j$, $\{ e_j \}$ is the standard orthonormal basis of $\C ^n$, and $J_i$ is an isometry.

Then it follows that
\ba \ip{w_j (z)}{v_k} & = & \ip{ (A-i)(A-z) ^{-1} u_j }{(A-i) (A+i) ^{-1} u_k } \nonumber \\
& =& \ip{ u_j}{ (A-i) (A-\ov{z}) ^{-1} u_k } = \ip{u_j}{ w_k (\ov{z})}.\ea
This shows that $$ A(z) = \left[ \ip{u_j}{ w_k (\ov{z} ) } \right], $$ and a similar calculation shows that
$$ B(z) = \left[ \ip{v_j}{w_k (\ov{z})} \right].$$ It follows that $A(\ov{z}) ^* = B(z)$.

\subsection{Reproducing Kernel formulas for $\H _\Gamma$}
\label{RKform}

Let $\Gamma$ be a generalized model for $B$ of rank $(m_+ , m_-)$ where at least one of $m_\pm$ is equal to $n$. Then by Proposition \ref{modelspace} we have an isometry $U_\Gamma :\H \rightarrow \H _\Gamma$ such that $$ U_\Gamma B = Z _\Gamma U_\Gamma, $$ so that $Z_\Ga$ is unitarily equivalent to $B$.

For any $w \in \C \sm \R$ let $P_w $ be the projection onto $\ran{B-w} = \ker{B^* -\ov{w} } ^\perp$, and let $Q_w := U_\Gamma P_w U_\Gamma ^*$, the projection onto $\ran{Z_\Gamma -w}$. Now define
$$ L_w := U_\Gamma b_{\ov{w}} (B) P _w  U_\Gamma ^* = b_{\ov{w} } (Z_\Gamma ) Q_w, $$
the partial isometric extension of $b_{\ov{w}} (Z_\Gamma ) $ to all of $\H _\Gamma$. It is clear that
$$ L_w = Q_{\ov{w}} b_{\ov{w}} (Z_\Gamma ) Q_w , $$ and that $L_w ^* = L _{\ov{w}}$.

We can now calculate formulas for the reproducing kernel of $\H _\Gamma$, using the same procedure as in \cite[Section 4]{AMR}. Let $k_w (z) = k_{w} ^\Gamma (z)  $ denote the reproducing kernel of $\H _\Gamma$.
Now given any $u, v \in \J$ and $\alpha \in \C \sm \R$,
\ba \ip{ (L_\alpha ^* k_w ) (z) u}{v} & = &\ip{L_\alpha ^*  k_w u}{k_z v} \nonumber \\
& = & \ip{k_w u}{L_\alpha k_z v } = \ov{ \ip{b_{\ov{\alpha}} (Z _\Gamma ) Q_\alpha k_z v}{k_w u}} \nonumber \\
& =& \ov{b_{\ov{\alpha}} (w)} \ip{k_w u}{Q_\alpha k_z v} \nonumber \\
& = & \frac{1}{\ov{b_\alpha (w)}} \left( \ip{k_w (z) u}{v} - \ip{ \left( (1 - Q_\alpha) k_w \right) (z) u }{v} \right). \ea

But also,
\ba \ip{ (L _\alpha ^* k_w ) (z) u}{v} & = & \ip{(L_{\ov{\alpha}} k_w ) (z) u}{v} = b_\alpha (z) \ip{Q_{\ov{\alpha}} k_w u}{k_z v} \nonumber \\
& =& b_\alpha (z) \left( \ip{k_w (z) u}{v} - \ip{ \left( (1-Q_{\ov{\alpha}}) k_w \right) (z) u }{v} \right).\ea

Solving for $\ip{k_w (z) u}{v}$ and using that $u,v \in \J$ were arbitrary yields:

\be k_w ^\Gamma (z) = \frac{\left( (1-Q_\alpha) k_w \right) (z) - b_\alpha (z) \ov{b_\alpha (w)} \left( (1-Q_{\ov{\alpha}} ) k_w \right) (z)}{1-b_\alpha (z) \ov{b_\alpha (w)}}, \label{arbkern} \ee
for any $\alpha \in \C \sm \R$.

Now suppose $\Gamma$ is a rank $(m_-, n)$ quasi-model for $B$ and that $Z_\Ga $ is unitarily equivalent to $B$.  This happens for example if
$\Ga = \Ga _A$ for some $A \in \Ext{B}$. Also choose $\J := \C ^n$, and $J: \C ^n \rightarrow \ker{B^* +i}$ to be an isometry, and $\alpha = i$ in equation (\ref{arbkern}).

Let $\{ u_k \}$ be an orthonormal basis for $\ker{B^* -i}$ such that $\{ u_k \} _{k=1} ^{n_+}$ is a basis for $\ran{\Ga (-i)}$, and let $\{ v_k \} _{k=1} ^{n}$ be an orthonormal basis for $\ker{B^*+i}$ such that $v_k = J e_k$, and $\{ e_k \}$ is an orthonormal basis of $\C ^n$.  We assume here that $J$ is an isometry.

Using that $1-Q_{-i} = \sum _{l=1} ^{n} \ip{\cdot}{\hat{u} _l} \hat{u} _l $ we can compute:

\ba \ip{ \left( (1-Q_{-i}) k_w \right) (z) e_j }{e_k} & = & \sum _{l=1} ^{n_+} \ip{k_w e_j}{\hat{u} _l} \ip{\hat{u} _l}{k_z e_k } _\Gamma \nonumber \\
& = & \sum _{l=1} ^{n_+} \ip{ U_\Gamma \Gamma (w) e_j }{U_\Gamma u_l } \ip{ U_\Gamma u _l }{U_\Gamma \Gamma (z) e_k } \nonumber \\
& =& \sum _{l=1} ^{n_+} \ip{\Gamma (w) e_j}{u_l}\ip{u_l}{\Gamma (z) e_k}. \nonumber \ea

Let $w_j (\ov{w}) := \Gamma (w) e_j \in \ker{B^* -\ov{w} }$, and define the $n\times n$ matrix
\be \alpha (z) := \left[ \ip{u_j}{w_k (\ov{z} )} \right]. \ee

Hence the above can be written:
$$ \ip{ \left( (1-Q_{-i}) k_w \right) (z) e_j }{e_k} =  \sum _{l=1} ^{n} \ip{w_j (\ov{w}) }{u_l}\ip{u_l}{w_k (\ov{z})}.$$

Compare this to

\ba \left( \alpha (z)  \alpha (w) ^* e_j , e_k \right)  & = & \sum _{l=1} ^{n} \left( e_j , \alpha (w) e_l \right) \left( \alpha (z) e_l , e_k \right) \nonumber \\
& =& \sum _{l=1} ^{n} \ip{ w_j (\ov{w}) }{u_l} \ip{u_l}{w_k (\ov{z} )}. \nonumber \ea
This proves that \be \left( (1- Q _{-i} ) k_w (z) \right) = \alpha (z) \alpha (w) ^*. \ee

A similar calculation shows that since $(1- Q _i )  = \sum _{l=1} ^{n} \ip{\cdot}{\hat{v_l}} \hat{v} _l $ we get that
$$ \ip{\left( (1-Q_i) k_w \right) (z) e_j}{e_k} = \sum _{l=1} ^n \ip{w_j (\ov{w})}{v_l}{v_l}{w_k (\ov{z})}.$$
If we take $\beta (z) := \left[\ip{v_l}{w_k (\ov{z})} \right]$
then as before it is not hard to check that
$$ \left( \beta (z) \beta (w) ^* e_j , e_k \right) = \sum _{l=1} ^{n} \ip{w_j (\ov{w}) }{v_l} \ip{v_l}{w_k (\ov{z} )} ,$$
which shows that
\be \left( (1-Q_i ) k_w \right) (z) = \beta (z) \beta (w) ^*. \ee

It follows that our formula for the reproducing kernel in $\H _\Gamma$ can be written:

\be k_w ^\Ga (z)  = \frac{ \beta (z) \beta (w) ^* - b(z) \ov{b(w)} \alpha (z)  \alpha (w) ^* }{ 1- b(z) \ov{b(w)}} \label{modelkern}. \ee

Now in the case where both $z,w \in \Pi _\Ga $ (in particular for $\Gamma _A$ we have that $\Pi _A ^+ $ is dense in $\C _+$)  we have that $\{ w_j (\ov{z} ) \}$ and $\{w_j (\ov{w}) \}$ are bases for $\ker{B^* -\ov{z} }$ and $\ker{B^* -\ov{w} }$ respectively, so that for such $z, w$ we have $\beta = B$ and $\alpha =A$, where $A, B$ are the matrices in the definition of $\Theta _B$ (see Subsection \ref{alternate}),
$$ \Theta _B (z) = b(z) B(z) ^{-1} A(z).$$

Hence for any $z,w \in \Pi _\Ga$
\be k_w ^\Gamma (z) = B(z) \left( \frac{\mathbb{1} -\Theta _B (z) \Theta _B (w) ^* }{1-b(z) \ov{b(w)}} \right) B(w) ^*. \label{rkgamma} \ee
Also observe that by the formula (\ref{modelkern}) we have that
$$ k_i ^\Gamma (z) = B(z)  B(i) ^* = B(z)  ,$$ since $B(i) = [ \ip{ J e_l }{ \Ga (i) e_k } ] = [ \ip{ J e_l }{J e_k } ] = \bm{1} $. Also we have that
$k_{i} (z) = \Gamma ^* (z) \Gamma (i)  = \Gamma (z )^* J $ so that $B(z) = \Ga (z) ^* J$.

\section{Cyclicity}

\label{cyclicity}

The goal of this section is to show that the characteristic function $\Theta _ B$ of $B$ is inner implies that $\ker{B^* -w}$ is cyclic for any
$A \in \Ext{B}$. This will enable us, in the subsequent section, to extend the isometry $U_A : \H \rightarrow \H _A$ to an isometry $V_A : \K \rightarrow \K _A$ where $A$ is self-adjoint in $\K$, $\K _A \supset \H _A$ is a larger reproducing kernel Hilbert space on $\C \sm \R $ containing $\H _A$,
and $V_A | _{\H _A } = U_A$. This larger space $\K _A$ contains information about the extension $A \in \Ext{B}$ that will be key for our characterization of $\Ext{B}$.

Here we say that a subspace $S \subset \H$ is cyclic for $A \in \Ext{B}$ if
$$\bigvee \mr{vN} ( A) S = \H, $$ where $\mr{vN} (A) $ is the von Neumann algebra generated by the unitary operator $b(A)$.

It will be convenient to apply some of the dilation theory for contractions as developed in \cite{NF}. The tools we are going to use are described below:

Given a contraction $T \in B (\H )$, recall that the defect
indices of $T$ are defined to be the pair of positive integers $(\mf{d} _T , \mf{d} _{T^*} )$ where $$ \mf{d} _T := \mr{dim} \left( \ov{\ran{ 1 - T^* T} }  \right). $$
Namely $\mf{d} _T := \mr{dim} \left( \mf{D} _T \right)$ where $$ \mf{D}_T := \ran{D_T = \sqrt{1-T^*T}}.$$
Given $B \in \sym{1}{\H}$, we will be studying the contraction $$ V:= b_w (B) Q_w $$ where $Q_w$ is the projection onto $\ran{B -\ov{w}} = \ker{B^* -w} ^\perp$ for some fixed $w \in \C \sm \R$ and $b _w(B)$ is the $w-$Cayley
transform of $B$, $$b _w (z) = \frac{z-w}{z-\ov{w}}.$$ This is a partial isometry, and it is clear that the defect indices of $V$ are equal to the deficiency indices of $b _w (B)$, namely
$(n,n)$. A contraction is called c.n.u. (completely non-unitary) if it has no non-trivial unitary restriction.  It is clear that since $B$ is simple, this implies that $V$
is c.n.u.  The model theory of Nagy-Foias \cite{NF} associates a contractive operator-valued function $\Theta _T$  called the Nagy-Foias characteristic function of $T$, to any
c.n.u. contraction $T$.  This function is defined by $$ \Theta _T (z) := (-T + zD_{T^*} (1-zT^*) ^{-1} D_T ) | _{\mf{D} _T}.$$ In our case where $T=V$ is a partial isometry, this expression
simplifies to: $$ \Theta _V (z) = z P_- (1-zV^*) ^{-1} P_+ ,$$ where $P_+, P_-$ are the projectors onto $\mf{D} _V = \ker{B^*-w}$ and $\mf{D} _{V^*} = \ker{B^* -\ov{w}}$ respectively. Since
$V$ is a partial isometry, in the case where $w=i$, the Nagy-Foias characteristic function $\Theta _V$ of $V = b _i (B) Q_w$ coincides with the Livsic characteristic
function, $$ \theta _{b _i (B)} := \Theta _B \circ b _i ^{-1}, $$ of the isometric linear transformation $b _i (B)$, as shown for example in \cite[Section 6]{Martin-uni}.

Now recall that any contraction $T$ acting on $\H$ has a minimal unitary dilation $U$ acting on some larger Hilbert space $\K \supset \H$. Recall that a unitary $U$ on $\K \supset \H$
is called a unitary dilation of $T$ if for any $n \in \mathbb{N} \cup \{0 \}$ we have that $$ T^n = P_\H U^n | _\H.$$ Such a dilation is called minimal if $\K$
is the smallest reducing subspace for $U$ containing $\H$, and the minimal unitary dilation of $T$ is unique up to a unitary transformation that fixes the Hilbert space $\H$
\cite[Theorem 4.3]{Paulsen}. Since our contraction $V = b _w (B) Q_w$ is c.n.u. (because $B$ is simple), it follows from \cite[II.6.4]{NF}, that the spectral measure
of the minimal unitary dilation $U$ of $V$ is equivalent to Lebesgue measure.  This just means that any of the positive Borel measures defined by $ \sigma (\Om ) = \ip{\chi _\Om (U) f}{f}$, where $\Om \subset \T$ is
a Borel set, are equivalent (have the same sets of measure zero) to Lebesgue measure. Here $\chi _\Om $ denotes the characteristic function of the Borel set $\Om$,
and $\chi _\Om (U)$ is a projection by the functional calculus for unitary operators.  It follows that $U$ has no eigenvalues so that $b_w ^{-1} (U)$ is a densely defined non-canonical self-adjoint extension of $B$.
Moreover the fact that $U$ is minimal implies that $\mc{K} \ominus \H$ contains no non-trivial reducing subspace $S$  for $U$ (since otherwise $U | _{\K \ominus S}$ would be the minimal
unitary dilation of $V$). Hence $$b_w ^{-1} (U) \in \Ext{B}.$$

Now as in \cite[II.2]{NF} set $$ \mc{L} := \ov{(U-V ) \H} = \bigvee U \ker{B^* -w}, $$ and let
$$ \mc{R} _* := \K \ominus \left( \ov{ \bigoplus _{n \in \mathbb{Z} } U ^n \mc{L} } \right) = \K \ominus \left( \ov{ \bigoplus _{n \in \mathbb{Z} } U ^n \ker{B^* -w} }\right).$$
Now by \cite[Proposition 2.1 VI.2]{NF}, the Nagy-Foias characteristic function $\Theta _V$ is a unitary if and only if both $V^n \rightarrow 0$ and $(V^*) ^k \rightarrow 0$ in the strong operator topology. Note that if $n< \infty$ then since $V$ has equal defect indices, the only way $\Theta _V$ can be an isometry is if it is in fact unitary, \emph{i.e.} inner. Since $\Theta _V$ coincides with the Livsic characteristic function of $b_w (B)$ (this is a consequence of the fact that $V$ is a partial isometry, as discussed above), we conclude that $\Theta _B$ is inner if and only if  both $V^k \rightarrow 0$ strongly and $(V^*) ^k \rightarrow 0$ strongly. In the notation of \cite{NF}, if $\Theta _B$ is inner so that $V^k \rightarrow 0$ and $(V^* )^k \rightarrow 0$ strongly, and $V$ has defect indices $(n,n)$, $V$ is called a contraction of class $C_0 (n)$.

By \cite[II.3.1]{NF} the projection $P _*$ onto $\mc{R} _*$ can be calculated by the formula:
\be P _* h = \lim _{n \rightarrow \infty} U ^{-n} V^n h \label{residual}.\ee Note that $\ker{B^* -w}$ is cyclic for $b_w ^{-1} (U) \in \Ext{B}$ if and only if $P _* = 0$.

\begin{thm}
    Suppose $B \in \sym{n}{H}$, $Q_w$ is the projection onto $\ran{B-\ov{w}}$, and $P_w = 1 - Q_w$ projects onto $\ker{B^* -w}$. Let $V = b_w (B) Q_w$, $U$ the minimal unitary dilation
of $V$, and $A = b_w ^{-1} (U) \in \Ext{B}$. Then for any $h \in \H$,
\be (1 - P_* ) h = \sum _{j=0} ^\infty U^{-j} P_w U^{j} h = \sum _{j=0} ^\infty U^{-j} P_w V^j h.   \label{ucyclic} \ee

\label{thm:ucyclic}
\end{thm}

\begin{lemma}
Let $B \in \sym{n}{\H}$, $A \in \Ext{B}$. Given any $h \in \H$ then for any $k \in \N$:

\be h = \sum _{j=0} ^k b _w ^\dag (A) ^j P_w (b_w (B) Q_w )^j h +b_w ^\dag (A) ^{k+1} (b_w (B) Q_w ) ^{k+1} h  . \label{partway} \ee

\label{expansion}

\end{lemma}

In the above recall that $b_w (z) := \frac{z-w}{z-\ov{w}}$ and $b_w ^\dag (z) = \ov{b_w (\ov{z})} = \frac{z-\ov{w} }{z-w}$.

\begin{proof}

    This clearly holds if $h  \in \ran{B-\ov{w}} ^\perp$. If $h \perp \ker{B^* -w}$ then
$$ h = Q_w h = b_w ^\dag (B) h_1, $$ for some $h_1 \in \ran{B-w} = \dom{b_w ^\dag (B)}$. Now
$$ h_1 = Q_w h_1 + P_w h_1, $$ and if we define $h_2 := Q_w h_1$ then $h_2 = b_w ^\dag (B) h_3$ for some
$h_3 \in \ran{B-w}$.
Now $$ h_3 = b_w (B) h_2 = b_w (B) Q_w h_1 = b_w (B) Q_w b_w (B) Q_w h = (b_w (B) Q_w )^2 h, $$ and
\ba h & =& b_w ^\dag (B) h_1 = b_w ^\dag (B) ( h_2 + P_w h_1) \\
& =& b_w ^\dag(B) (P_w h_1 + b_w ^\dag (B) h_3 ) \\
& =& b_w ^\dag (A) (P_w h_1 + b_w ^\dag (A) h_3 ) \\
& =& b_w ^\dag (A) P_w b_w (B) Q_w h + b_w ^\dag (A) ^2 (b_w (B) Q_w )^2 h .\ea

Repeating this process $k$ times yields $h_{2k+1} = (b_w (B) Q_w ) ^{k+1} $ and one obtains the formula stated above, namely,

$$ h = \sum _{j=0} ^k b_w ^\dag (A) ^j P_w (b_w (B) Q_w )^j h +b_w ^\dag (A) ^{k+1} (b_w (B) Q_w ) ^{k+1} h. $$

\end{proof}

\begin{proof}{ (Theorem \ref{thm:ucyclic})}
In the case where $A = b_w ^{-1} (U)$ where $U$ is the minimal unitary dilation of $V = b_w (B) Q_w$, the formula (\ref{partway}) becomes:
\ba h & = &   \sum _{j=0} ^k  U^{-j} P_w V^j h +U ^{-(k+1)} V ^{k+1} h \nonumber \\
& = &   \sum _{j=0} ^k U^{-j} P_w U^j h +U ^{-(k+1)} V ^{k+1} h \ea
Now we use the formula (\ref{residual}) of Nagy-Foias to conclude that if $A = b_w ^{-1} (U)$ where $U$ is the minimal unitary dilation of $V$, that $U^{-(k+1)} V^{k+1} h \rightarrow P _* h$,
proving the formula (\ref{ucyclic}) and the theorem.
\end{proof}

\begin{cor}
    Suppose that $B \in \sym{n}{\H}$. Then $\ker{B^* -w}$ and $\ker{B^* -\ov{w}}$ are cyclic for every $A \in \Ext{B}$ if and only if $\Theta _B$ is inner. If $\Theta _B$ is inner then the formulas
\be h =  \sum _{j=0} ^\infty b_w ^\dag (A) ^j P_{w} (b_w (B) Q_{w} )^j h,  \label{ucyclic2} \ee hold for any $A \in \Ext{B}$ and $h \in \H$. \label{innercyclic}
\end{cor}

\begin{remark}
\label{Acyclic}
    Let $\{ w_j \}$ be some fixed orthonormal basis of $\ker{B^* -w}$, and let $J_w : \C^n \rightarrow \ker{B^* -w}$ be defined by $J_w e_k =w_k$,
where $\{e_k \}$ is an orthonormal basis of $\C ^n$. Consider $L^2 _\Sigma$ where $\Sigma$ is the $\C^{n\times n}$ matrix-valued positive Borel measure defined by $\Sigma (\Om ) = J_w^* P_w P_A (\Om ) P_w J_w$, $P_A (\Om ) := \chi _\Om (A)$, where $\chi _\Om $ is the characteristic function of the Borel set $\Om$.
The above corollary shows in particular that for any fixed $h \in \H$ there is a vector function $\vec{f} = (f_1 , ... , f_n) \in L^2 _\Sigma$ such that
$$h = f_1 (A) w_1 + f_2 (A) w_2 + ... + f_n (A) w_n,$$ and that, remarkably, this equation holds independently of the choice of $A \in \Ext{B}$, \emph{i.e.}
the same $\vec{f}$ works for all $A \in \Ext{B}$ when $h $ is held fixed. \label{uniquecyclic} Although we will not pursue this in this paper, this fact can be used to provide a new proof, and potentially a slight extension of the Alexandrov isometric measure theorem, \cite[Theorem 2]{Alexiso}.
\end{remark}

\begin{proof}
By Lemma \ref{expansion}, given any $A \in \Ext{B}$ and $h \in \H$,
$$ h = \sum _{j=0} ^k b_w ^\dag (A) ^j P_w (b_w (B) Q_w )^j h +b_w ^\dag (A) ^{k+1} (b_w (B) Q_w ) ^{k+1} h. $$   Hence to prove the formula (\ref{ucyclic2}), it suffices to show
that $\| (b_w (B) Q_w) ^k  h \| = \| b_w ^\dag (A ) ^k (b_w (B) Q_w ) ^k ) h \| \rightarrow 0$.

If $\Theta _B$ is inner then $V ^n \rightarrow 0$ strongly (where recall $V = b_w (B) Q_w$) so that $$ 0 = \lim _{n \rightarrow \infty} \| V^n h \|, $$ and so the formula (\ref{ucyclic2}) holds.
If $A= b_w ^{-1} (U)$, then the fact that $\ker{B^* -w}$ is cyclic follows from the fact that $\mc{R} _* = \{ 0 \}$.
For arbitrary $A \in \Ext{B}$, the formula (\ref{ucyclic2}) shows that the cyclic subspace $S_w$ for any fixed $A \in \Ext{B}$ generated by $\ker{B^* -w}$ contains $\H$. Hence if $A$ is self-adjoint in $\K$, then $S_w = \K$, as otherwise $\K \ominus S_w$ would be a non-trivial subspace of $\K \ominus \H$ which is reducing for $A$ (this contradicts one of our assumptions on $\Ext{B}$). This proves that $\ker{B^* -w}$ is cyclic for any $A \in \Ext{B}$.

Conversely if $\ker{B^* -w}$ is cyclic for any $A \in \Ext{B}$, then it is cyclic for $b_w ^{-1} (U)$ where $U$ is the minimal unitary dilation of $V = b_w (B) Q_w $, and it follows from the definition of $R _*$ that $P_* = 0$,
and hence $T^n \rightarrow 0$ strongly. If $n <\infty$ this implies $T $ is a contraction of class $C_0 (n)$, implying that the characteristic function $\Theta _B$ of $B$ is inner as discussed previously.
If $n = \infty$ our assumption that $\ker{B^* -\ov{w}}$ is cyclic also implies that $(T^*) ^k \rightarrow 0$ strongly as well so that we get that $\Theta _B$ is inner.
\end{proof}

Note that the above proof also shows:

\begin{cor}
If $B \in \sym{n}{\H}$, $n < \infty$, and there is a $w \in \C \sm \R$ such that $\ker{B^*-w } $ is cyclic for every $A \in \Ext{B}$, then
$\Theta _B$ is inner.
\end{cor}

\section{A larger reproducing kernel Hilbert space $\K _A \supset \H _A$}

\label{larger}

\begin{defn} Given any $A \in \Ext{B}$, let \be \Omega _A (z) := U_{-i, z} J, \ee where recall that provided $A = b^{-1} (U)$ and
 $U$ does not have $1$ as an eigenvalue then $$ U_{-i , z} J = (A+i)(A-\ov{z}) ^{-1} J, $$ where recall that $J = J _{-i} = P_{-i} J_{-i}$ and $J : \C ^n \rightarrow \ker{B^* +i}$.
In the exceptional case where $A \in \Ext{B}$ is defined using a unitary extension $U$ of $b(B)$ and $1 \in \sigma _p (U)$, recall that $U_{w,z}$ is given by formula (\ref{except}).
We will assume in this section that $J$ is an isometry. Note that \be \Gamma _A (z) = P_\H \Om _A (z). \ee

We define a new reproducing kernel Hilbert space $K _A$ as the abstract $\C ^n$-valued reproducing kernel Hilbert
space on $\C \sm \R$ with reproducing kernel
$$ K _w  (z) := \Omega (z) ^* \Omega (w) $$
\end{defn}

The existence of $\mc{K} _A$ follows from the fact that $K_w (z)$ is a
positive kernel function, and the abstract theory of reproducing kernel Hilbert spaces \cite[Theorem 10.11]{Paulsen-rkhs}.

Observe that the difference
$$ K_w (z) - k_w (z) = \Om (z) ^* (\bm{1} - P_\H) \Om (w), $$
is a positive kernel function. The theory of reproducing kernel Hilbert spaces then implies that $\H _A $ is contractively contained in $\K _A$ \cite[Theorem 10.20]{Paulsen-rkhs}.

For $\vec{v} \in \C ^n$ the function $K_w \vec{v} $ defined by $$ K_w \vec{v} (z) := K_w (z) \vec{v}, $$ is a point evaluation vector in $\K _A$, \emph{i.e.}
$$ \ip{h}{K_w \vec{v}} _{\K _A} = \left( h(w), \vec{v} \right) _{C^n}, $$ for any $h \in \K _A$.

\begin{defn}
    Suppose that $\Theta _B$ is inner. Given $A \in \Ext{B}$ self-adjoint in $\K \supset \H$, recall that we define $U_A : \H  \rightarrow \H _A$ by $$ U_A (f) (z) = \hat{f} (z) = \Gamma (z) ^* f.$$
Now define a linear map $V_A : \K \rightarrow \K _A $ by $$ (V _A f) (z)  = \Omega _A (z) ^* f,$$ for $f \in \K$.    \label{cc}
\end{defn}

Note that if $g \in \H$ that
$$ (V_A g) (z) =  J^* (A-z) ^{-1} (A-i)  g = \Gamma ^* (z) g = (U_A g) (z), $$ so that for any $g \in \H$,
$$ U_A g (z) = V_A g (z).  $$ Hence if $E_A : \H _A \rightarrow \K _A$ is the contractive embedding then
\be V_A P_\H = E_A U_A. \ee

Also observe that if $\vec{u} \in \C ^n$, then
$$ K_w \vec{u} = V_A \Om (w) \vec{u}, $$ is the $\vec{u}$ point-evaluation vector in $\K _A$ at $w$.

\begin{prop}
    The linear map $V_A: \K \rightarrow \K _A$ is an isometry of $\K$ onto $\K _A$. Hence if $h \in \H$, then
    $$ \| V_A h  \| _{\K _A} = \|  h \| = \| U_A h \| _{\H _A} ,$$  so that $\H _A \subset \K _A$ isometrically and $U_A = V_A | _\H$.
\end{prop}

\begin{proof}
    Recall that since we assume that $B$ is such that $\Theta _B$ is inner, Corollary \ref{innercyclic} implies that $\ker{B^* +i}$ is cyclic for $A$.

    Since $\ker{B^* +i}$ is cyclic, $\K $ is spanned by vectors of the form $\Om (w) J \vec{v}$ for $w \in \C \sm \R$ and $\vec{v} \in \C ^n$. In particular
for any $\vec{v} \in \C ^n$, the vector $V_A \Om (w) \vec{v} \in \K _A$ since
$$ (V_A \Om (w) \vec{v} ) (z) := (\Om (z) ^* \Om (w)) \vec{v}  = K _w (z) \vec{v} ,$$ and $K_w \vec{v} \in \K_A$. The set of all point evaluation vectors $K_w \vec{v}$, $K_w \vec{v} (z) := K_w (z) \vec{v}$  for $w \in \C \sm \R$ are by definition dense in $\K _A$ so that this also proves $V_A$ is onto $\K _A$.

To see that $V_A$ is an isometry use that vectors of the form $f = \sum _j c_j \Om (w_j )  \vec{v} _j  $, for $\vec{v} _j \in \C ^n$ and $w_j \in \C$ are dense in $\K$, so that
\ba \ip{f}{f}  & = & \sum _{ij} c_i \ov{c_j} \ip{\Om (w_i) \vec{v} _i}{ \Om (w_j) \vec{v} _j } _\K  \\
& = &  \sum _{ij} c_i \ov{c_j} \left( K _{w_i} (w_j) \vec{v}_i , \vec{v} _j \right) _{\C ^n}  \\
& = & \sum _{ij}  c_i \ov{c_j} \ip{ K _{w_i} \vec{v} _i  }{K_{w_j}  \vec{v} _j} _{\K _A} \\
& = & \ip{V_A f}{V_A f} _{\K _A}. \ea

Now  if $h \in \H$, then $$ \| E_A U_A h \| _{\K _A } = \| V_A h \| _{\K _A} =  \| h \| = \| U_A h \| _{\H _A}. $$ Hence the contractive embedding $E_A : \H _A \rightarrow \K _A$
is actually an isometric inclusion, and $\H _A \subset \K _A$ as a Hilbert subspace.

\end{proof}

\subsection{Cauchy transforms and characteristic functions for $A \in \Ext{B}$}

For any $A \in \Ext{B}$, let $U:= b(A)$ be the corresponding unitary extension of $V:= b(B)$, and define $\sigma _U$ as the $\C ^{n\times n}$ matrix-valued measure on the unit circle $\T$ given by
$$ \sigma _U (\Om ) = \pi J^* P _U (\Om ) J, $$ where $P_U (\Om ) := \chi _\Om (U)$ is the projection-valued measure of $U$ defined using the functional calculus and recall that $J : \C ^n \rightarrow \ker{B^* +i} = \ker{V}$ is a fixed isometry.
We also define the $\C ^{n \times n}$ positive matrix-valued measure on $\R$, $\Sigma _A$ by
$$ \Sigma _A (\Om ) := \int _\Om \pi (1+t^2) J^* P_A (dt) J, $$ and note that if $\sigma _A (\Om ) := J^* P_A (\Om ) J$, then $\sigma _A = \sigma _U \circ b $, where $b(z) = \frac{z-i}{z+i}$ as before.

\begin{defn}
    If $A \in \Ext{B}$ with $A = b^{-1} (U)$, let $\Phi [A ;B]$ be the contractive analytic function on $\C _+$ corresponding to the pair $(\sigma _U (\{ 1 \}),  \Sigma _A)$ as described in Section \ref{Herglotz}.
When there is no chance of confusion we will suppress dependence on $B$ and use the simplified notation $\Phi _A$ for $\Phi [A ;B]$.  We call $\Phi [A;B]$ the characteristic function of $A$ relative to $B$, or simply the characteristic function of $A$ when it is clear which $B$ is used in the definition of $\Phi [A;B]$.
\end{defn}

In more detail, if $\phi := \phi [U ; V]$, then
$$ \re{ g_\phi (z) } = \int _\T \re{\frac{\alpha +z}{\alpha -z} } \sigma _U (d\alpha ), $$ where
$$ \phi = \frac{g_\phi -\bm{1} }{g_\phi + \bm{1} }$$ Equivalently if we impose the normalization condition discussed in Section \ref{Herglotz},
$$ g _\phi (z) = \int _\T \frac{\alpha +z}{\alpha -z}  \sigma _U (d\alpha ).$$  By the relationship between Herglotz functions on the disc and upper half-plane, as discussed in Section \ref{Herglotz}, we have that
$$ \re{G_{\Phi _A} (z) } = \sigma _U ( \{ 1 \} ) \im{z} + \intfty \re{ \frac{1}{i\pi} \frac{1}{t-z} } \Sigma _A (dt ), $$ or equivalently
\ba G _{\Phi _A } (z) & = & -iz \sigma _U ( \{ 1 \} ) + \frac{i}{\pi} \intfty \frac{tz +1}{t-z} \frac{1}{1+t^2} \Sigma _A (dt) \nonumber \\
& =& -iz \sigma _U ( \{ 1 \} ) +  \intfty \frac{tz +1}{i(t-z)}  (\sigma _U \circ b)  (dt) \nonumber \\
& = & -iz \sigma _U ( \{ 1 \} ) +  \intfty \frac{tz +1}{i(t-z)}  J^* P_A (dt) J. \nonumber \ea
In particular if $U$ does not have $1$ as an eigenvalue, then $\Phi _A$ is uniquely determined by $\Sigma _A$. Note that since $U$ is unitary, the projection-valued measure $P _U $ is unital which 
implies that $\sigma _U$ is a unital probability measure so that $g_\phi (0 ) = \bm{1}$, and this in turn implies that $\phi ( 0 ) = 0$, and that 
$$ \Phi [A ; B ] (i ) = 0, $$ for any $A \in \Ext{B}$. 

\begin{remark}
\label{WeylTit}
    Our definition of the characteristic function $\Phi[A;B]$ of the extension $A$ relative to $B$ is really an equivalent reformulation of the concepts
of the Weyl-Titchmarsh function and the Livsic characteristic function of the pair $(B, A)$ \cite{Don,MT1,MT2}.

Namely in \cite{Don}, Donoghue defines the Weyl-Titchmarsh function of a pair $(B,A)$, where $B$ is a densely defined simple symmetric operator with deficiency indices $(1,1)$ and $A \in \Ext{B}$ by the
formula \ba M (B, A ) (z) & := &  \ip{(Az +i )(A-zI ) ^{-1} g_+}{g_+} \nonumber \\
& =& \intfty \frac{tz+1}{t-z} \ip{P_A (dt) g_+}{g_+}, \nonumber \ea  where $g_+$ is a fixed normalized element in $\ker{B^* -i }$. In this case where $B$ has indices $(1,1)$,
we can define our isometry $J : \C \rightarrow \ker{B^* + i}$ in the construction of $\Ga _A$ and $\Om _A$ by $J e_1 = g_-$ where $e_1 = 1$ is a trivial orthonormal basis of $\C$ and $g_-$ is a fixed unit element of $\ker{B^* +i}$. In this case the Herglotz function $G_{\Phi _A}$ is just
\ba G _{\Phi _A } (z) & = &  -i  \intfty \frac{tz +1}{(t-z)}  J^* P_A (dt) J \nonumber \\
& =& -i \intfty \frac{tz +1}{(t-z)} \ip{P_A (dt) g_-}{g_-}. \nonumber \ea  This would be simply the Weyl-Titchmarsh function for the pair (B, A) multiplied by $-i$, if we had defined $\Phi [A;B]$ using
the deficiency subspace $\ker{B^* -i}$ instead of $\ker{B^* +i}$. Namely if we instead define $\check{\sigma}   (\Om ) = J_i ^* P_U (\Om ) J_i$, $\check{g} _U$ the corresponding Herglotz function on $\D$,
and $\check{G} _{A}$ the corresponding Herglotz function on $\C _+$, then $\check{G} _A = -i M(B,A)$.  Note here that since $B$ is densely defined
$U = b (A)$ does not have $1$ as an eigenvalue.

In \cite{MT1}, the Livsic function of the pair $(A,B)$, where $B$ as above has indices $(1,1)$ is defined to be
$$ s(B,A) (z) := \frac{M(z) -i}{ M(z) +i }.$$ Again if we had chosen to work with $\ker{B^* -i}$ instead of $\ker{B^* +i}$ then we would have
that $s(B,A) (z) = \check{\Phi} [A ; B] (z)$, where $\check{\Phi} [A ; B ] (z)$ is the contractive analytic function corresponding to the Herglotz function $\check{G} _A (z)$.

One can construct a natural bijective map between the sets of functions $\check{\Phi} [A ; B ]$ and the functions $\Phi [A ; B]$ where $\check{\Phi} [A ; B ] $ is defined using an isometry
$J_i : \C ^n \rightarrow \ker{B^* -i }$, and $\Phi [A ; B ]$ is defined using as isometry $J : \C ^n \rightarrow \ker{B^* +i}$ using the conjugation maps $C_B $ and $C _{B_T}$ described
in Section \ref{Herglotz}. Namely recall that if $B \in \sym{n}{\H}$ has Livsic characteristic function $\Theta _B$, then $B_T \in \sym{n}{\H _T}$ is a simple symmetric linear transformation with Livsic function $\Theta _B ^T$,
and there is a pair of anti-unitary maps $C_B : \H \rightarrow \H _T$ and $C_{B _T} : \H _T \rightarrow \H$ such that $C_B ^* = C_{B_T}$, $C _B \dom{B} = \dom{B_T}$, $C _{B_T} \dom{B_T} = \dom{B}$,
and $C_B B = B_T C_B$.

There is a bijective correspondence between unitary extensions $U$ of $b(B)$ and positive operator valued measures $Q_U$ on the unit-circle $\T$ which diagonalize $b(B)$, \emph{i.e.} such that for
any $f \in \ker{b(B)} ^\perp$, $$ b(B) f = \int _\T \alpha Q_U (d\alpha ).$$ Indeed if $U$ is a unitary extension of $b(B)$, then $Q _U (\Om ) := P_\H P _U (\Om ) P _\H$ is such a measure diagonalizing $b(B)$,
and conversely given such a measure $Q$, Naimark's dilation theorem provides a unitary extension $U$ on a larger Hilbert space $\K \supset \H$ with the property that  $Q (\Om ) := P_\H P _U (\Om ) P _\H$.
One can then check that the map $Q \mapsto \check{Q}$ defined by $$ \check{Q} (\Om ) := C_B Q (\Om ) C_{B_T}, $$ is a bijective map from the positive operator-valued measures diagonalizing $b(B)$ to those diagonalizing $b(B_T)$. Since
there is a bijection between such measures and extensions $A_T \in \Ext{ B_T}$, this constructs a bijection from extensions $A \in \Ext{B}$ to extensions $A_T \in \Ext{B _T}$.

Now let $\{ u _j \} , \{ v_j \}$ be orthonormal bases of $\ker{B^* -i }$ and $\ker{B^* +i }$ respectively, and let $\check{u} _j = C_B u_j$ and $\check{v} _j = C_B v_j$ be corresponding basis elements for
$\ker{B_T ^* \pm i }$, and suppose that $J_{\pm i} : \C ^n \rightarrow \ker{B^* \mp i }$ are isometries defined by $J_{-i} e_k = v_k$, $J_i e_k = u_k$, and define $\check{J} _{\pm i}$ similarly. Then
$C _B J_{\pm i} = \check{J} _{\pm i }$, and it follows that if
\ba M (A ,B ) (z) & := & \sigma _U (\{ 1 \} ) + \intfty \frac{tz + 1}{t-z} J_i ^* P_A (dt) J_i  \nonumber \\
& = & J_i ^* Q_U (\{ 1 \} )  J_i + \intfty \frac{tz + 1}{t-z} J_i ^*  Q_A (dt)  J_i, \nonumber \ea
where $Q_A (\Om ) := P_\H P_A (\Om ) P _\H$, then this is the suitable generalization of Donoghue's Weyl-Titchmarsh function to the case where $B$ has indices $(n,n)$ and is not necessarily densely defined. Moreover
\ba M (A ,B ) (z) & = &  \check{J} _{-i} ^* C_B Q_U (\{ 1 \} )  C_{B_T} \check{J} _{-i} + \intfty \frac{tz + 1}{t-z} \check{J} _{-i} ^* C_B Q_A (dt) C_{B_T} \check{J}_{-i} \nonumber \\
& =& \check{J} _{-i} P_{U_T} (\{ 1 \} ) \check{J} _{-i} + \intfty \frac{tz + 1}{t-z} \check{J} _{-i} ^*  P_{A_T} (dt)  C_{B_T} \check{J}_{-i} \nonumber \\
& =& i \Phi [B_T ; A_T ] (z), \ea so that $M (A, B) = i \Phi (B_T ; A_T ).$

This relationship between our characteristic function $\Phi [A;B]$ of the extension $A$ relative to $B$ and the Weyl-Titchmarsh function $M (B, A)$ of the pair $(B,A)$, allows one to translate all of our upcoming results
on $\Phi [A;B]$ and its relationship to $\Theta _B$ into equivalent statements about $M(B ;A )$.
\end{remark}

\begin{thm}
    If $\wt{\Phi} _A$ is the contractive analytic function with Herglotz function $\pi G_{\Phi _A}$, then $\K _A = \L (\wt{\Phi} _A )$.
\end{thm}

In particular if $U = b(A)$ does not have $1$ as an eigenvalue then $\K _A$ is the space of Cauchy transforms of the positive operator-valued measure $\pi \Sigma _A$.

\begin{proof}
    Let $\wt{\Phi} := \wt{\Phi} _A$.  It suffices to show that $K_w (z) = K_w ^{\wt{\Phi}} (z)$ where $K _w (z) = \Om (z ) ^* \Om (w)$ is the reproducing kernel for $\K _A$.
First $$ K_w (z) = \Om (z) ^* \Om (w) = J^* U_{-i, z} ^* U_{-i ,w} J,$$ where $U_{-i ,z}$ is given by equation (\ref{except}) so that
\ba  K_w (z) & = & 4 J^*  \left( (i+z) U + (i-z) \right) ^{-1} \left( (\ov{w} -i) U^* - (\ov{w} +i ) \right) ^{-1} J  \nonumber \\
& =&   \frac{ 4 \sigma _U ( \{ 1 \} )}{\left( (i+z) + (i-z) \right) \left( (\ov{w} -i) - (\ov{w} +i ) \right)} + \nonumber \\
& & + \frac{4}{\pi} \int _{\T \sm \{ 1 \} } \frac{1}{ (i+z) \alpha + (i-z) } \frac{1}{ (\ov{w} -i) \ov{\alpha} - (\ov{w} +i ) } \sigma _U (d\alpha ) \nonumber \\
& =&  \sigma _U ( \{ 1 \} ) + \frac{1}{\pi ^2 } \intfty \frac{1}{(t-z) (t-\ov{w} ) }  \pi \Sigma _A (dt) \nonumber \\
& =& K_w  ^{\wt{\Phi}} (z), \nonumber \ea where the last equality follows from equation (\ref{Hergkern}) and the definition of $\wt{\Phi} _A$.
\end{proof}

Now let us compute the Livsic characteristic function of the operator $\mf{Z} := \mf{Z} _{\wt{\Phi} _A} \in \sym{n}{\L (\wt{\Phi} _A ) }$ which acts as multiplication by
the independent variable in $\L (\wt{\Phi} _A) = \K _A$. We have
$$ \{ u_j := K_{-i} K_{-i} (-i) ^{-1/2} e_j \} \quad \quad \mbox{orthonormal basis of} \ \ker{\mf{Z} ^* -i} , $$
$$ \{ v_j := K_{i} K_{i} (i) ^{-1/2} e_j \} \quad \quad \mbox{orthonormal basis of} \ \ker{\mf{Z} ^* +i} , $$
$$ \{ w_j (z) = K_{\ov{z}} e_j \} \quad \quad \mbox{basis of} \  \ker{\mf{Z} ^* -z}. $$

Note that $K_i (i) = K_{-i} (-i) = J^* J = \bm{1}$. By Section \ref{alternate}, we can compute the Livsic characteristic function of $\mf{Z}$ in two ways:

We have that
$$ D (z) = \left[ \ip{w_j (z)}{u_k} \right] = \left[ \ip{ K_{\ov{z}} e_j }{ K_{-i} K_{-i} (-i) ^{-1/2} e_k } \right] =  K_{\ov{z}} (-i) , $$
and
$$ C (z) = \left[ \ip{w_j (z)}{v_k} \right] = K_{\ov{z}} (i). $$ Similarly
$$ \wt{D}(z) = \left[ \ip{v_j}{w_j (\ov{z})} \right] = K_i (z),$$ and
$$ \wt{C} (z) =  \left[ \ip{u_j}{w_j (\ov{z})} \right] =  K_{-i} (z). $$ Livsic's theorem implies that the functions

\be \La _A (z) := b(z) D(z) ^{-1} C(z) \quad \quad \mbox{and} \quad \quad \wt{\La} _A (z) := b(z) \wt{D} (z) ^{-1} \wt{C}(z) ,\ee

are both contractive and equal (modulo multiplication to the left and right by fixed unitaries) to the Livsic characteristic function $\Theta _{\mf{Z}}$ of $\mf{Z}$. Recall that the Livsic characteristic function is only defined up to
unitary coincidence, so this means that there are fixed unitary matrices $U,V$ such that
$$U \La _A V = \wt{\La} _A.$$ Explicitly we have
\be \wt{\La}  _A (z) =  b(z)  K _i (z) ^{-1} K_{-i} (z)  \quad \quad \mbox{and} \quad  \La  _A (z) = b(z) K_{\ov{z} } (-i) ^{-1} K _{\ov{z}} (i ). \ee

\begin{thm}
The contractive analytic functions $\La _A$ and $\wt{\La } _A$ are both equal to $\Phi _A = \Phi [B ; A ]$. \label{FMshift}
\end{thm}

\begin{proof}
Since $\La _A$ is the characteristic function of $\mf{Z} _{\wt{\Phi}  _A}$, and since $K_w ^{\wt{\Phi} _A} (z) = \pi K_w ^{\Phi _A } (z)$, it follows that
$$ \La _A (z) = b(z) K_{\ov{z}} ^{\Phi _A}  (-i) ^{-1} K_{\ov{z} } ^{\Phi _A } (i).$$ This shows that $\La _A$ is the Livsic characteristic function of $\mf{Z} _{\Phi _A}$, and Lemma \ref{FS} of Section \ref{Herglotz}
implies that $\La _A$ is the Frostman shift of $\Phi _A$ which vanishes at $i$. However since $\Phi _A (i ) = 0$, this Frostman shift is just equal to $\Phi _A$ and $\Phi _A = \La _A$. 

Now,
$\wt{\La} _A (z) =  b(z)  K _i (z) ^{-1} K_{-i} (z)$. Using that $G_A (\ov{z} ) ^* = -G_A (z)$, one can calculate that
\ba b(z) K_i (z) ^{-1} K_{-i} (z) & = & (G_A (z) +G_A(i) ^* ) ^{-1} (G_A(z) - G_A (i) ) \nonumber \\
&  = &  b(z) K_{\ov{z}} (-i) ^{-1} K _{\ov{z}} (i) = \La _A (z). \nonumber \ea
This shows that $\wt{\La} _A (z) =  \La _A (z) $.
\end{proof}

\begin{thm}
Given any $A \in \Ext{B}$, we have that $\Phi _A \geq  \Theta _B$, \emph{i.e.} $\Theta _B (z) ^{-1} \Phi _A (z)$ is a contractive analytic function in $\C _+$. \label{extchar}
\end{thm}

\begin{proof}
    Consider again the symmetric linear transformation $\mf{Z} $ which acts as multiplication by $z$ in $\L (\wt{\Phi} _A ) = \K _A$, where $\wt{\Phi}$ is the contractive analytic function
corresponding to the measure $\pi \Sigma _A$. We can construct a canonical model for $\mf{Z}$ by
choosing $\J := \C ^n$ with orthonormal basis $\{ e_j \}$ and defining $$ \Gamma (z) := K _z e_j , $$ where $K_z (w)$ is the reproducing kernel for $\K _A$.
If we do this we find that $ (\K _A ) _\Gamma = \K _A$ and that $U _\Gamma $ is just the identity on $\K _A$. Hence it follows from Section \ref{RKform}, and in fact from \cite{AMR},
that we can express the reproducing kernel for $\K _A$ as
$$ K_z (z) = \frac{ K_i (z) K_i (i) ^{-1} K_i (z) ^* - |b(z)|^2 K_{-i} (z) K_{-i} (-i) ^{-1} K_{-i} (z) ^* }{1-|b(z) | ^2 } .$$  We can write this as
$$(1- |b(z) | ^2)  K _z (z) =  \wt{D}(z) \wt{D} (z) ^* - |b (z) |^2 \wt{C} (z) \wt{C} (z) ^*, $$ where $\Phi _A (z) = \La _A (z) = \wt{\La} _A (z) = b(z )\wt{D} (z) ^{-1} \wt{C} (z)$, and $\wt{D} (z) = K _{i} (z) $.

Also note that if $k_w (z)$ is the reproducing kernel for $\H _A$, then for any $z \in \Pi _A ^+$ (which is dense in $\C _+$) we can write
$$ (1- |b(z) | ^2 ) k_z (z) = B(z) B(z) ^* - | b(z) |^2 A(z)  A(z)^*, $$ where $\Theta _B (z) = b(z) B(z) ^{-1} A(z)$ and $B(z) = k_{i} (z) = \Gamma (z) ^* \Gamma (i) = \Gamma (z) ^* J  = \Om (z) ^* \Om (i) = K _{i} (z)$. Hence $B(z) =  \wt{D} (z)$.

Now since $A \in \Ext{B}$, $\H_A $ is isometrically contained in $\K _A$ so that $  K_z (z) -  k_z (z) \geq 0$. Hence we have that

\be  \wt{D}(z) \wt{D} (z) ^* - |b(z)| ^2 \wt{C}(z) \wt{C}(z) ^*   \geq   \wt{D} (z)  \wt{D} (z) ^* - |b(z) |^2 A(z) A(z) ^*. \ee
so that
$$ A(z) A(z) ^* \geq  \wt{C} (z) \wt{C} (z) ^*,$$
and hence
$$ B^{-1} (z) A(z) A(z) ^* B^{-1} (z) ^* \geq \wt{D} (z) ^{-1} \wt{C} (z) \wt{C} (z) ^* (\wt{D} (z) ^*) ^{-1}.$$
Since $\wt{\La} _A (z) = \La _A (z) = \Phi _A (z) $, this shows that
\be \Theta _B (z) \Theta _B (z) ^* \geq \Phi _A (z) \Phi _A (z) ^*, \ee proving the theorem.

\end{proof}

\begin{eg}
\label{fdeg}
Consider the finite dimensional partial isometry $V$:
$$ V := \left( \begin{array}{cc} 0 & 0 \\ 1 & 0 \end{array} \right).$$ Clearly $V  \in \ms{V} _1  (\C ^2 )$.

Now let $$ U := \left( \begin{array}{ccc} 0 & 3/5 & 4/5 \\ 1 & 0 & 0 \\ 0 & 4/5 & -3/5  \end{array} \right).$$ This is a unitary
matrix acting on $\C ^3$, and $U| _{\ker{V} ^\perp} = V | _{\ker{V} ^\perp}$ so that $V \subseteq U$ and so if $B := b^{-1} (B) \in \sym{1}{\C ^2}$ then we
have that $A := b^{-1} (U) \in \Ext{B}$.

Our goal is to calculate $\Phi _A$ and to verify that $\Phi _A \geq \Theta _B$.

First we calculate $\Theta _B$, the characteristic function of $B = b^{-1} (V) = i (1+V)(1-V) ^{-1}$. We will denote the standard bases of $\C ^n$ by $\{ e_k \}$. Now
$$ \ker{B^* -i } = \ker{V} = \bigvee \{ e_2 \} \quad \mbox{and} \quad \ker{B^* +i} = \ran{V} ^\perp = \bigvee \{ e_1 \}. $$
Note that to avoid writing column vectors we will write $(a, b) ^T$ to denote the transpose of the row vector $(a , b)$,
and sometimes we will omit the $T$ in our calculations.

To calculate the Livsic characteristic function we also need to determine $\ker{B^* -z}$. First we calculate $\ran{B-z}$:
$$ \ran{B-z} = i(1+V)\ker{V} ^\perp -z (1-V ) \ker{V} ^\perp = \left( (i-z) + (i+z) V \right) \ker{V} ^\perp. $$
Since $\ker{V} ^\perp $ is spanned by $e_1$ and $V e_1 = e_2$, we get that $\ran{B-z}$ is spanned by
$$ \left( (i-z) , (i+z ) \right) ^T.$$ It follows that if $(c, d) ^T \in \ker{B^* -\ov{z} }$, that
$$ ( \ov{c} , \ov{d} ) \cdot (i-z , i+z ) = 0, $$ and this shows that $\ker{B^* -z}$ is spanned by
$$ w(z) :=  (z-i, z+i) ^T.$$

Finally $$ \Theta _B (z) = b(z) \frac{ \left( w(z), e_1 \right) }{\left( w(z) , e_2 \right)} = \left( \frac{z-i}{z+i} \right) ^2.$$

To calculate $\Phi _A$, we first need to calculate the projection-valued measure of $U$. We begin by calculating the eigenvalues and eigenvectors of $U$: We have
$$ \mbox{det}\left( \la - U \right) = \la ^3 + 3/5 \la ^2 -3/5 \la -1 = (\la - \la _1 ) (\la - \la _2 ) (\la - \la _3 ), $$
where $\la _1 =1$, $\la _2  := -4/5 +i 3/5 =: \beta $ and $\la _3 = \ov{\la _2} = \ov{\beta}$.  A normalized eigenvector for $\la _1 =1$ is:
$$\hat{b} _1 := ( 2/3 , 2/3, 1/3 ) ^T, $$ and (non-normalized) eigenvectors for $\beta , \ov{\beta}$ are:
$$ \vec{b} _2 = (1, \beta , 5/4 (\beta ^2 - 3/5) \beta ) ^T, $$ and
$$ \vec{b} _3 = (1, \ov{\beta} , 5/4 (\ov{\beta} ^2 - 3/5) \ov{\beta} ) ^T.$$
It follows that the projection-valued measure of $U$ is given by
$$ P_U =  \sum _{i=1} ^3 \left( \cdot , \hat{b} _i \right)\hat{b} _i \delta _{\la _i}, $$
where the $\delta _{\la _i}$ are Dirac point measures of weight one at the points $\la _i$, and the $\hat{b} _i$ are normalized eigenvectors to the eigenvalues $\la _i$.  Now the scalar measure $\sigma _U$ which determines
$\phi _U$, where $\Phi _A = \phi _U \circ b$, is given by
$$ \sigma _U (\Om ) = \ip{v}{P _U (\Om ) v}, $$ where $v = e_1$ is a unit vector spanning $\ran{V} ^\perp = \ker{B^* +i }$. Hence
$$  \sigma _U (\Om ) =  \sum _{k=1} ^3 | \left( e_1 , \hat{b} _k \right) | ^2 \delta _{\la _k}.$$
Now $\left( e_1 , \hat{b} _1 \right) = 2/3$, and since $\vec{b} _2 = C \vec{b_3}$ is the component-wise complex conjugate of $\vec{b} _3$, it follows that
$ | \left( e_1 , \hat{b} _2 \right) | ^2 =  | \left( e_1 , \hat{b} _3 \right) | ^2 =: a$. Finally since $P _U$ is unital, $\sigma _U$ must be a probability measure:
$$1 = \sum _{k=1} ^3  | \left( e_1 , \hat{b} _k \right) | ^2  = 4/9 + 2 a,$$ proving that $a = 5 /18$. In conclusion,
$$ \sigma _U = \frac{4}{9} \delta _1 + \frac{5}{18} \delta _\beta + \frac{5}{18} \delta _{\ov{\beta}}, $$ where $\beta = -4/5 +i 3/5$.
It follows that
$$ g _{\phi _U} (w) = \int _{\T} \frac{\alpha +w}{\alpha -w} \sigma _U (d\alpha ), $$
$$ G _{\Phi _A } (z) = -i \sigma _U ( \{ 1 \} ) z + \intfty \left( \frac{zt +1}{i (t-z)} \right) \wt{\sigma} _U (dt), $$
where $\wt{\sigma} _U := \sigma _U \circ b$.
An easy calculation shows that $b^{-1} (\beta ) = 1/3$ and $b^{-1} (\ov{\beta} ) = -1/3$, and so it follows that
$$ G _{\Phi _A } (z) = -i \frac{4}{9} z + \frac{5}{18} \frac{z/3 + 1}{ 1/3 -z} + \frac{5}{18} \frac{z/3 -1}{ 1/3 +z}. $$
Notice that $G _{\Phi _A } (i ) = 1$ as expected. Hence
\ba \Phi _A (z) & = & \frac{\left( \frac{4}{9} + \frac{5}{18} \frac{z+3}{1-3z} + \frac{5}{18} \frac{z-3}{1+3z} \right) -i }{\left( \frac{4}{9} + \frac{5}{18} \frac{z+3}{1-3z} + \frac{5}{18} \frac{z-3}{1+3z} \right) +i}. \nonumber \\
& =& \frac{ z (1-3z) (1+3z) +\frac{5}{8} \left( (z+3) (1+3z) + (z-3) (1-3z) \right) -i \frac{9}{4} (1-3z) (1+3z) }{ z (1-3z) (1+3z) +\frac{5}{8} \left( (z+3) (1+3z) + (z-3) (1-3z) \right) + i \frac{9}{4} (1-3z) (1+3z)}.
\nonumber \ea
The numerator simplifies to
$$ n(z) = -9z^3 + i \frac{81}{4} z^2 +\frac{27}{2} z -i\frac{9}{4}. $$ Let
$p(z) = \frac{n(z)}{-9} = z^3 -i\frac{9}{4} z^2 -\frac{3}{2}  z + \frac{i}{4}$. It follows that $\Phi _A (z)$ is the product of three Blaschke factors, one for each of the roots of $p(z)$.
It is easy to calculate that $p(z)$ factors as $p(z) = (z-i) ^2 (z-\frac{i}{4})$, and so (up to a unimodular constant),
$$ \Phi _A (z) =  \frac{(z-i) ^2 (z-i/4)}{(z+i)^2 (z+ i/4)}, $$ which is indeed greater or equal to
$$ \Theta _B (z) =  \left( \frac{z-i}{z+i} \right) ^2.$$

\end{eg}

\begin{defn}
    We say that $A_1  \sim A_2 $ if $\Phi _{A _1} = \Phi _{A _2}$. This is clearly an equivalence relation.
Let $\ext{B} := \Ext{B} / \sim $. That is $\ext{B}$ is the set of all $\sim$ equivalence classes of $\Ext{B}$.
\end{defn}

Suppose that $A_1, A_2 \in \Ext{B}$ are such that $A_k = b^{-1} (U_k )$ for $U_k \in \Ext{b(B)}$ which do not have $1$ as an eigenvalue. Then:
\begin{thm}
    $A_1  \sim A_2 $ if and only if $ A_1 \simeq A_2 $ via a unitary $U$ whose restriction to $\H$ is the identity. \label{equivalence}
\end{thm}

The above result is easily extended to include the exceptional case where one (or both) $A_1, A_2$ are defined using $U_1 ,U_2 \in \Ext{b(B)}$ where $1$ is an eigenvalue
of either $U_1$ or $U_2$. Namely the statement of the theorem becomes: Suppose $A _1, A_2 \in \Ext{B}$ are defined using $U_1, U_2 \in \Ext{b(B)}$. Then $A_1 \sim A_2$ if and only
if $U_1 \simeq U_2$ via a unitary $U$ which fixes $\H$.

\begin{proof}
    If such a unitary $U$ exists then
$$ \Sigma _1 (\Om ) := \Sigma _{A_1} (\Om) = \int _\Om \pi (1 +t^2) J^* P_1 (dt) J, $$
and
\ba J ^* P_1 (dt) J & = & J ^* U^*U P_1 (dt) J \nonumber \\
&  = & J ^* U^* P_2 (dt) U  J \nonumber \\
& = & J^*  P_2 (dt) J, \ea since $UJ = J$ as $U | _\H = \bm{1} _\H$.
It follows that $\Sigma _1 = \Sigma _2 $ which implies $\Phi _1 = \Phi _2$.

Conversely suppose that $\Phi _1 = \Phi _{A_1} = \Phi _{A_2} = \Phi _2$.
It follows then that $\Sigma _1 = \Sigma _2$ so that $$ J^* P_1 (\Om) J =  J^*  P_2 (\Om) J.$$ It follows that for any bounded Borel function $g$ on $\R$, $$ J^* g(A_1) J = J^* g(A_2) J.$$

Since $\ker{B^* +i} = J \C ^n$ is cyclic
for $ A_j $ (by Theorem \ref{innercyclic} since $\Theta _B$ is inner), for $j=1,2$, it follows that we can define a unitary $U : \K _1  \rightarrow \K _2$ as follows.
Let $\{ v_k = J e _k \} $ be an orthonormal basis of $P_{-i} \H$. Any $f \in \K _1$ can be written
$$ f = f_1 (A_1)  v_1 +...+ f_n (A_1)  v_n, $$ and define $$ Uf = f_1 (A_2)  v_1 +...+ f_n (A_2 )  v_n.$$ This is isometric because
\ba \ip{f_k (A_1 )  v_k}{f_j (A_1)  v_j} & = & \ip{ J^* \ov{f_j} (A_1) f_k (A_1) J e_k}{e_j} \nonumber \\
& = &  \ip{J ^* \ov{f_j} (A_2) f_k (A_2) J e_k}{e_j} \nonumber \\
& =&   \ip{U f_k (A_1 )  v_k}{ Uf_j (A_1)  v_j}.\ea The map $U$ is also onto because $\ker{B^* +i}$ is cyclic for $A_1$ and $A_2$.

\end{proof}

\begin{remark}
Suppose that $A \in \Extu{B}$, which is to say that there is an isometry $U: \H \rightarrow \K $ such that $A \in \Ext{UB U^*}$. In this case we define
$\Phi _A := \Phi [ A ; U B U^* ].$

Note that if $A \in \Extu{B}$ has characteristic function $\Phi _A$, then there is a corresponding $A' \in \Ext{B}$ such that $\Phi _{A'} = \Phi _A$. This follows from Naimark's dilation theorem \cite[Theorem 4.6]{Paulsen}.

Indeed if $A \in \Extu{B}$ so that $A \in \Ext{UBU^*}$ for some isometry $U:\H \rightarrow \K$ then $$ Q(\Om ) := P_\H U^* P_A (\Om) U P_\H, $$ is a positive operator-valued measure acting on $\H$.
Assume for now that $A$ is defined using a $W \in \Ext{Ub(B)U^*}$ which does not have $1$ as an eigenvalue so that $A = b^{-1} (W)$.

Since $A  = b^{-1} (W)$ and $1 \notin \sigma _p (W)$, $A$ is a densely defined self-adjoint operator and  $P_A (\R ) := \chi _\R (A) =  \bm{1} _\K$,
as otherwise $P_A (\R ) \K $ is a non-trivial reducing subspace for $A$ which contains $\H.$ In other words the projection-valued measure of $A$ is unital.

By Naimark's dilation theorem there is a larger Hilbert space $\K '  \supset \H$ and a unital projection-valued measure $P(\Om)$
acting on $\K ' $ such that the compression $$P_\H P(\Om) P _\H = Q(\Om), $$ for any Borel set $\Om$. This projection-valued measure $P$ is called a dilation of $Q$, and it can be chosen to be minimal in the sense that $\K ' = \bigvee P(\Om ) \H$. If $A'$ is the self-adjoint operator corresponding to this projection valued measure, $$ A' := \intfty t P(dt), $$ then it follows that $A' \in \Ext{B}$. It is also clear that by definition, $\Phi _A = \Phi _{A'}$.

If $A$ is defined using $W \in \Ext{Ub(B)U^*}$ with $1 \in \sigma _p (U)$, define $Q (\Om ) = P_\H P_U (\Om ) P_\H$, a unital positive operator-valued measure (POVM) on the unit circle.
Again apply Naimark's dilation theorem to obtain a unitary operator $U ' $ on $\K ' \supset \H$. As before it follows that if $A' \in \Ext{B}$ is defined using $U' \in \Ext{b(B)}$, that $\Phi _{A '} = \Phi _A$.

\label{unieqiv}
\end{remark}

\begin{thm}
The map $A \in \ext{B} \mapsto \Phi _A$ is a bijection onto the set of all contractive analytic functions $\Phi _A$ which are greater or equal to $\Theta _B$.

\label{bijection}
\end{thm}

This needs some setup: Given $B \in \sym{n}{\H}$ with characteristic function $\Theta _B$ let $V := b(B)(1-P_i)$, the partial isometric extension of $b(B)$, and define $\theta _V := \Theta _B \circ b^{-1}$,
a contractive analytic function on the unit disc, $\D$. Here, as before $P_i$ projects onto $\ker{B^* -i}$.

Recall that the Alexandrov-Clark measures for $\theta _V$ are defined as the $n\times n$ matrix-valued measures
$\delta _U$ for any $U \in \mc{U} (n)$ (the group of $n\times n$ unitary matrices) associated with the Herglotz functions $$ g _U := \frac{1+\theta _V U^*}{1-\theta _V U^*}, $$ via the Herglotz representation theorem for the unit disk \emph{i.e.}
$$ \re{ g _U (z)} = \int _\T \re{\frac{\alpha +z}{\alpha -z} } \delta _U (d\alpha) .$$ Let $G _U := g _U \circ b$ be the corresponding Herglotz function on $\C _+$. We define the Alexandrov-Clark measures of $\Theta _B$ to be the measures $\Delta _U$ on $\R$ such that
$$ \re{ G _U (z) } = \delta _U (\{ 1 \} ) \im{z} + \intfty \re{\frac{1}{i\pi} \frac{1}{t-z} } \Delta _U (dt ).$$ Recall that as discussed in Section \ref{Herglotz} (see equation (\ref{measrel})) we have that
$$ \Delta _U (\Om ) := \int _\Om \pi (1 +t ^2 ) (\delta _U \circ b) (dt), $$ where $(\delta _U \circ b ) (\Om ) = \delta _U (b(\Om ))$ and $b (z) = \frac{z-i}{z+i}$, $b: \R \rightarrow \T \sm \{ 1 \}$.

Now let $Z$ denote the unitary operator of multiplication by $z$
in $L^2 _{\theta } (\T)$ (the $L^2$ space of vector-valued functions on $\T $ which are square integrable with respect to the measure $\delta _{\bm{1}}$).

Let $\{b_j ^- (z) = e_j \}$ be a basis for the constant functions in $L^2 _\theta$. Since $\Theta (i) = 0 = \theta  (0)$, it follows that
this is an orthonormal basis. Similarly define $b_j ^+ (z) := \frac{1}{z} e_j$. For any $A \in \C ^{n\times n}$ let
$$ Z(A) := Z + P_- (\wt{A}- 1 ) P_- Z, $$ where $P_- $ projects onto the closed span of the $b_j ^-$, and $\wt{A} = j^* A j$ where $j$ is an isomorphism
defined by $j e_k = b_k ^-$ which takes $\C ^n$ onto the range of $P_-$.  Then as shown in \cite{Livsic2} $Z(0)$ has Livsic characteristic function $\theta _V$, and so it follows that there is a unitary
transformation $W : \H \rightarrow L^2 _\theta$ that implements the equivalences
$Z(0) \simeq V  = b(B)(1 - P_i )$, and $Z(U) \simeq V  (U)$ for any $U \in \mc{U} (n)$, and such that $W : \ker{B^* -i } = \ker{V } \rightarrow \ker{Z (0) } = \bigvee b_j ^- $
sends $u_j \mapsto b_j ^-$  \cite{Livsic2, Martin-uni}, where $\{ u_j \}$ is an orthonormal basis of $\ker{B^* -i }$.

Moreover the results of \cite{Martin-uni} show that
$$ \delta _U (\Om) = \left[ \ip{ \chi _\Om (Z (U) ) b_i ^- }{b_j ^-} \right]. $$
Using the fact that $G_U = g _U \circ b$, and the relationship between Herglotz functions and measures on the upper half-plane and the disk as described in Section \ref{Herglotz},
it follows that
$$ \re{G_U (w)} = \delta _U (\{ 1 \} ) \im{w} +  \intfty \re{\frac{1}{i\pi} \frac{1}{t-z} } \pi (1+t^2 ) \wt{\Delta} _U (dt), $$ where
$\wt{\Delta } _U := \delta _U \circ b$ so that
\ba \wt{\Delta} _U (\Om ) & = & \left[ \ip{\chi _{b(\Om)} ( Z_U ) b_i ^+ }{b_j ^+} \right] \nonumber \\
& =&  \left[ \ip{\chi _{b(\Om)} ( b( B (U) ) ) u_i }{ u_j } \right] \nonumber \\
& = & \left[ \ip{\chi _{\Om} ( B (U)  ) u_i }{ u_j } \right] ,\ea
and $$ \delta _U (\{ 1 \} ) = \left[ \ip{ \chi _{\{1 \} } (Z (U) ) b_i ^+ }{b_j ^+ } \right] = \left[ \ip{ \chi _{\{1 \} } (b (B(U)) )  u _i }{u _j } \right].$$

\begin{thm}
  \label{alexclark}   For any $U \in \mc{U} (n)$, $\Phi _{B (U)} =  U^* \Theta _B $.
\end{thm}

\begin{proof}
Let $B_T \in \sym{n}{\H _T}$ be a symmetric linear transformation with characteristic function $\Theta _B ^T$, and let $\{ \wt{u} _j \}$ , $\{ \wt{v} _j \}$ be orthonormal
bases of $\ker{B_T ^* -i}$ and $\ker{B_T ^* +i }$, respectively.

By Corollary \ref{conjugation}, there is a conjugation $C _T := C _{B_T} : \H _T \rightarrow \H$ which intertwines $B_T$ and $B$. Let $\{ u_k \} $, $\{ v_k \}$ be the orthonormal bases
of $\ker{B^* -i }$ and $\ker{B^* +i}$ respectively given by $C _T \wt{u} _j = v _j $ and $C _T \wt{v} _j = u_j$. Further recall that $C_T ^* = C _B$ is a conjugation intertwining $B$ and $B_T$ so that
$C _B v_j = \wt{u} _j$ and $C _B u_j = \wt{v} _j$.  Also define
$J :\C ^n \rightarrow \ker{B^* -i}$ by $J e_k = v_k$, for some orthonormal basis $\{ e_k \} $ of $\C ^n$.

Let $V$ and $V_T$ be the partial isometric extensions of the Cayley transforms of $B$, and $B_T$. Given any $U \in \mc{U} (n)$, let
$$ V(U) := V + \hat{U} := V + \sum _{i,j} U_{ij} \ip{\cdot}{u_i}v_j.$$ The set of all $V(U)$, for $U \in \mc{U} (n)$ is is the set of all canonical unitary extensions of $V$,
and the set of all $B(U):= b^{-1} (V(U))$ is the set of all canonical self-adjoint extensions of $B$.
Similarly define $$ V_T (U) := V_T + \wt{U} := V_T + \sum _{i,j} U_{ij} \ip{\cdot}{\wt{u} _i}\wt{v} _j.$$

Consider the self-adjoint extension $B _T (U) = b^{-1} (V_T (U))$, where $V_T$ is the partial isometric extension of $b(B_T)$. Then
$$ \dom{B_T (U)} = \ran{ 1 - V_T (U) } = \dom{B_T} + \left( 1 - \wt{U} \right) \wt{S}_{-i}, $$
where $\wt{S}_{\pm i} = \wt{P}_{ \pm i} \H$, and $\wt{P} _{\pm i}$ are the projections onto $\ker{B _T ^* \pm i}$. Similarly define $S _{\pm i }$ and $P _{\pm i}$.
As above, $\wt{U}$ is defined by $$ \wt{U} = \sum _{ij} U_{ij} \ip{\cdot}{\wt{u} _i} \wt{v} _j : \wt{S} _{i} \rightarrow \wt{S} _{-i}.$$

Given any $g \in \dom{B _T (U)}$, it follows that there is some $\wt{f} = \sum \ip{\wt{f}}{\wt{u} _i} \wt{u} _i \in \wt{S} _{i}$ and $g_T \in \dom{B _T}$ such that
\ba  g & = & g_T + \sum \ip{\wt{f}}{\wt{u} _i} \wt{u} _i - \sum  U_{ij} \ip{\wt{f}}{\wt{u} _i} \wt{v} _j \nonumber \\
& =& g_T + \wt{f} -\wt{U} \wt{f}, \nonumber \ea  so that
$$ B_T (U) g =  B _T g_T + i\wt{f} +i \wt{U} \wt{f}. $$
Now $C_T g_T = g_B \in \dom{B}$, and
\ba C_T g & = & g_B + \sum \ov{\ip{\wt{f}}{\wt{u} _i}} v_i - \sum \ov{U_{ij} } \ov{\ip{\wt{f} }{\wt{u} _i}} u_j  \nonumber \\
& =&  g_B + f - \hat{W} f, \ea where $C_T \wt{f} = f  := \sum \ov{\ip{\wt{f}}{\wt{u} _i}} v_i \in S_{-i}$ and
$$\hat{W} = \sum _{ij} \ov{U_{ij}} \ip{\cdot}{v_i} u_j.$$

Comparing this to
$$ \hat{U} ^* = \sum _{ij} \ov{U_{ji}} \ip{\cdot}{v_i}u_j, $$ we see that $\hat{W} = \widehat{U ^T} ^*$

Now if $R \in \mc{U} (n)$ then $b(B(R) ) ^* = V^* + \hat{R} ^*$ and if $b^\dagger (z) = \frac{z+i}{z-i}$ then its inverse with respect to composition is
$b^{-1} (z) ^\dagger = -i \frac{1+z}{1-z}$, so that we also have that $\dom{B(R)} = \ran{1 - V (R) ^* }$.
It follows that
$$ C_T g =  g_B + f -\widehat{U ^T} ^* f  \in \dom{B( U^T )}, $$ and
$$ B(U^T) C_T g =  B g_B -i f -i \widehat{U^T} ^* f, $$  while
$$ C_T B_T (U) g = C_T B _T g_T  + C _T (i \wt{f} + i \wt{U} \wt{f} ) = B(U^T) C_T g, $$   and this proves that
\be C _T B_T (U) = B(U^T ) C_T. \ee  It further follows that
$$ C_T V_T(U) = V(U ^T ) ^* C_T.$$

Now let $\delta _U$ be the Alexandrov-Clark measure associated with the Herglotz functions
$$g _U (z) := \frac{1 + \theta ^T U ^*}{1- \theta _T U^*}, $$ where $\theta ^T := \Theta ^T \circ b$, and as before let
let $\wt{\Delta } _U := \delta _U \circ b ^{-1} $. As discussed before this proof, the results of \cite{Martin-uni} show that
$$ \wt{\Delta } _U (\Om) = \left[ \ip{\chi _\Om (B_T (U) ) \wt{u} _i }{\wt{u} _j} \right], $$
so that
\ba \wt{\Delta } _U (\Om) & =& \left[ \ip{C_T \wt{u} _j }{ C_T \chi _\Om ( B_T (U) ) \wt{u} _i} \right] \nonumber \\
& = & \left[ \ip{u _j }{  \chi _\Om ( B (U^T) ) u _i} \right] \nonumber \\
& =& (J^* P_{B(U^T)} (\Om ) J ) ^T. \nonumber \ea

Similarly,
\ba \delta _U (\{ 1 \} ) & = & \left[ \ip{ C_B C_T \chi _{\{1 \} } (V _T(U) )  \wt{u} _i }{\wt{u} _j }  \right] \nonumber \\
& =&  \left[ \ip{ v_j }{ \chi _{ \{1 \} } (V (U ^T ) ^* ) v _i } \right] \nonumber \\
& =&  \left[ \ip{ v_j }{ \chi _{ \{1 \} } (V (U ^T ) ) v _i } \right] \nonumber \\
&= & (J ^* P _{V (U ^T)} (\{1 \} ) J) ^T. \ea

In conclusion we have that if $\Phi := \Phi _{B (U ^T )}$, that
$ G_\Phi ^T = G _U,$ so that
$$ G _\Phi  = G_U  ^T = \frac{1 + (U^*) ^T \Theta _B}{1 - (U^*) ^T \Theta _B }.$$ This proves that $\Phi _{B (U ^T)} = (U ^T) ^* \Theta _B$,
or equivalently that $\Phi _{B (U) } = U^* \Theta _B$.
\end{proof}

\begin{proof}{ (of Theorem \ref{bijection})}

This map is automatically injective by the definition of $\ext{B}$. To show that it is surjective, let $\Phi$ be a contractive analytic function such that $ \Phi \geq \Theta _B$, \emph{i.e.} $\Theta _B  ^{-1} \Phi $ is a contractive analytic function. Let $\Theta := \Theta _B$.

Now we have $B \simeq \mf{Z} _{\Theta }$ acting in $\L (\Theta )$, and by Corollary \ref{lessthan}, $\mf{Z} _{\Theta } \lessim \mf{Z} _\Phi$. 
Furthermore by Theorem \ref{alexclark} we have that there is a canonical self-adjoint extension $A$ of $\mf{Z} _\Phi$ whose characteristic function $\Phi _A = \Phi [A ; \mf{Z} _\Phi ]$ relative to $\mf{Z} _\Phi$ is $\Phi$.
Moreover one can see from Example \ref{motive} that the isometry $V : \L (\Theta) \rightarrow \L (\Phi )$ which obeys $V \mf{Z} _\Theta \subset \mf{Z} _\Phi V$ also satisfies $V P_{-i } = Q_{-i} V$ and $V^*V P_{-i} = P_{-i}$
where $P_{-i}$ projects onto $\ker{\mf{Z} _\Theta ^* +i}$ while $Q_{-i}$ projects onto $\ker{\mf{Z} _\Phi ^* +i}$. To see this note that since $\Theta (i) = 0 =\Phi (i)$ that
$$ K_i ^\Theta (z) = \frac{2}{1-\Theta (z)} \frac{i}{\pi} \frac{1}{z+i} \quad \mbox{and} \quad  K_i ^\Phi (z) = \frac{2}{1-\Phi (z)} \frac{i}{\pi} \frac{1}{z+i}. $$
Observe that $$ V_1 (z) = \frac{1 -\Theta (z)}{2},$$ is an isometry of $\L (\Theta )$ onto $K^2 _\Theta$, that $K^2 _\Theta$ is isometrically contained in $K^2 _\Phi$ (since $\Theta$ is inner), and that multiplication by
$$ V_2 (z) := \frac{2}{1 -\Phi (z) }, $$ is an isometry of $K^2 _\Phi $ into $\L (\Phi )$. Since $V$ acts as multiplication by $V(z) = V_2 (z) V_1 (z)$, it is an isometry that obeys $V K_i ^\Theta \vec{v} = K_i ^\Phi  \vec{v}$ for any $\vec{v} \in \C ^n$.

It follows that the isometry $V : \L (\Theta ) \rightarrow \L (\Phi )$ obeys $V \ker{ \mf{Z} _\Theta ^* +i } = \ker{\mf{Z} _\Phi ^* +i }$. This shows that the characteristic function $\Phi _A = \Phi _A [A; \mf{Z} _\Phi] = \Phi$ of $A$ with respect to $\mf{Z} _\Phi$ is the same as the characteristic function $\Phi _A$ of $A \in \Extu{B}$ with respect to $B$. By Remark \ref{unieqiv}, there is an $A' \in \Ext{B}$ with $\Phi _{A'} = \Phi _{A} = \Phi$.

Putting it all together we have that
$$ B \simeq \mf{Z} _{\Theta _B} \lessim \mf{Z} _\Phi , $$ so that $B \lessim \mf{Z} _\Phi$, $A' \in \Ext{B}$ and $\Phi _{A'} = \Phi \geq \Theta _B$.  This proves surjectivity.
\end{proof}

\section{Partial order calculations}

\label{pord}

In this section we study the partial order $\lessim$ on symmetric linear transformations described in the introduction:

\begin{defn}
    Given $B_1 , B_2 \in \symm$ we say that $B _1 \lessim B_2$ if $B_1 \simeq B_1 ' \subset B_2$. Recall here $\simeq$ denotes unitary equivalence.
\end{defn}

We assume in this section that $n < \infty$, and under this assumption, it is not difficult to verify that $\lessim$ is indeed a partial order on the unitary equivalence classes of $\symm$ (see \cite{MR-isomult}).  Also, using the Cayley transform, this also defines a partial order on $\isom$. Namely, given
$V_1 , V_2 \in \isom$, $V_1 \lessim V_2$ if and only if $V_1 \simeq V_1 ' \subseteq V_2$, where recall that $V_1 ' \subseteq V_2$ means that
$V_2 | _{\ker{V_1 ' } ^\perp} = V_1 '  | _{\ker{V_1 ' } ^\perp}$.  This is the same, modulo unitary equivalence as the partial order defined on partial isometries by Halmos and McLaughlin in \cite{Halmos}.  That is, they define $V_1 \leq V_2$ if $V_1 \subseteq V_2$.

The main goal of this section is, given $B _1 , B_2  \in \symm$ with $\Theta _1 := \Theta _{B _1}$ inner, to provide necessary and sufficient conditions on the characteristic function $\Theta _2 := \Theta _{B_2}$ of $B_2$ so that $B_1 \lessim B_2$.

Let $B_1 \in \sym{m}{\H _1}$ and $B_2 \in \sym{n}{\H _2 }$ be symmetric linear transformations, and suppose that $B_1 \lessim B_2$.  As always in this paper we assume that $\Theta _1$ is inner.

\begin{remark}
\label{densecon}
 Let $\Sigma _2$ be the Herglotz measure of $\Theta _2$. For now we assume that the Herglotz measure $\sigma _2$ of $\theta _2 := \Theta _2 \circ b^{-1}$ is such that $\sigma _2 (\{ 1 \} ) =0$. Recall from Section \ref{Herglotz} that this implies that
$W_{\Theta _2} : L^2 _{\Sigma _2} \rightarrow \L (\Theta _2 )$ is an onto isometry so that $B_2 \simeq M _\Sigma$. This is the case, in particular, when $B_2$ is densely defined.
\end{remark}

Let $\Sigma := \pi \Sigma _2 $ where $\Sigma _2$ is the Herglotz measure on $\R$ corresponding to $\Theta _2$. By Remark \ref{densecon}, we can and do assume that $B_2 = M _\Sigma$.y Recall here that $M_\Sigma$ is the symmetric operator of multiplication by the independent variable in $L^2 _\Sigma$, where $L^2 _\Sigma$ is the Hilbert space of column vector-valued functions $f$ which are square integrable with respect to $\Sigma$, \emph{i.e.} if $f,g \in L^2 _\Sigma $ then
$$\ip{f}{g} _\Sigma = \intfty \left( \Sigma (dt) f (t) , g(t) \right) _{\C ^n }.$$

Let $\{ e_k \}$ be the standard basis of $\C ^n$, and let $\{ v _k \} _{k=1} ^n$ be a fixed orthonormal basis of $\ker{B_2 ^* +i}$. Define an isomorphism $J : \C ^n \rightarrow \ker{B_2 ^* +i}$ by
$J e_k = v_k$. By our previous result, Theorem \ref{alexclark}, on Alexandrov-Clark measures, there is a canonical $\ms{A} \in \Ext{B_2}$ such that
$$ \Phi [\ms{A} ; B_2] = \Theta _2, $$ and
$$ \Sigma (\Om ) = \int _\Om \pi ^2 (1+t^2) J^* P_\ms{A} (dt) J. $$ Since we assume that $B_2 = M _\Sigma$ it actually follows that $\ms{A} = M^\Sigma$, the self-adjoint operator of multiplication by $t$ in $L^2 _\Sigma$.

\begin{remark}

We can further choose $$ v_k := \frac{i}{\pi} \frac{1}{t+i} e_k, $$ this follows because if $\wt{\Phi }$ is the contractive analytic function corresponding to $\Sigma$, then the deBranges-Cauchy transform
isometry, $$ W : L^2 _\Sigma \rightarrow \L (\wt{\Phi } ), $$ is onto and acts as $$ Wh (z) = \frac{1}{i\pi} \intfty  \ov{\frac{1}{\pi (t-\ov{z})}} \Sigma (dt) h(t), $$ and the $e_j$-point evaluation vector at $z=i$ in $\L (\wt{\Phi}  )$ is
$$ K_i (z) e_j = \intfty \frac{i}{\pi (t+i)} \ov{\frac{i}{\pi (t-\ov{z}) }} \Sigma (dt) e_j = W v_j (z).$$ Since $K_i  e_j$ spans $\ker{\mf{Z} _{\wt{\Phi}} ^* +i}$, the $W^* K_i e_j = v_j$ span $\ker{M _\Sigma ^* +i}$. Moreover this choice of $v_k$ defines an orthonormal basis since $\Theta _2 (i) = 0$ implies that
\ba \mathbb{1} & = & \re{B_{\Theta _2} (i) } = \intfty \re{\frac{1}{i\pi} \frac{1}{t-i} } \Sigma _2 (dt) \nonumber \\
& = & \frac{1}{\pi} \intfty \frac{1}{1+t^2} \Sigma _2 (dt) = \frac{1}{\pi ^2} \intfty \frac{1}{1+t^2} \Sigma (dt). \nonumber \ea  It follows from this formula that $\ip{v_k}{v_j} = \delta _{kj}$.
\end{remark}

We are also free to assume that $B_1 \subset B_2 = M _\Sigma $ so that $B_1 \in \sym{m}{S}$ where $S \subset L^2 _\Sigma$. Let $A$ be the restriction of $\ms{A} = M ^\Sigma$ to the intersection of its domain with its smallest reducing subspace containing $S$.
Then $A \in \Ext{B_1}$. Let $\{ \wt{v} _k \} _{k=1} ^m$ be an orthonormal basis of $\ker{B_1 ^* +i}$, and let $\wt{J} : \C ^m \rightarrow \ker{B_1 ^* +i}$ be an isomorphism defined by $\wt{J} e_k = \wt{v} _k$.
Then if we define
$$ \wt{\Sigma } ' (\Om) := \int _\Om \pi (1+t^2) \wt{J} ^* P_{A} (dt) \wt{J}, $$ then we have that
$\Phi  := \Phi [A ; B_1]$ is the contractive analytic function corresponding to $\wt{\Sigma} ' $ and we define $\wt{\Sigma} = \pi \wt{\Sigma } '$.
Now since $\{ v_k \}$ is a cyclic set for $M ^\Sigma =\ms{A}$, we have that
\be  \wt{v} _j = D_{j1} (\ms{A}) v_1 + ... + D_{jn} (\ms{A}) v_n, \quad 1 \leq j \leq m \label{dmatrix} \ee  for certain functions $D_{jk} $ where $1 \leq j \leq m$ and $1 \leq k \leq n$ and $\frac{1}{t+i}(D_{j1}, ... , D_{jn} ) ^T \in L^2 _\Sigma $ for $1\leq j \leq m$. Here the superscript $T$ denotes transpose (we view elements of $L^2 _\Sigma$ as column vector functions).

Now if $f, g \in L^2 _{\wt{\Sigma}}$, then it follows that
\ba \ip{f}{g} _{\wt{\Sigma}} & = & \intfty \left( \wt{\Sigma} (dt) f (t) , g(t) \right) _{\C ^n } \nonumber \\
& =& \intfty (\ov{g_1 (t)} , ... , \ov{g_n (t)} ) \left( \begin{array}{ccc} ( \wt{\Sigma} (dt) e_1 , e_1 ) & \cdots & (\wt{\Sigma} (dt) e_n , e_1 ) \\
\vdots & \ddots & \\
( \wt{\Sigma} (dt) e_1 , e_n ) & \cdots & (\wt{\Sigma} (dt) e_n ,e_n ) \end{array} \right) \left( \begin{array}{c} f_1 (t) \\ \vdots \\ f_n (t) \end{array} \right).\ea

Hence we have that $$ \wt{\Sigma} (\Om ) = \left(
\begin{array}{ccc}
\ip{P_{A} (\Om) \wt{v} _1}{ \wt{v} _1} & \cdots & \ip{P_A (\Om) \wt{v} _m}{ \wt{v} _1 } \\
\vdots & \ddots & \\
\ip{ P_A (\Om ) \wt{v} _1 }{ \wt{v} _m } & \cdots & \ip{P_A (\Om) \wt{v} _m}{\wt{v} _m }  \end{array} \right), $$ and using the relationship (\ref{dmatrix}) between the
$\wt{v_j}$ and the $v_k$ we get that
\be \wt{\Sigma } (dt ) = D(t) ^* \Sigma (dt) D(t ), \ee where
\be D(t) := \left( \begin{array}{ccc} D_{11} (t) & \cdots & D_{m1} (t)  \\
\vdots & \ddots & \\
D_{1n} (t) & & D_{mn} (t) \end{array} \right), \ee  $D(t) : \C ^m \rightarrow \C ^n$.

Now since $A \in \Ext{B_1}$, $B_1 \in \sym{m}{S}$, let $\K :=$ the cyclic subspace of $L^2 _\Sigma$ generated by $S$ and $A$. Let
$U : L^2 _{\wt{\Sigma}} \rightarrow L^2 _\Sigma $ be defined by multiplication by $D(t)$, $V_A : \K \rightarrow \K _A$ be the model space isometry,
and $W : L^2 _{\wt{\Sigma}} \rightarrow \K _A$ be the deBranges Cauchy transform isometry.

\begin{claim}
    The linear map $U : L^2 _{\wt{\Sigma}} \rightarrow L^2 _\Sigma $ is an isometry which obeys
    $$ V_A U = W, \quad \mbox{and} \quad U U^* = \bm{1} _\K.$$ \label{isoid}
\end{claim}

Since $U$ acts as multiplication by the matrix function $D(t)$, it is easy to see that $U$ will intertwine $M ^{\wt{\Sigma}}$ and $M ^\Sigma$. However we also want to verify
that $U$ takes the domain of $B \subset M _{\wt{\Sigma}}$ into the domain of $M _{\Sigma}$.  This claim will allow us to do this.

\begin{proof}
If $g \in L^2 _{\wt{\Sigma}}$ then the map $U : L^2 _{\wt{\Sigma}} \rightarrow L^2 _\Sigma $ defined by
$U g (t) = D(t) g(t)$ is clearly an isometry since
$$ \| U g \| ^2 = \intfty \left( D(t) ^* \Sigma (dt) D(t) g(t) , g(t) \right)  = \| g \| ^2.$$

Recall that $\K _A$ is the space of Cauchy transforms of the measure $\wt{\Sigma }$. The linear map $V_A $ is an isometry from $K \subset L^2 _\Sigma$ onto $\K _A$.
We can extend $V_A$ to a partial isometry acting on all of $L^2 _\Sigma$ by the formula
$$ V_A f (z) := \Om _A (z) ^* P_\K f = \wt{J}^* (A-i)(A-z) ^{-1} P _\K f = \wt{J}^* (\ms{A} -i) (\ms{A} -z) ^{-1} f, $$ for any $f \in L^2 _\Sigma$.
Then for any $f \in L^2 _\Sigma$, $V_A f (z) $ is a column vector with components
$$ (\Om _A  (z) ^* f  ) _j = \intfty \frac{t-i}{t-z}  ( \Sigma (dt) f(t), \wt{v} _j (t) ). $$
Since $\wt{v} _j (t) = D_{j1} (t) v_1 (t) + ... + D_{jn} (t) v_n (t), $ it follows that the above can be written as
\ba (\Om _A (z) ^* f) _j & = & \intfty \frac{t-i}{t-z}  \left( D(t) ^* \Sigma (dt) f (t) , v_j (t) \right) _{\C ^m} \nonumber \\
& = & \frac{1}{i\pi} \intfty \frac{1}{t-z}  \left( D (t) ^* \Sigma (dt) f (t) , e_j \right) _{\C ^m} , \ea so that
\be V_A f (z) = \Om _A (z) ^* f = \frac{1}{i\pi} \intfty \frac{1}{t-z} D (t) ^* \Sigma (dt) f (t). \ee

On the other hand the Cauchy transform isometry $W : L^2 _{\wt{\Sigma} } \rightarrow \K _A$ obeys
\be W g (z) = \frac{1}{i\pi} \intfty \frac{1}{t-z} \wt{\Sigma} (dt) f(t). \ee

Finally, observe that for any $g \in L^2 _{\wt{\Sigma}} $,
$$ V_A U g (z) = \frac{1}{i\pi} \intfty \frac{1}{t-z} D(t) ^* \Sigma (dt) D(t) g(t) = W g (z).$$
This proves that \be V_A U = W. \ee

Now $U: L^2 _{\wt{\Sigma}} \rightarrow L^2 _\Sigma$ is an isometry, $V_A : L^2 _\Sigma \rightarrow \K _A$ is a partial isometry with initial space $\K$ and
$W : L^2 _{\wt{\Sigma}} \rightarrow \K _A = \L (\wt{\Phi } _A)$ is an onto isometry. If $\ran{U}$ is not contained in $\ker{V_A } ^\perp$, then we could find an
$f \in L^2 _{\wt{\Sigma}}$ such that $Uf = g_\K + g _\perp$ with $g_\K \in \K$ and $g_\perp \neq 0$ in $L^2 _\Sigma \ominus \K$. But then it would follow that
$$ \| V_A U f \| = \| g _\K \| < \| f \|, $$ which would contradict the fact that $$\| V_A U f \| = \| W f \| = \| f \|.$$

Hence $$ U = V_A ^* V_A U =  V_A ^* W, $$ so that $$ U U ^* = V_A ^* W W^* V_A = V_A ^* \bm{1} _{\K _A} V_A = \bm{1} _\K. $$

\end{proof}

Now $W M_{\wt{\Sigma}} W^* = \mf{Z} _{\wt{\Phi} _A}$ acts as multiplication by $z$ in $\K _A = \L (\wt{\Phi } _A )$.
Also $V_A B_1 V_A ^*$ acts as multiplication by $z$ in $\K _A$ so that $V_A B_1  \subset \mf{Z} _{\wt{\Phi} _A} V_A $. This follows because
$V_A B_1  = U_A B_1  = Z_A U_A$, where $Z_A = U_A B_1 U_A ^*$ acts as multiplication by
$z$ in $\H _A \subset \K _A$.

 It follows that
$$ U^* B_1 = W^* V_A B_1 \subset W^* \mf{Z} _{\wt{\Phi} _A} V_A =  M _{\wt{\Sigma}} U^*. $$

Now if $f \in \dom{B_1} \subset \dom{M_\Sigma}$, then by the definition of the domain of $M _\Sigma$,
$$ \intfty \Sigma (dt) f(t) = 0.$$
We also have that $U^* f \in \dom{M_{\wt{\Sigma}}}$ so that
$$ 0 = \intfty \wt{\Sigma} (dt) D^{-1} (t) f(t)  = \intfty D^* (t) \Sigma (dt) f(t).$$

Alternatively if $B_1 ' =  M _{\wt{\Sigma}} | _{U^* \dom{B_1}} \simeq B_1$, then for any $g \in \dom{B_1 '}$ we have that $Ug \in \dom{M _\Sigma}$ so that
$$ \intfty \wt{\Sigma} (dt) g(t) = 0, $$ and
$$ 0 = \intfty \Sigma (dt) D(t) g(t) = 0. $$

In summary we have established the necessity half of:

\begin{thm}
Let $B_1 \in \sym{m}{\H _1}, B_2 \in \sym{n}{\H _2}$ with characteristic functions $\Theta _1$ and $\Theta _2$ (where we fix a choice of $\Theta _2$ to obey the condition of Remark (\ref{densecon})).
If $\Theta _1$ is inner then $B_1 \lessim B_2$ if and only if the following three conditions hold:
\bn
    \item There exists a contractive $\C ^{m \times m}-$valued analytic function $\Phi$ such that $\Phi \geq \Theta _1$.
    \item The Herglotz measure $\wt{\Sigma }$ of $\Phi$ is absolutely continuous with respect to the $\C^{n\times n}-$valued Herglotz measure $\Sigma$ of $\Theta _2$,
    $$ \wt{\Sigma} (dt) = D^* (t) \Sigma (dt) D(t), $$ for a $\C ^{m\times n}$ matrix-valued function $D(t)$ whose columns divided by $t+i$ belong to $L^2 _\Sigma$.
    \item Suppose that $A \in \Ext{B}$ is the extension such that $\Phi _A = \Phi$.  If $\wt{B} _1 := M _{\wt{\Sigma}} | _{W^* V_A \dom{B_1}}$ where
    $W : L^2 _{ \wt{\Sigma}} \rightarrow \K _A$ is the deBranges isometry, then for any $f \in \dom{\wt{B} _1}$ we have that
    $$ \intfty \wt{\Sigma} (dt) f(t) = 0 \quad \mbox{and} \quad \intfty \Sigma (dt) D(t) f(t)  = 0.$$
\en

\label{pochar}

\end{thm}

\begin{proof}

To prove the sufficiency half of the above theorem, suppose that the above three conditions are satisfied and choose $A \in \Ext{B_1}$ so that $\Phi _A = \Phi$ (such an $A$ exists by Theorem \ref{bijection}).

We know that $B_1$ is unitarily equivalent to a restriction of $M _{\wt{\Sigma}}$. Here are the details:
Let $\wt{\Sigma}$ be the matrix-valued measure which is $\pi$ times the Herglotz measure for $\Phi _A$.
Let $W$ be the Cauchy transform isometry which takes $L^2 _{\wt{\Sigma}} $ onto $\K _A = \L (\wt{\Phi } _A ) $ where $\wt{\Phi} _A$ is the contractive analytic function
corresponding to $\wt{\Sigma}$. Then it is clear that $W^*V_A B_1  \subset W^* \mf{Z} _{\wt{\Phi} _A} V_A = M_{\wt{\Sigma}} W^* V_A$. Let $\wt{B} _1$ be the closure of $M_{\wt{\Sigma}}$ restricted to $W^* V_A \dom{B _1}$.

Let $U$ act as multiplication by $D(t)$.  The second condition in the above theorem ensures that $U: L^2 _{\wt{\Sigma}} \rightarrow L^2 _\Sigma$ is an isometry. The third condition in the above theorem ensures that this isometry $U: L^2 _{\wt{\Sigma}} \rightarrow L^2 _{\Sigma}$ maps $\dom{\wt{B} _1}$ into $\dom{M _{\Sigma}}$ and since $U$ acts as multiplication by $D(t)$, $U \wt{B} _1 \subset M_\Sigma U$. In conclusion,
$B_1 \simeq \wt{B} _1 \lessim M _\Sigma \simeq B_2$, so that $B_1 \lessim B_2$.
This proves the sufficiency of the above three conditions when $\Theta _1$ is inner.
\end{proof}

\begin{remark}
\label{dc2}
The technical assumption on the characteristic function $\Theta _2$ from Remark \ref{densecon} can be easily removed to obtain a fully general result:

Consider the Herglotz integral representation of the Herglotz function $G _{\Theta _2}$:
$$ \re{G_{\Theta _2 } (z) } = P \im{z} + \intfty \re{ \frac{1}{i\pi} \frac{1}{t-z} } \Sigma _2 (dt ). $$ By Theorem \ref{alexclark}, we see that there is a canonical self-adjoint extension
$\mf{Z} _{\Theta _2} (\bm{1} ) $ of $\mf{Z} _\Theta$ such that $\Phi [ \mf{Z} _{\Theta _2} (\bm{1} ) ; \mf{Z} _{\Theta _2} ] = \Theta _2$, and it follows that $P = \chi _{\{1\} } \left( b (\mf{Z} _{\Theta _2} (\bm{1} ) ) \right)$. In particular if $1$ is not an eigenvalue of the Cayley transform $U =  b (\mf{Z} _{\Theta _2} (\bm{1} ))$ of $ b (\mf{Z} _{\Theta _2} (\bm{1} ))$, then it follows from Section \ref{Herglotz} that the deBranges Cauchy transform isometry $W _{\Theta _2} : L^2 _{\Sigma _2 } \rightarrow \L (\Theta _2 )$ is onto. In this case, as in \cite[Section 3.5, Section 5.4]{AMR}, one can check that $W_{\Theta _2}$ implements a unitary equivalence between $\mf{Z} _{\Theta _2}$ and $M _{\Sigma _2}$, the symmetric operator of multiplication by $t$ in $L^2 _{\Sigma _2}$ on the domain
$$ \dom{M _{\Sigma _2} } = \{ f \in L^2 _{\Sigma _2} | \ tf \in L^2 _{\Sigma _2} ; \ \intfty \Sigma _2 (dt) f(t) = 0 \}, $$ and moreover that $W_{\Theta _2 } $ implements a unitary equivalence between $\mf{Z} _{\Theta _2} (\bm{1})$, and $M^{\Sigma _2}$, the self-adjoint operator of multiplication by $t$ in $L^2 _{\Sigma _2}$.

By \cite[Proposition 5.2.2]{Martin-uni}, it follows that the canonical unitary extension $b(B(U))$ for $U \in \mc{U} (n)$ has $1$ as an eigenvalue if and only if
$$ \ker{\lim _{z \rightarrow 1} \left( (\Theta _B \circ b^{-1}) (z) ^* - U ^* \right) } \neq \{ 0 \}, $$ where $z \in \D $ approaches $1$ non-tangentially.

Note that in particular
if $B_2$ is densely defined, then every canonical self-adjoint extension of $B_2$ is densely defined, and this happens if and only if no unitary extension of $b(B_2)$ has $1$ as an eigenvalue, so that in this case
$P =0$, and $W_{\Theta _2}$ is onto. More generally the Livsic characteristic function $\Theta _2$ of $B_2$ is really only defined up to conjugation by fixed unitary matrices. It follows that we can always fix a choice of $\Theta _2$ so that $\mf{Z} _{\Theta _2} (\bm{1} )$ does not have $1$ as an eigenvalue, so that $W _{\Theta _2} : L^2 _{\Sigma _2} \rightarrow \L (\Theta _2)$ is an onto isometry, and we can assume without loss of generality
that $B = M _{\Sigma _2}$. That is we fix a choice of $\Theta _2$ so that
\be \ker{\lim _{z\rightarrow 1} \left( (\Theta _2 \circ b^{-1})  (z) ^* - U ^* \right)} = \{ 0 \}.  \ee

Alternatively, and perhaps more satisfactorily, it should be possible to remove the technical assumption from Remark \ref{densecon} completely by re-expressing the conditions of the above Theorem in terms of spaces of square integrable functions on the unit circle, and the reproducing kernel Hilbert space on $\C \sm \T$ obtained by taking the Cauchy transforms of such spaces.  However as we have preferred to express our results in terms of $L^2$ spaces on the real line and Herglotz spaces on $\C \sm \R$, we will not develop the necessary machinery to pursue this here.
\end{remark}

\begin{eg}
\label{fdeg2}
    This example is a continuation of Example \ref{fdeg}. Recall that we defined $$ V := \left( \begin{array}{cc} 0 & 0 \\ 1 & 0 \end{array} \right),$$  $V  \in \ms{V} _1  (\C ^2 )$.

Now let $$ W := \left( \begin{array}{ccc} 0 & 3/5 & 4/5 \\ 1 & 0 & 0 \\ 0 & 0& 0  \end{array} \right).$$ It is straightforward to check that
$W \in \ms{V} _1 (\C ^3)$, and that $W | _{\ker{V} ^\perp} = V | _{\ker{V} ^\perp}$.
Also note that in Example \ref{fdeg} we defined $$ U := \left( \begin{array}{ccc} 0 & 3/5 & 4/5 \\ 1 & 0 & 0 \\ 0 & 4/5 & -3/5  \end{array} \right).$$ This is a unitary matrix, and
moreover $U | _{\ker{W} ^\perp} = W | _{\ker{W} ^\perp}$, so that $U$ is a unitary extension of $W$.  Note however, that as shown in Example \ref{fdeg} that $1$ is an eigenvalue of this choice of $U$.
Hence in order to apply Theorem \ref{pochar} we will instead work with a different canonical unitary extension of $W$.
Let  $$ X := \left( \begin{array}{ccc} 0 & 3/5 & 4/5 \\ 1 & 0 & 0 \\ 0 & -i4/5 & i3/5  \end{array} \right).$$ Then $X$ is a canonical unitary extension of $W$ and is hence also a unitary extension of $V \subseteq W$.
As before let $B := b^{-1} (V), T:= b^{-1} (W)$ and $A:= b^{-1} (X)$ so that $B \subset T \subset A$. Then $A \in \Ext{T}$ is a canonical self-adjoint extension of $T$, and $A \in \Ext{B}$ is a non-canonical extension of $B$
and $B \lessim T$. Let $\Theta _B $ and $\Theta _T$ be the characterisitic functions of $B , T$.

Our goal in this example is to verify that the three conditions of Theorem \ref{pochar} are satisfied. Recall that
$$ \Theta _B (z) = \left( \frac{z-i}{z+i} \right) ^2, $$ and also recall that by Theorem \ref{alexclark}, that since $X$ is a canonical unitary extension of $W$, that up to a unimodular constant,
$$   \Theta _T (z) = \Phi [ A = b^{-1} (X) ; T ] (z).$$ Since $\Theta _T$ is only defined up to unimodular constants, we can and do fix $\Theta _T = \Phi [A ; T]$. To show that the first condition of Theorem \ref{pochar} is satisfied, we need to calculate $\Phi [A;B]$,
and to verify that it is greater or equal to $\Theta _B$. Recall that we did a similar calculation for the unitary matrix $U$ which is a different unitary extension of $W$ in Example \ref{fdeg}.

We begin by calculating the Herglotz measure of $\Phi [A ; B]$. Use that $\ran{V} ^\perp = \ker{B^* +i }$ is spanned by $e_1$, so that if $\sigma _V$ is the Herglotz measure of $\phi [X ; V ] = \Phi [A ; B] \circ b^{-1}$,
that $$ \sigma _V (\Om ) =  \left( e_1 , P _X (\Om ) e_1 \right).$$ Again this is a probability measure and $\left| ( e_1 , \hat{b} _2 ) \right| = \left| ( e_1 , \hat{b} _3 ) \right| =: a$ so that
$$ 1 = \sum _{k=1} ^3 \left| ( e_1 , \hat{b} _k ) \right| ^2 = \frac{1}{6} + 2a, $$ $a = \frac{5}{12}$ and
$$ \sigma _V = \frac{1}{6} \delta _i + \frac{5}{12} \delta _\la + \frac{5}{12} \delta _{-\ov{\la}}. $$ Finally as before
$$ \Sigma _B = \pi  \frac{1}{3} \delta _{-1} + \pi (1 + \beta ^2 ) \frac{5}{12} \delta _\beta + \pi (1 + \beta ^{-2} ) \frac{5}{12} \delta _{\beta ^{-1}}. $$
As in Example \ref{fdeg}, if $\Phi _A := \Phi [A;B]$ then
$$G_{\Phi _A }(z)  = i \sigma _X (\{1 \} ) z + \intfty \frac{zt+1}{i(t-z)} \wt{\sigma} _X (dt), $$ where $\wt{\sigma} _X := \sigma _X \circ b$. Since $1$ is not an eigenvalue of $X$, this becomes
$$ G_{\Phi _A} (z) = -i \frac{1}{6} \frac{z-1}{z+1} -i \frac{5}{12} \frac{z\beta +1}{\beta -z}-i\frac{5}{12} \frac{z+\beta}{1-\beta z}. $$ Using that $\Phi _A = \frac{G_{\Phi _A } +1}{G_{\Phi _A } -1}$, and simplifying as in Example \ref{fdeg} shows that $\Phi _A$ is the product
of three Blaschke factors with zeroes at the roots of the polynomial:
$$ p(z ) := 2 (z-1)(\beta -z)(1-\beta z) +5 (z+1)(z\beta +1)(1-\beta z) +5 (z+1)(\beta -z)(1-\beta z) -12i (z+1)(\beta -z )(1-\beta z).$$ It is a bit more tedious to calculate the roots of this polynomial this time.
However it is not hard to check that $p(i) =0$, and one can verify that $p$ has a double root at $z=i$ and that the third root of $p$ is located at the point $\mu = \frac{i-4}{i+4} \in \C _+$.
It follows that up to a unimodular constant, $$ \Phi _A (z) = \left( \frac{z-i}{z+i} \right) ^2 \frac{z-\mu}{z-\ov{\mu}},$$ which is indeed greater or equal to $\Theta _B$.

We now show that the Herglotz measure of $\Phi [A ; B ]  (z)$, is absolutely continuous with respect to the Herglotz measure of $\Theta _T  = \Phi [A ; T ]$ so that
the second condition of Theorem \ref{pochar}  is also satisfied:

Let us calculate the Herglotz measure $\Sigma _T$ of $\Theta _T =\Phi [A ; T]$. Now $A = b^{-1} (X)$, and we calculate $\sigma _X$, the Herglotz measure of $\theta _X := \Theta _T \circ b^{-1}$,
as in Example \ref{fdeg} by calculating the spectral measure of the unitary matrix $X$. The determinant of $(z - X)$ can be calculated to be
$$ \mbox{det} (z -X ) = (z-i)(z-\la )(z +\ov{\la} )  =: p(z), \quad \la := \frac{2}{5} \sqrt{6} -i \frac{1}{5}. $$ The eigenvectors $\vec{b} _k$, $1\leq k \leq 3$ of $X$ to the
eigenvalues $\la _1 = i $, $\la _2 = \la$ and $\la _3 = - \ov{\la}$ are given by
$$ \vec{b} _k := ( 1, \ov{\la} _k , \frac{5}{4} \la _k - \frac{3}{4} \ov{\la _k} ) ^T. $$ If $\hat{b} _k := \frac{\vec{b} _k }{\| \vec{b} _k \| },$ then one can check that
$$ \hat{b} _1 = \frac{1}{\sqrt{6}} (1, -i , 2i ). $$ The spectral measure of $X$ is then
$$ P _X := \sum _{k=1} ^3 \left( \cdot , \hat{b} _k \right) \hat{b} _k \delta _{\la _k}, $$ and since $v = e_3$ spans $\ran{W} ^\perp = \ker{T^* +i}$,
$$ \sigma _X (\Om) = \left( e_3 , P_X (\Om ) e_3 \right). $$ Since $P_X$ is unital, this means that $\sigma _X$ is a probability measure so that
$$ 1 = \sum _{k=1} ^3 \left| ( e_3 , \hat{b} _k ) \right| ^2 = \frac{2}{3} + \left| (e_3 , \hat{b} _2 ) \right| ^2 +   \left| (e_3 , \hat{b} _3 ) \right| ^2.$$
Using that $\la _2 = \la = -\ov{\la} _3$, we get that $\| \vec{b} _2 \| = \| \vec{b} _3 \|$ and that $\left| ( e_3 , \hat{b} _2 ) \right| = \left| ( e_3 , \hat{b} _3 ) \right| =: a$ so that
$ 1 = \frac{2}{3} + 2a$ and $a = \frac{1}{6}$. In conclusion,
$$ \sigma _X = \frac{2}{3} \delta _i + \frac{1}{6} \delta _\la + \frac{1}{6} \delta _{-\ov{\la}}.$$ Now we use the fact that
$$ \Sigma _T  (\Om ) = \int _\Om \pi (1 +t^2 ) (\sigma _X \circ b) (dt), $$ to calculate that
$$\Sigma _T  = \pi \frac{4}{3} \delta _{-1} + \pi (1 + \beta ^2 ) \frac{1}{6} \delta _\beta + \pi (1 + \beta ^{-2} ) \frac{1}{6} \delta _{\beta ^{-1}}, $$
where $\beta := b^{-1} (\la )$.  It follows that the Herglotz measure $\Sigma _B$ of $\Phi [A; B]$ is indeed absolutely continuous with respect to the Herglotz measure $\Sigma _T$ of $\Theta _T = \Phi [A; T]$.

Note that one can calculate that up to a unimodular constant $$\Theta _T (z) = \frac{(z-i)(z-\mu _1 )(z - \mu _2)}{(z+i)(z-\ov{\mu _1} )(z - \ov{\mu _2})}, $$
where $$ \mu _1 := i (4 + \sqrt{15}), \quad \mbox{and} \quad \mu _2 = i (4 -\sqrt{15} ),$$ so that $\Theta _B$ is not a divisor of $\Theta _T$.

Finally we verify that the third condition of Theorem \ref{pochar} is satisfied. First we need to calculate the domain of $B = b^{-1} (V)$. We have that
$\ker{V} ^\perp$ is spanned by $e _1$, and $\dom{B} = (1-V) \ker{V} ^\perp$ so that $\dom{B}$ is spanned by the vector $(1 , -1)$ (or if we view $\C ^2$ as a subspace of $\C ^3$ and $B \subset T$
then this is the vector $(1, -1, 0 )$).

Let $\wt{\Sigma } := \pi \Sigma _B$, and let $\Sigma := \pi \Sigma _T$. Let $\wt{W} : L^2 _{\wt{\Sigma }} \rightarrow \L ( \wt{\Phi} [ A ; B ] ) $, and $W : L^2 _\Sigma \rightarrow \L (\wt{\Phi} [A ; T ])$ be the
corresponding deBranges isometries onto the Herglotz spaces. Also let $V _A : \C ^2 \rightarrow \L (\wt{\Phi} [A ; B ] ) =: \wt{\K}  _A $, where $\wt{\K} _A$ is the model reproducing kernel Hilbert space
defined using the extension $A \in \Ext{B}$ and $\Om _A (z) := (A+i) (A - \ov{z} ) ^{-1}  \wt{J}$ and $\wt{J} : \C \rightarrow \ker{B^* +i}$ is defined by $\wt{J} e_1 = e_1$ (here $e_1$ is a normalized basis vector for $\C$).

We need to calculate the image of $(1, -1)$ under the map $\wt{W} ^* V _A$ which takes $\dom{B}$ into $\dom{M _{\wt{\Sigma}}}$:

\ba (V_A (1, -1) ) (z) &  = & \Om _A (z) ^* (1 , -1) ^T  \nonumber \\
& =& \left( (A-i)(A-z) ^{-1} e_1 , e _1 \right) - \left(  (A-i) (A-z) ^{-1} e_2 , e _1 \right). \nonumber \ea  Using that $e_2 = X e_1$ where $X = b(A)$, we get this is
\ba (V_A (1, -1) ) (z) &  = & \left( (A-i)(A-z) ^{-1} e_1 , e _1 \right) - \left(  (A-i) (A-z) ^{-1} (A-i) (A+i) ^{-1} e_1 , e _1 \right) \nonumber \\
& =& \intfty \frac{t-i}{t-z} (1 -b(t)) \left( P_A (dt) e_1 , e_1 \right) \nonumber \\
& =& \frac{1}{i\pi} \intfty \frac{1}{t-z} \left( \frac{i}{\pi} \frac{1}{t+i} (1-b(t)) \right) \wt{\Sigma } (dt) \nonumber \\
& =& \wt{W} f (z), \nonumber \ea
where $f \in L^2 _{\wt{\Sigma }} $ is
$$ f (t) = (1 -b(t) ) \frac{i}{\pi} \frac{1}{t+i}.$$
We can now verify that $f \in \dom{M _{\wt{\Sigma}}}$ by checking that
$\intfty \wt{\Sigma} (dt) f(t) = 0$. This integral is equal to
$$  \intfty  \wt{\Sigma} (dt) f(t) = \frac{1}{3} \frac{1 - b(-1)}{-1 +i} + \frac{5}{12} (1 +\beta ^2 ) \frac{1 -b (\beta ) }{\beta +i} + \frac{5}{12} ( 1 + \beta ^{-2} ) \frac{1 - b (\beta ^{-1} )}{\beta ^{-1} +i }.$$
Now using that $b (-1) = i $ and $b (\beta ) = \la = -\frac{i}{5} + \frac{2}{5} \sqrt{6}$, this can be simplified to yield
$$ \intfty  \wt{\Sigma} (dt) f(t) = \frac{-1}{3} + \frac{5}{12} (2i \la  - 2i \ov{\la}) = 0, $$ so that indeed $\wt{W} ^* V_A (1, -1 ) ^T \in \dom{M _\Sigma}$.

To verify the final condition of Theorem \ref{pochar}, we need to show that if $U _{\Sigma } : L^2 _{\wt{\Sigma}} \rightarrow L^2 _\Sigma$ is the isometry which acts as multiplication by $D(t)$
where $$ \wt{\Sigma} (dt) = \ov{D(t)} \Sigma (dt) D(t), $$  then
$U_\Sigma f \in \dom{M _{\Sigma}}$. First we calculate $D(t)$ and $U _\Sigma$. We have by construction that
$$ \wt{\Sigma} (dt) = \pi ^2 (1 +t^2 ) \left( P_A (dt) e_1 , e_1 \right), $$ and now observe that $X e _1 = e_2$ and that
$ X e_2 = 3/5 e_1 -i 4/5 e_3 = 3/5 X^* e_2 -i4/5 e_3. $ Rearranging this yields $ e_2 = -i \frac{4}{5} (X^2 -3/5  ) ^{-1} X e_3$ so that
$e_1 = - i \frac{4}{5} ( X ^2 - \frac{3}{5} ) ^{-1} e_3$, where recall that $X = b(A)$. It follows that
\ba \left( P_A (dt) e_1 , e_1 \right) & = & \left(  i \frac{4}{5} (b(A) ^{-2} - 3/5) ^{-1} P_A (dt) \frac{4}{5i} (b(A) ^2 -3/5) ^{-1} e_3 , e_3 \right)  \nonumber \\
& =&  \ov{D(t) } \left( P_A (dt) e_3 , e_3 \right) D(t), \nonumber \ea
with $$ D(t) = -i \frac{4}{5} \frac{1}{b(t) ^2 - 3/5}. $$ Hence to complete the verification of the third condition of Theorem \ref{pochar}, we simply need show that
if
$$ g(t) := D(t) f(t) = \frac{4}{5 \pi} \frac{1}{b(t) ^2 - 3/5} \frac{1}{t+i} (1 - b(t) ) \in L^2 _\Sigma, $$ that
$$ \intfty \Sigma (dt) g(t) = 0. $$
Here is the calculation:
\ba \intfty \Sigma (dt) g(t) & =& \frac{4}{3} \frac{1}{i -3/5} \frac{1}{-1 +i} (i -1) + \frac{1}{6} (1 + b^{-1} (\la ) ^2 ) \frac{1}{\la - 3/5} \frac{1}{\la +i} (\la -1 ) \nonumber \\
& & + \frac{1}{6} (1 + \beta ^{-2} ) \frac{1}{\ov{\la} ^2 -3/5} \frac{1}{b^{-1} (-\ov{\la}) +i } ( - \ov{\la} -1) \nonumber \\
& =& \frac{-5}{6} + \frac{1}{6} \frac{-2i \la }{\la ^2 -3/5} + \frac{1}{6} \frac{2i \ov{\la}}{\ov{\la} ^2 -3/5} \nonumber \\
& =& \frac{-5}{6}  + \frac{i}{3} \left(  \frac{\ov{\la}}{\ov{\la} ^2 -3/5} - \frac{\la}{\la ^2 -3/5} \right) \nonumber \\
& =& -\frac{5}{6} + \frac{i}{3} \frac{8}{5} \frac{\la - \ov{\la}}{|\la ^2 -\frac{3}{5} | ^2} \nonumber \\
& =&  0. \nonumber \ea

In summary we have shown that if $f = \wt{W} ^* V_A (-e_1 +e _2)$, where $\dom{B}$ is spanned by $-e_1 + e_2$, that both
$$ \intfty \wt{\Sigma} (dt) f (t) = 0 \quad \quad \mbox{and} \quad \quad \intfty \Sigma (dt) D(t) f(t) = 0, $$ so that
the third and final condition of Theorem \ref{pochar} is satisfied.
\end{eg}

\section{Outlook}

There are several directions in which the results of this paper can be extended.

We have assumed throughout that $B \in \symm$ has an inner Livsic characteristic function.  A good portion of the theory we have developed here does not depend on this fact,
and it would be good to generalize the results contained here to the case where the Livsic function is an arbitrary contractive analytic function (vanishing at $z=i$).  We have done some
work on this already, in particular Example \ref{motive} can be generalized to show that if $\Theta \leq \Phi$ are arbitrary contractive analytic functions that there is a bounded
multiplier $V : \L (\Theta ) \rightarrow \L (\Phi )$ which intertwines $\mf{Z} _\Theta$ and $\mf{Z} _\Phi$. However it
is not clear whether $\mf{Z} _\Theta \lessim \mf{Z} _\Phi$ in this general case, or whether more general definitions of partial order, and extensions of a symmetric linear transformation are needed.
Also if $A \in \Ext{B}$ where $\Theta _B$ is not inner, then one can show that in general $\H _A $ is only boundedly contained in $\K _A$, and so is not just a Hilbert subspace. Once these results are successfully
generalized to arbitrary simple symmetric and isometric linear transformations with equal indices, a natural question is whether our partial order results can be extended to arbitrary contractions. Namely
given contractions $T_1, T_2$, perhaps one could define that $T_1 \lessim T_2$ if $T_1 \simeq T_1 ' \subseteq T_2$. Perhaps this could be accomplished by using the fact that the problem of unitary equivalence
of contractions is equivalent to the problem of unitary equivalence of partial isometries, see \cite[Theorem 1]{Halmos} and the discussion following it.

There should be several interesting consequences of the results already obtained in this paper.  For example as discussed in Remark \ref{Acyclic}, we can use the theory developed here to provide an alternate proof
of the Alexandrov isometric measure theorem, \cite[Theorem 2]{Alexiso}. In fact the result we obtain is a generalization of the operator theoretic result of Krein \cite[Chapter 1, Corollary 2.1]{Krein} which uses
the theory of entire symmetric operators and hence holds for the case where $\Theta _B$ is a meromorphic scalar-valued inner function.  We point out that this result of Krein can be used to prove the Alexandrov isometric measure theorem, and that de Branges has also proven this result in the case where $\Theta $ is meromorphic in his book \cite[Theorem 32]{dB}. Our generalization holds for arbitrary inner functions, and it should be
possible to extend this to vector-valued Hardy spaces and matrix-valued inner functions as well. Our theory should also allow us to extend the main result of \cite{Martin-near} to the
case of arbitrary inner functions and nearly invariant subspaces, as well as to vector-valued versions of nearly invariant subspaces.

Finally as discussed in Remark \ref{WeylTit}, there is a natural bijection between the sets $\Ext{B}$ and $\mr{POVM} (B)$, the set of all unital positive operator valued measures which diagonalize $B$.  It is easy to see
with an application of Naimark's dilation theorem that $\mr{POVM} (B)$ is a convex set, and we think it could be interesting to study the properties of this convex set, for example to determine its extreme points, and
to study its Choquet theory. It is known that $\mr{POVM} (B)$ is a face in the set of all unital positive-operator valued measures on $\R$ \cite[Theorem 13.6.3]{Robthesis}, and consequently that every projection valued
measure corresponding to a canonical $A \in \Ext{B}$ is an extreme point of this set (although this can be proven directly). Naimark has proven that if $B \in \sym{n}{\H}$ and $A \in \Ext{B}$ is self-adjoint in $\K$ where $\K \ominus \H$ is finite dimensional, then the positive operator-valued measure corresponding to $A$ is an extreme point of $\mr{POVM} (B)$ \cite{Naimark-extreme}. Moreover Gilbert has proven that if $B \in \sym{n}{\H}$, then
the set of all $Q \in \mr{POVM}(B)$ which correspond to $A \in \Ext{B}$ defined on $\K$ with $\K \ominus \H$ finite dimensional is dense in a natural topology on $\mr{POVM} (B)$ \cite{Gilbert-extreme}.  It could be interesting to see whether the extreme points of $\mr{POVM} (B)$ can be given a function theoretic characterization in terms of the characteristic functions $\Phi [A ; B]$ of the corresponding extensions of $B$.


\begin{thebibliography}{10}

\bibitem{Glazman}
N.I. Akhiezer and I.M. Glazman.
\newblock {\em Theory of Linear Operators in Hilbert Space}.
\newblock Dover Publications, New York, NY, 1993.

\bibitem{Martin-uni}
R.T.W. Martin.
\newblock Unitary perturbations of compressed n-dimensional shifts.
\newblock {\em Op. Th. Comp. Anal.}, 7:767--799, 2013.

\bibitem{Martin-dB}
R.T.W. Martin.
\newblock Representation of symmetric operators with deficiency indices $(1,1)$
  in de \uppercase{B}ranges space.
\newblock {\em Op. Th. Comp. Anal.}, 5:545--577, 2011.

\bibitem{AMR}
W.T.~Ross A.~Aleman, R.T.W.~Martin.
\newblock On a theorem of {L}ivsic.
\newblock {\em J. Funct. Anal.}, 264:999--1048, 2013.

\bibitem{Krein}
M.L. Gorbachuk and V.I. Gorbachuk, editors.
\newblock {\em M.G. Krein's Lectures on Entire Operators}.
\newblock Birkhauser, Boston, 1997.

\bibitem{Halmos}
P.R. Halmos and J.E. Mc{L}aughlin.
\newblock Partial isometries.
\newblock {\em Pacific J. Math.}, 13:361--371, 1963.

\bibitem{NF}
B.~Sz.-Nagy and C.~Foia\c{s}.
\newblock {\em Harmonic analysis of operators on \uppercase{H}ilbert space}.
\newblock American Elsevier publishing company, Inc., New York, N.Y., 1970.

\bibitem{Livsic2}
M.S. Livsic.
\newblock Isometric operators with equal deficiency indices.
\newblock {\em AMS trans.}, 13:85--103, 1960.

\bibitem{Livsic}
M.S. Livsic.
\newblock A class of linear operators in \uppercase{H}ilbert space.
\newblock {\em AMS trans.}, 13:61--83, 1960.

\bibitem{Silva-entire}
L.~O. Silva and J.~H. Toloza.
\newblock On the spectral characterization of entire operators with deficiency
  indices (1,1).
\newblock {\em J. Math. Anal. Appl.}, 367:360--373, 2010.

\bibitem{Habock}
U.~Habock.
\newblock {\em Reproducing kernel spaces of entire functions}.
\newblock Diploma Thesis, Technishcen Universitat Wien, 2001.

\bibitem{dB}
L.~de Branges.
\newblock {\em Hilbert spaces of entire functions}.
\newblock Prentice-Hall, Englewood Cliffs, NJ, 1968.

\bibitem{dB-Herglotz}
L.~de~Branges.
\newblock Perturbations of self-adjoint transformations.
\newblock {\em Amer. J. Math.}, 84:543--560, 1962.

\bibitem{Sarason-dB}
D.~Sarason.
\newblock {\em Sub-Hardy Hilbert spaces in the unit disk}.
\newblock John Wiley \& Sons Inc., New York, NY, 1994.

\bibitem{Martin-semi}
R.T.W. Martin.
\newblock Semigroups of partial isometries and symmetric operators.
\newblock {\em Integral Equations Operator Theory}, 70:205--226, 2011.

\bibitem{Paulsen}
V.~Paulsen.
\newblock {\em Completely Bounded Maps and Operator Algebras}.
\newblock Cambridge University Press, New York, NY, 2002.

\bibitem{Alexiso}
A.B. Aleksandrov.
\newblock Isometric embeddings of coinvariant subspaces of the shift operator.
\newblock {\em J. Math. Sci.}, 92:3543--3549, 1998.

\bibitem{Paulsen-rkhs}
V.~Paulsen.
\newblock {\em An Introduction to the theory of reproducing kernel {H}ilbert
  spaces}.
\newblock www.math.uh.edu/{~}vern/rkhs.pdf, 2009.

\bibitem{Don}
W.F. Donoghue.
\newblock On perturbation of spectra.
\newblock {\em Commun. Pure and Appl. Math.}, 18:559--579, 1965.

\bibitem{MT1}
K.A. Makarov and E.~Tsekanovskii.
\newblock On the {W}eyl-{T}itchmarsh and {L}ivsic functions.
\newblock {\em Proc. Sympos. Pure Math.}, 87:291--313, 2013.

\bibitem{MT2}
K.A. Makarov and E.~Tsekanovskii.
\newblock On the addition and multiplication theorems.
\newblock arxiv:1210.8504v1:1--23, 2013.

\bibitem{MR-isomult}
R.T.W. Martin and W.T. Ross.
\newblock A partial order on partial isometries.
\newblock {\em In preparation.}, 2013.

\bibitem{Martin-near}
R.T.W. Martin.
\newblock Near invariance and symmetric operators.
\newblock {\em Accepted by Oper. Matrices}, 2013.

\bibitem{Robthesis}
R.T.W. Martin.
\newblock {\em Bandlimited functions, curved manifolds and self-adjoint
  extensions of symmetric operators}.
\newblock University of Waterloo, 2008.

\bibitem{Naimark-extreme}
M.A. Naimark.
\newblock Extremal spectral functions of a symmetric operator.
\newblock {\em Izvest. Akad. Nauk SSSR, Ser. Mat.}, 1942.

\bibitem{Gilbert-extreme}
R.C. Gilbert.
\newblock Extremal spectral functions of a symmetric operator.
\newblock {\em Pacific J. Math.}, 14:75--84, 1964.

\end{thebibliography}
\end{document}